\pdfoutput=1
\documentclass[hidelinks,final,US,11pt,reqno,bibtex,nopagebackref,notcite,notref]{amsart}
\usepackage{mathieu}

\DeclareMathOperator{\ad}{ad}

\DeclareMathOperator{\Hom}{Hom}
\DeclareMathOperator{\End}{End}
\DeclareMathOperator{\Ext}{Ext}
\DeclareMathOperator{\Der}{Der}
\DeclareMathOperator{\divergence}{div}
\DeclareMathOperator{\tot}{tot}
\DeclareMathOperator{\id}{id}
\DeclareMathOperator{\rk}{rk}

\DeclareMathOperator{\pbw}{pbw}
\DeclareMathOperator{\hkr}{hkr}
\DeclareMathOperator{\Bott}{Bott}
\DeclareMathOperator{\dR}{dR}
\DeclareMathOperator{\can}{can}
\DeclareMathOperator{\Todd}{Td} % Todd class
\DeclareMathOperator{\todd}{td} % Todd cocycle
\DeclareMathOperator{\sgn}{sgn}

\DeclareMathOperator{\CE}{CE}
\DeclareMathOperator{\alt}{alt}
\DeclareMathOperator{\poly}{poly}
\DeclareMathOperator{\formal}{formal}
\DeclareMathOperator{\sinverse}{Out}
\DeclareMathOperator{\tinverse}{In}
\DeclareMathOperator{\trace}{tr}
\DeclareMathOperator{\supertrace}{str}
\DeclareMathOperator{\Ber}{Ber}
\DeclareMathOperator{\At}{At}
\DeclareMathOperator{\trivial}{trivial}
\DeclareMathOperator{\im}{Im}
\DeclareMathOperator{\conf}{Conf}
\DeclareMathOperator{\hypercohomology}{\mathbb{H}}
\DeclareMathOperator{\sheaf}{sheaf}

\newcommand{\NN}{\mathbb{N}} 
\newcommand{\NO}{\mathbb{N}_0} 
\newcommand{\ZZ}{\mathbb{Z}} 
 
\newcommand{\RR}{\mathbb{R}} 
\newcommand{\CC}{\mathbb{C}} 
\newcommand{\KK}{\Bbbk} 
\newcommand{\HH}{\mathbb{H}}
\newcommand{\RP}{\mathbb{RP}}
\newcommand{\abs}[1]{\left| #1 \right|} 
\newcommand{\into}{\hookrightarrow}
\newcommand{\onto}{\twoheadrightarrow}
\newcommand{\xto}[1]{\xrightarrow{#1}}

\newcommand{\inv}{^{-1}}
\newcommand{\dual}{^{\vee}}
\newcommand{\transpose}{^{\top}} 
\newcommand{\argument}{-}
\newcommand{\shuffle}[2]{\mathfrak{S}_{#1}^{#2}}
\newcommand{\sections}[1]{\Gamma(#1)}
\newcommand{\enveloping}[1]{\mathcal{U}(#1)}
\newcommand{\XX}{\mathfrak{X}}
\newcommand{\OO}{\Omega}

\newcommand{\duality}[2]{\left\langle#1\middle|#2\right\rangle}
\newcommand{\lie}[2]{[#1,#2]} 
\newcommand{\schouten}[2]{\left[#1,#2\right]} 
\newcommand{\gerstenhaber}[2]{\left\llbracket #1,#2\right\rrbracket} 

\newcommand{\cM}{\mathcal{M}} % graded manifold
\newcommand{\cF}{\mathcal{F}} % Fedosov dg Lie algebroid
\newcommand{\torsion}{T^\nabla}

\newcommand{\fedosov}{Q}
\newcommand{\fedosova}{\schouten{Q}{\argument}}

\newcommand{\frakg}{\mathfrak{g}}
\newcommand{\frakh}{\mathfrak{h}}
\newcommand{\frakm}{\mathfrak{m}}
\newcommand{\kont}{\mathscr{U}}
\newcommand{\graphs}{\mathscr{G}}
\newcommand{\outedge}[1]{\sinverse(#1)}
\newcommand{\inedge}[1]{\tinverse(#1)}
\newcommand{\dabott}{d_A^{\Bott}}
\newcommand{\dau}{d_A^{\mathcal{U}}}
\newcommand{\dace}{d_A}
\newcommand{\liederivative}[1]{\mathcal{L}_{#1}}
\newcommand{\hochschild}{d_{\mathscr{H}}}
\newcommand{\atiyahcocycle}{R^\nabla_{1,1}}
\newcommand{\atiyahclass}{\alpha}
\newcommand{\tTodd}{\widetilde{\Todd}} % tilde Todd class
\newcommand{\ttodd}{\widetilde{\todd}} % tilde Todd cocycle
\newcommand{\homogeneoustriple}[3]{\mathopen{[} #1 \mathpunct{:} #2 \mathpunct{:} #3 \mathclose{]}}
\newcommand{\xnabla}{X^\nabla}
\newcommand{\xtwo}{X_2}
\newcommand{\xhigher}{X_{\geqslant 3}}

\newcommand{\cI}{\mathscr{I}}
\newcommand{\cA}{\mathcal{A}}
\newcommand{\cL}{\mathcal{L}}
\newcommand{\cB}{\mathcal{B}}
\newcommand{\cO}{\mathcal{O}}
\newcommand{\cV}{\mathcal{V}}
\newcommand{\cE}{\mathcal{E}}
\newcommand{\EE}{\mathcal{E}}
\newcommand{\Tau}{\etendu{\perturbed{T}}}
\newcommand{\nablacan}{\nabla^{\can}}
\newcommand{\hhb}[1]{\hat{\partial}_{#1}}
\newcommand{\half}{\frac{1}{2}}
\newcommand{\vf}{V_{\formal}}
\newcommand{\kkf}[2]{\KK_{\formal}^{#1|#2}}
\newcommand{\kf}[1]{\KK_{\formal}^{#1}}

\newcommand{\UcL}[1]{{\mathcal{U}(\cL)}^{\otimes {#1}+1}}

\newcommand{\salgebra}{\mathscr{C}}
\newcommand{\etendu}[1]{\mathop{#1_{\natural}}}
\newcommand{\bsigma}{\etendu{\sigma}}
\newcommand{\Phii}{\mathcal{I}}
\newcommand{\Phia}{\Psi}
\newcommand{\perturbed}[1]{\breve{#1}}
\newcommand{\Tpoly}[1]{\mathcal{T}_{\poly}^{#1}}
\newcommand{\Dpoly}[1]{\mathcal{D}_{\poly}^{#1}}
\newcommand{\verticalTpoly}[1]{\mathscr{T}_{\poly}^{#1}}
\newcommand{\verticalDpoly}[1]{\mathscr{D}_{\poly}^{#1}}
\newcommand{\nTpoly}[1]{\Tpoly{#1}(N_\mathcal{F})}
\newcommand{\nDpoly}[1]{\Dpoly{#1}(N_\mathcal{F})}
\newcommand{\dHH}{\mathfrak{d}_{\mathscr{H}}}

\title{Formality and Kontsevich--Duflo type theorems for Lie pairs}
\dedicatory{Dedicated to the memory of our friend and colleague Krzysztof Wysocki}
\thanks{Research partially supported by NSF grants DMS-1406668 and DMS-1101827, and NSA grant H98230-14-1-0153.}

\author{Hsuan-Yi Liao}
\address{School of Mathematics, Korea Institute for Advanced Study}
\email{hsuanyiliao@kias.re.kr}

\author{Mathieu Stiénon}
\address{Department of Mathematics, Pennsylvania State University}
\email{stienon@psu.edu} 

\author{Ping Xu}
\address{Department of Mathematics, Pennsylvania State University}
\email{ping@math.psu.edu}

\begin{document}

\begin{abstract}
Kontsevich's formality theorem states that
there exists an $L_\infty$ quasi-isomorphism
 from the dgla $T_{\poly}^\bullet (M)$
of polyvector fields on a smooth manifold
$M$ to the dgla $D^\bullet_{\poly}(M)$ of
polydifferential operators on $M$, which extends the classical
Hochschild--Kostant--Rosenberg map. In this paper, we 
extend Kontsevich's formality theorem to \emph{Lie pairs},
a framework which includes a range of diverse geometric contexts
such as complex manifolds, foliations, and $\frakg$-manifolds 
(that is, manifolds endowed with an action of a Lie algebra $\frakg$).
The spaces $\tot\big(\sections{\Lambda^\bullet A^\vee}\otimes_R\Tpoly{\bullet}\big)$
and $\tot\big(\sections{\Lambda^\bullet A^\vee}\otimes_R \Dpoly{\bullet}\big)$ 
associated with a Lie pair $(L,A)$ each carry an $L_\infty$ algebra structure 
canonical up to $L_\infty$ isomorphism.
These two spaces serve as replacements for the spaces of polyvector fields and polydifferential operators, respectively.
Their corresponding cohomology groups $\hypercohomology^\bullet_{\CE}(A,\Tpoly{\bullet})$ 
and $\hypercohomology^\bullet_{\CE}(A,\Dpoly{\bullet})$ admit canonical Gerstenhaber algebra structures.
We establish the following formality theorem for Lie pairs: there exists an $L_\infty$ quasi-isomorphism
from $\tot\big(\sections{\Lambda^\bullet A^\vee}\otimes_R\Tpoly{\bullet}\big)$
to $\tot\big(\sections{\Lambda^\bullet A^\vee}\otimes_R\Dpoly{\bullet}\big)$
whose first Taylor coefficient is equal to $\hkr\circ(\todd_{L/A}^{\nabla})^{\frac{1}{2}}$.
Here $(\todd_{L/A}^{\nabla})^{\frac{1}{2}}$ acts on $\tot\big( \sections{\Lambda^\bullet A^\vee}\otimes_R \Tpoly{\bullet}\big)$ by contraction.
Furthermore, we prove a Kontsevich--Duflo type theorem for Lie pairs:
the Hochschild--Kostant--Rosenberg map twisted by the square root of the Todd class of the Lie pair $(L, A)$ 
is an isomorphism of Gerstenhaber algebras from $\hypercohomology^\bullet_{\CE}(A,\Tpoly{\bullet})$ 
to $\hypercohomology^\bullet_{\CE}(A,\Dpoly{\bullet})$. 
As applications, we establish formality theorems and Kon\-tse\-vich--Duflo type theorems 
for complex manifolds, foliations, and $\frakg$-manifolds. 
In the case of complex manifolds, we recover the Kontsevich--Duflo theorem of complex geometry.
\end{abstract}

\maketitle

\tableofcontents

\section*{Introduction}

In the late 1990's, Kontsevich revolutionized the field of deformation quantization with his formality theorem: 
there exists an $L_\infty$ quasi-isomorphism from the dgla $T_{\poly}^\bullet (M)$ of polyvector fields on a smooth manifold $M$ 
to the dgla $D^\bullet_{\poly}(M)$ of polydifferential operators on $M$ extending the classical Hochschild--Kostant--Rosenberg map. 
Indeed, the formality theorem implies the existence of deformation quantizations for every smooth Poisson manifold \cite{MR2062626, MR2699544, MR2102846,MR1944574}. 
In his paper \cite{MR2062626}, Kontsevich gave an explicit formula for the formality quasi-isomorphism in the case $M=\RR^d$ 
and then outlined how the result can be generalized to arbitrary smooth manifolds. 
Later, Dolgushev gave a detailed proof of the globalization to arbitrary smooth manifolds 
of Kontsevich's formality quasi-isomorphism for $\RR^d$ \cite{MR2102846}
based on Fedosov's patching technique \cite{MR1293654,MR1327535}. See also~\cite{MR1944574} for another proof.

In this paper, we extend Kontsevich's formality theorem to \emph{Lie pairs},
a framework which includes a wide range of diverse geometric contexts
including complex manifolds, foliations, and $\frakg$-manifolds.
By a \emph{Lie pair} $(L,A)$, we mean an inclusion $A\hookrightarrow L$ of Lie $\KK$-algebroids over a smooth manifold $M$. 
(Throughout the paper, we use the symbol $\KK$ to denote either of the fields $\RR$ and $\CC$.)
Recall that a \emph{Lie $\KK$-algebroid} is a $\KK$-vector bundle $L\to M$, whose space of sections is endowed with a Lie bracket 
$[\argument,\argument]$, together with a bundle map $\rho:L\to TM\otimes_\RR\KK$ called \emph{anchor} 
such that $\rho:\sections{L}\to\XX(M)\otimes \KK$ is a morphism of Lie algebras
and $[X,fY]=f[X,Y]+\big(\rho(X)f\big)Y$ for all $X,Y\in\Gamma(L)$ and $f\in C^\infty(M,\KK)$.
A $\KK$-vector bundle $L\to M$ is a Lie algebroid if and only if $\sections{L}$ is a \emph{Lie--Rinehart algebra}~\cite{MR0154906} 
over the commutative ring $C^\infty(M, \KK)$.
Lie pairs arise naturally in a number of subdisciplines of mathematics such as complex geometry, foliation theory, and Lie theory.
A complex manifold $X$ determines a Lie pair (over $\CC$): $L= T_X\otimes\CC$ and $A=T_X^{0,1}$.
A foliation on a smooth manifold $M$ determines a Lie pair (over $\RR$): $L=TM$ and $A$ is the integrable distribution on $M$ tangent to the foliation. 
A manifold equipped with an action of a Lie algebra $\frakg$ gives rise to a Lie pair in a natural way 
(see~\cite[Example~5.5]{MR1460632} and~\cite{MR1650045,MR3650387}). 
 
Given a Lie pair $(L,A)$, the quotient $L/A$ is naturally an $A$-module. 
When $L$ is the tangent bundle to a manifold $M$ and $A$ is an integrable distribution on $M$, 
the $A$-action on $L/A$ is given by the Bott connection \cite{MR0362335}. 
The spaces $\tot\big(\sections{\Lambda^\bullet A^\vee}\otimes_R\Tpoly{\bullet}\big)$
and $\tot\big(\sections{\Lambda^\bullet A^\vee}\otimes_R \Dpoly{\bullet}\big)$ 
associated with a Lie pair $(L,A)$ serve as replacements for the spaces of polyvector fields and polydifferential operators respectively.
Each one of them carries an $L_\infty$ algebra structure canonical up to $L_\infty$ isomorphism 
and their respective cohomology groups $\hypercohomology^\bullet_{\CE}(A,\Tpoly{\bullet})$ 
and $\hypercohomology^\bullet_{\CE}(A,\Dpoly{\bullet})$ admit canonical Gerstenhaber algebra 
structures \cite{arXiv:1901.04602}.
 
Denoting the algebra of smooth functions on the manifold $M$ by $R$, 
we set $\Tpoly{k}=\sections{\Lambda^{k+1}(L/A)}$ for $k\geqslant 0$, $\Tpoly{-1}=R$, 
and $\Tpoly{\bullet}=\bigoplus_{k=-1}^{\infty}\Tpoly{k}$.
The Bott $A$-connection on $L/A$ makes every $\Tpoly{k}$ an ${A}$-module.
We can thus consider the complex of $A$-modules with trivial differential 
\begin{equation*} \begin{tikzcd}
0 \arrow[r] & \Tpoly{-1} \arrow[r, "0"] &
\Tpoly{0} \arrow[r, "0"] & \Tpoly{1} \arrow[r, "0"] &
\Tpoly{2} \arrow[r, "0"] & \cdots .
\end{tikzcd} \end{equation*}
Its Chevalley--Eilenberg hypercohomology cochain complex
is denoted $\big(\tot( \sections{\Lambda^\bullet A^\vee}\otimes_R \Tpoly{\bullet}),\dabott\big)$. 
Similarly, we set $\Dpoly{\bullet}=\bigoplus_{k=-1}^{\infty}\Dpoly{k}$,
where $\Dpoly{-1}=R$, $\Dpoly{0}=\frac{\enveloping{L}}{\enveloping{L}\sections{A}}$, and 
$\Dpoly{k}$ with $k\geqslant 1$ is the tensor product $\Dpoly{0}\otimes_R\cdots\otimes_R\Dpoly{0}$ 
of $(k+1)$ copies of the left $R$-module $\Dpoly{0}$. 
Multiplication in $\enveloping{L}$ from the left by elements of $\sections{A}$ induces 
an $A$-module structure on the quotient $\frac{\enveloping{L}}{\enveloping{L}\sections{A}}$. 
This action of $A$ on $\Dpoly{0}$ extends naturally to an action of $A$ on $\Dpoly{k}$ for each $k$. 
In fact, $\Dpoly{0}$ is a cocommutative coassociative coalgebra over $R$
whose comultiplication $\Delta:\Dpoly{0}\to\Dpoly{0}\otimes_R\Dpoly{0}$ 
is a morphism of $A$-modules. Therefore the Hochschild complex 
\[ \begin{tikzcd}
0 \arrow[r] & \Dpoly{-1} \arrow[r, "\hochschild"] & \Dpoly{0} \arrow[r, "\hochschild"] &
\Dpoly{1} \arrow[r, "\hochschild"] & \Dpoly{2} \arrow[r, "\hochschild"] & \cdots 
\end{tikzcd} \]
determined by the comultiplication $\Delta:\Dpoly{0}\to\Dpoly{0}\otimes_R\Dpoly{0}$ is a complex of $A$-modules. 
Its Chevalley--Eilenberg hypercohomology cochain complex is denoted 
$\big(\tot( \sections{\Lambda^\bullet A^\vee}\otimes_R \Dpoly{\bullet}),\dau+\dHH\big)$.

For instance, for the Lie pair $L=T_X\otimes\CC$ and $A=T^{0,1}_X$ stemming from a complex manifold $X$, the pair of spaces 
$\tot\big(\sections{\Lambda^\bullet A^\vee}\otimes_R\Tpoly{\bullet}\big)$
and $\tot\big(\sections{\Lambda^\bullet A^\vee}\otimes_R\Dpoly{\bullet}\big)$
are precisely the standard dglas $\big(\Omega^{0,\bullet}(\Tpoly{\bullet}(X)),\bar{\partial}\big)$
and $\big(\Omega^{0,\bullet}(\Dpoly{\bullet}(X)),\bar{\partial}+\dHH \big)$.
The corresponding Chevalley--Eilenberg hypercohomology groups
$\hypercohomology_{\CE}^{\bullet}(A,\Tpoly{\bullet})$ and $\hypercohomology_{\CE}^{\bullet}(A,\Dpoly{\bullet})$
are isomorphic to the sheaf cohomology group 
$\hypercohomology^\bullet(X,\Lambda^\bullet T_X)$
and the Hochschild cohomology group $HH^{\bullet}(X)$, respectively.

The skew-symmetric extension of the natural inclusion $\sections{L/A}\hookrightarrow\Dpoly{0}$ 
to the complex of $A$-modules $\Tpoly{\bullet}$ yields a morphism of $A$-modules $\hkr:\Tpoly{\bullet}\to\Dpoly{\bullet}$. 
The induced morphism of Chevalley--Eilenberg hypercohomology cochain complexes 
$\hkr: \tot\big(\sections{\Lambda^\bullet A^\vee}\otimes_R\Tpoly{\bullet}\big)
\to \tot\big(\sections{\Lambda^\bullet A^\vee}\otimes_R\Dpoly{\bullet}\big)$, 
which is also called \emph{Hochschild--Kostant--Rosenberg} map, is actually a quasi-isomorphism. 
It is thus natural to ask whether $\hkr$ can be extended to an $L_\infty$ quasi-isomorphism 
analogous to Kontsevich's formality quasi-isomorphism for smooth manifolds. 
The answer is negative in general and the reason is quite simple. 
For a smooth manifold $M$, the Hoch\-schild--Kostant--Rosenberg map induces 
an isomorphism of Lie algebras (in fact an isomorphism of Gerstenhaber algebras) 
from the polyvectors fields $\Tpoly{\bullet}(M)$ on $M$ equipped with the Schouten bracket
to the Hochschild cohomology $H^\bullet(\Dpoly{\bullet}(M),\hochschild)$
equipped with the Gerstenhaber brack\-et. 
However, for a Lie pair $(L,A)$, the morphism in cohomology
$\hkr: \hypercohomology^\bullet_{\CE}(A,\Tpoly{\bullet})\to \hypercohomology^\bullet_{\CE}(A,\Dpoly{\bullet})$
induce by the Hochschild--Kostant--Rosenberg map 
preserves neither the Lie algebra nor the associative algebra structures. 
The Hochschild--Kostant--Rosenberg map $\hkr$ must indeed be modified; 
it must be tweaked by the square root of the Todd cocycle of the Lie pair.

The Atiyah class of a Lie pair $(L, A)$ was introduced and studied by 
Chen--Stiénon--Xu in~\cite{MR3439229}. 
It captures the obstruction to the existence of `compatible'
 $L$-connections on $L/A$ extending the Bott $A$-representation. 
The Atiyah class of Lie pairs is a simultaneous extension of 
both the classical Atiyah class of holomorphic vector bundles 
\cite{MR0086359} and the Molino class of foliations \cite{MR0346814}.
As was first observed for holomorphic tangent bundles by 
Kapranov~\cite{MR1671737} (see also \cite{MR1671725}), 
the Atiyah class of Lie pairs is the source of homotopy Lie algebras \cite{MR3439229,MR2989383,arXiv:1408.2903}. 
Let us briefly recall its definition. Given a Lie pair $(L,A)$ with quotient $B=L/A$, 
choose an $L$-connection $\nabla$ on $B$ extending the Bott $A$-representation. 
The curvature of $\nabla$ induces a section
$\atiyahcocycle\in\sections{A^\vee\otimes A^\perp\otimes\End(L/A)}$,
which is a Chevalley--Eilenberg $1$-cocycle for the Lie algebroid $A$ 
with values in the $A$-module $A^\perp\otimes\End(L/A)$. 
Its cohomology class $\atiyahclass_{L/A}\in H_{\CE}^1(A,A^\perp\otimes\End(L/A))$
does not depend on the choice of $L$-connection $\nabla$ and is called 
Atiyah class of the Lie pair $(L,A)$.

We can assign a Todd cocycle --- defined in terms of the Atiyah cocycle --- with each Lie pair $(L,A)$ 
in the exact same way the Todd cocycle of a complex manifold is derived from its Atiyah cocycle.
The Todd cocycle of a Lie pair $(L,A)$ is the Chevalley-Eilenberg cocycle
\begin{gather}
\todd_{L/A}^{\nabla} =\det\left(\frac{\atiyahcocycle}{1-e^{-\atiyahcocycle }}\right)
\in \bigoplus_{k=0} \sections{\Lambda^kA\dual\otimes\Lambda^k A^\perp} 
.\end{gather}
Its cohomology class $\Todd_{L/A}\in \bigoplus_{k=0} H_{\CE}^{k}(A,\Lambda^k A^\perp)$
is the Todd class of the Lie pair $(L,A)$. See Section~\ref{atiyah-todd} for details.

The main goal of this paper is to establish the following formality theorem for Lie pairs: 
\textsl{There exists an $L_\infty$ quasi-isomorphism from
$\tot\big(\sections{\Lambda^\bullet A^\vee}\otimes_R\Tpoly{\bullet}\big)$
to $\tot\big(\sections{\Lambda^\bullet A^\vee}\otimes_R\Dpoly{\bullet}\big)$
whose first Taylor coefficient is equal to $\hkr\circ(\todd_{L/A}^{\nabla})^{\frac{1}{2}}$, 
with $(\todd_{L/A}^{\nabla})^{\frac{1}{2}}\in \bigoplus_{k=0} \sections{\Lambda^kA\dual\otimes\Lambda^k A^\perp}$ 
acting on $\tot\big(\sections{\Lambda^\bullet A^\vee}\otimes_R\Tpoly{\bullet}\big)$ by contractions.}
See Theorem~\ref{thm:main}.

Furthermore, we obtain the following Kontsevich--Duflo type theorem for Lie pairs:
\textsl{Given a Lie pair $(L,A)$, the map
$\hkr\circ\Todd^{\frac{1}{2}}_{L/A}: \hypercohomology^\bullet_{\CE}(A,\Tpoly{\bullet}) \to \hypercohomology^\bullet_{\CE}(A,\Dpoly{\bullet})$, 
is an isomorphism of Gerstenhaber algebras.}
See Theorem~\ref{KD-thm}.

Our result is very much inspired by Kontsevich's seminal work \cite{MR2062626}, in which it is highlighted 
that the classical Duflo theorem is one of many consequences of the formality construction.
For every Lie algebra $\mathfrak{g}$, the symmetrization map
$\pbw: S(\frakg)\to\enveloping{\frakg}$ is an 
isomorphism of $\frakg$-modules called Poincaré--Birkhoff--Witt isomorphism. 
The induced isomorphism 
$\pbw:S(\frakg)^{\frakg}\to\enveloping{\frakg}^{\frakg}$ 
between subspaces of $\frakg$-invariants does not intertwine the obvious multiplications on
$S(\frakg)^{\frakg}$ and $\enveloping{\frakg}^{\frakg}$. 
However, it can be modified so as to become an isomorphism of 
associative algebras.
The Duflo element $J\in\hat{S}(\frakg^\vee)$ of a Lie algebra $\frakg$ 
is the formal polynomial on $\frakg$ defined by 
$J(x)=\det\left(\frac{1-e^{-{\ad}_x}}{{\ad}_x}\right)$, for all $x\in\frakg$.
Considered as a translation-invariant formal differential operator on $\frakg^\vee$, 
the square root of the Duflo element defines a transformation
$J^{\frac{1}{2}}:S(\frakg)\to S(\frakg)$.
A remarkable theorem due to Duflo~\cite{MR0269777} asserts that the composition 
$\pbw\circ J^{\frac{1}{2}}:S(\frakg)^{\frakg}\to\enveloping{\frakg}^{\frakg}$
is an isomorphism of associative algebras. Duflo's theorem generalizes 
a fundamental result of Harish-Chandra regarding the center 
of the universal enveloping algebra of a semi-simple Lie algebra. 
Duflo's original proof was based on deep and sophisticated techniques 
of representation theory including Kirillov's orbit method. 
As an application of his formality construction, 
Kontsevich proposed a new proof of Duflo's theorem by means of 
the induced associative algebra structure on tangent cohomology 
at a Maurer--Cartan element.
Indeed, Kontsevich's approach \cite{MR2062626} has led to an extension of Duflo's theorem: 
\textsl{For every finite dimensional Lie algebra $\frakg$, the map
$\pbw\circ J^{\frac{1}{2}}: H_{\CE}^\bullet(\frakg,S(\frakg))
\to H_{\CE}^\bullet(\frakg,\enveloping{\frakg})$
is an isomorphism of graded associative algebras.}
The classical Duflo theorem is the isomorphism of the cohomologies in degree $0$. 
A detailed proof of the above extended Duflo theorem was given
by Pevzner--Torossian~\cite{MR2085348} (see also~\cite{MR1990011,MR2077241}).

Kontsevich discovered a similar phenomenon in complex geometry.
Recall that the Hochschild cohomology groups $HH^\bullet(X)$ of a complex manifold $X$
are defined as the groups $\Ext_{\cO_{X\times X}}^\bullet(\cO_{\Delta},\cO_{\Delta})$ \cite{MR2141853}.
Gerstenhaber--Shack~\cite{MR981619} derived an isomorphism of cohomology groups 
$\hkr: \hypercohomology^\bullet(X,\Lambda^\bullet T_X)\to HH^\bullet(X)$ 
from the classical Hochschild--Kostant--Rosenberg map.
This isomorphism fails to intertwine the multiplications in both cohomologies 
but can be tweaked so as to produce an isomorphism of associative algebras. 
More precisely, Kontsevich~\cite{MR2062626} obtained the following theorem:
\textsl{The composition 
$\hkr\circ(\Todd_X)^{\frac{1}{2}}:\hypercohomology^\bullet(X,\Lambda^\bullet T_X)\to HH^\bullet(X),$
where the symbol $\Todd_X$ denotes the Todd class of the complex manifold $X$, 
is an isomorphism of associative algebras.} 
The multiplications on $\hypercohomology^\bullet(X,\Lambda^{\bullet}T_X)$ and $HH^{\bullet}(X)$ 
are respectively the wedge product and the Yoneda product.
Calaque--Van den Bergh~\cite{MR2646112} wrote a detailed proof of Kontsevich's theorem, 
and showed that the map $\hkr\circ(\Todd_X)^{\frac{1}{2}}$ actually 
preserves the Gerstenhaber algebra structures on both cohomologies.
We refer the reader to \cite{MR2565036,MR2986860} for a related result.

Hence Kontsevich's formality revealed a hidden connection between 
two areas of mathematics: complex geometry and Lie theory.
The mysterious and surprising similarity between the Todd class of a complex manifold 
and the Duflo element of a Lie algebra --- two seemingly unrelated objects --- 
led to further exciting developments. 
In 1998, Shoikhet~\cite{arXiv:math/9812009} announced 
the so called Kontsevich--Shoikhet theorem (Theorem~\ref{thm:Shoikhet}), 
which explains the deep ties between Lie theory and complex geometry and provides a unified framework for their study. 
The theorem states that a formula of Duflo type holds for the 
dg manifolds $(\RR^{m|n},Q)$.

Our approach is inspired by Dolgushev's proof of Kontsevich's global formality theorem for smooth manifolds \cite{MR2102846} 
and relies heavily on the Fedosov dg Lie algebroid constructed in~\cite{arXiv:1605.09656,arXiv:1901.04602} 
(and independently by Batakidis--Voglaire in the special case of matched pairs~\cite{MR3724780}).
Roughly speaking, a Fedosov dg Lie algebroid associated with a Lie pair $(L,A)$ is a dg Lie algebroid whose associated 
spaces of polyvector fields and polydifferential operators 
are homotopy equivalent to
$\big(\tot(\sections{\Lambda^\bullet A^\vee}\otimes_R\Tpoly{\bullet}),\dabott\big)$
and $\big(\tot(\sections{\Lambda^\bullet A^\vee}\otimes_R\Dpoly{\bullet}),\dau+\dHH\big)$, respectively
(in a style reminiscent of Dolgushev's Fedosov resolutions~\cite{MR2102846}). 
More precisely, having chosen some additional geometric data, 
one can endow the graded manifold $\cM=L[1]\oplus L/A$ 
with a structure of dg manifold $(\cM,Q)$ homotopy equivalent to $(A[1],d_A)$ \cite{arXiv:1605.09656}. 
We call any such a dg manifold $(\cM,Q)$ a `Fedosov dg manifold 
associated with the Lie pair $(L,A)$.' The Fedosov dg Lie algebroid $\cF$ 
is a certain dg Lie subalgebroid of the tangent dg Lie algebroid $T_\cM$ of
 the Fedosov dg manifold $(\cM,Q)$. 
In other words, $\cF$ is the dg Lie algebroid encoding a certain dg foliation
 of $(\cM,Q)$.
Since a Lie algebroid can be thought of as an extension of the tangent bundle
 of a manifold, 
the notions of polyvector fields and polydifferential operators admit extensions to the context of a Lie algebroid 
and these each carry a natural dgla structure \cite{MR1675117,MR1815717}.
Likewise, the notions of polyvector fields and polydifferential operators can
be extended in an appropriate sense to the context of a dg Lie algebroid, 
yielding two dglas $\Tpoly{\bullet}$ and $\Dpoly{\bullet}$ whose corresponding cohomology groups are naturally Gerstenhaber algebras. 
The ``polyvector fields'' and ``polydifferential operators'' associated to the 
Fedosov dg Lie algebroid $\cF$
can be viewed geometrically as polyvector fields and polydifferential operators
tangent to the dg foliation on the Fedosov dg manifold $(\cM,Q)$. 
In fact, one can identify the ``polyvector fields'' and 
``polydifferential operators'' on $\cF$ 
to $\big(\tot(\sections{\Lambda^\bullet L^\vee}\otimes_R\verticalTpoly{\bullet}),\;\fedosova,\;\schouten{\argument}{\argument}\big)$ and
$\big(\tot(\sections{\Lambda^\bullet L^\vee}\otimes_R\verticalDpoly{\bullet}),\;\gerstenhaber{Q+ m}{\argument},\;\gerstenhaber{\argument}{\argument}\big)$, respectively,
where $\verticalTpoly{\bullet}$ denotes the formal polyvector fields and $\verticalDpoly{\bullet}$ 
the formal polydifferential operators tangent to the fibers of the vector bundle $L/A\to M$.

By applying Kontsevich formality theorem fiberwisely to $\cF\to \cM$,
we prove that there exists an $L_\infty$ quasi-isomorphism
\[ \Phi : \big(\tot(\sections{\Lambda^\bullet L^\vee}\otimes_R\verticalTpoly{\bullet}),\;\fedosova,\;\schouten{\argument}{\argument}\big)
\to \big(\tot(\sections{\Lambda^\bullet L^\vee}\otimes_R\verticalDpoly{\bullet}),\;\gerstenhaber{Q+ m}{\argument},\;\gerstenhaber{\argument}{\argument}\big). \]

This $L_\infty$ quasi-isomorphism $\Phi$ is in fact a sequence of maps $(\Phi_n)_{n=1}^\infty$ --- its `Taylor coefficients' --- 
the first amongst which is a quasi-isomorphism of cochain complexes
\[ \Phi_1: \tot(\sections{\Lambda^\bullet L^\vee}\otimes_R\verticalTpoly{\bullet}) 
\to \tot(\sections{\Lambda^\bullet L^\vee}\otimes_R\verticalDpoly{\bullet}) .\]
The latter induces an isomorphism of Lie algebras on the level of cohomologies.
A standard argument of Kontsevich, Manchon--Torossian, and Mochizuki 
\cite{MR2062626,MR1990011,MR2077241,MR1872382} suffices to prove that $\Phi_1$ 
intertwines the associative multiplications carried by the cohomologies as well. 
Hence, in cohomology, $\Phi_1$ really is an isomorphism of Gerstenhaber algebras.

Next, we apply the Kontsevich--Shoikhet theorem (Theorem~\ref{thm:Shoikhet}) 
in order to prove that $\Phi_1$ is essentially the fiberwise HKR map enhanced by
the Todd class of the Fedosov dg Lie algebroid. 
More precisely, we prove that $\Phi_1$ 
is the composition \[ \Phi_1=\hkr\circ(\ttodd_\cF^{\can})^{\frac{1}{2}} \]
of the natural extension $\hkr$ of the fiberwise Hochschild--Kostant--Rosenberg map
$\verticalTpoly{\bullet}\to\verticalDpoly{\bullet}$ and the action on 
$\tot(\sections{\Lambda^\bullet L^\vee}\otimes_R\verticalTpoly{\bullet})$ (by contraction) 
of the square root of the Todd cocycle $\ttodd^{\can}_\cF$ of the Fedosov dg Lie algebroid
 $\cF$ associated with the canonical connection defined by Equation~\eqref{twisted Todd}.

Then our main theorem essentially follows from a careful combination of the above results 
together with various $L_\infty$ quasi-isomorphisms. 
Our approach was largely influenced and indeed relies on several standard techniques 
pioneered by Kontsevich in his seminal paper~\cite{MR2062626} and expounded at greater length 
in subsequent literature~\cite{arXiv:math/9812009,MR2646112}.
However, we emphasize the role of Fedosov dg Lie algebroids as it sheds new light on 
and indeed provides transparent understanding of Kontsevich's global formality theorem 
and, in particular, the Kontsevich--Duflo phenomenon.

In Section~\ref{applications}, we apply our results to a number of interesting
classes of examples of Lie pairs, namely those arising from complex manifolds, 
from regular foliations, and from $\frakg$-manifolds.
In each case, we obtain a formality theorem and a Kontsevich--Duflo type theorem. 
In the case of Lie pairs stemming from complex manifolds, we recover the
Kontsevich--Duflo theorem of complex geometry \cite{MR2062626, MR2646112}.
As far as we know, the formality and Kontsevich--Duflo type theorems 
obtained for geometric situations such as foliations and $\frakg$-manifolds are new. 
In the future, we plan to investigate the implications of the formality theorem 
in deformation quantization, in particular for the special instances of Lie pairs listed above.
While Kontsevich's formality is concerned with $L_{\infty}$ algebra structures,
Tsygan's formality for chains takes care of the corresponding $L_{\infty}$ module structures \cite{MR1729368}.
An explicit formula for chain formality was first constructed by Shoikhet \cite{MR2004726}.
Formality for chains is related to Tsygan's program of noncommutative 
calculus \cite{MR2986860,MR1783778},
which is itself closely related to the algebraic index theorem 
\cite{MR1350407,MR1901245,MR2132867,MR3286536,MR2601006,MR3316973,MR3295986}.
In a separate publication, we will study Tsygan's formality for Lie pairs
and its application to the index theorem. When the Lie pair
arises from a regular foliation, it would be interesting to explore
the connection with the work of Gorokhovsky--Lott \cite{MR3056105} and 
Pflaum--Posthuma--Tang \cite{MR3316973} on the transverse index theorem.

\subsection*{Terminology and notations} 

\subsubsection*{Natural numbers}
We use the symbol $\NN$ to denote the set of positive integers and the symbol $\NN_0$ for the set of nonnegative integers. 

\subsubsection*{Field $\KK$ and ring $R$}
We use the symbol $\KK$ to denote the field of either real or complex numbers. 
The symbol $R$ always denotes the algebra of smooth functions on $M$ with values in $\KK$. 

\subsubsection*{Tensor products}
For any two $R$-modules $P$ and $Q$, we write $P\otimes_{R} Q$ to denote the tensor product of $P$ and $Q$ 
as $R$-modules and $P\otimes Q$ to denote the tensor product of $P$ and $Q$ regarded as $\KK$-modules. 

\subsubsection*{Completed symmetric algebra}
Given a module $\cM$ over a ring, the symbol $\hat{S}(\cM)$ denotes 
the $\mathfrak{m}$-adic completion of the symmetric algebra $S(\cM)$, 
where $\mathfrak{m}$ is the ideal of $S(\cM)$ generated by $\cM$. 

\subsubsection*{Duality pairing}
For every vector bundle $E\to M$, we define a duality pairing 
\[ \sections{\hat{S}(E\dual)}\times\sections{S(E)} \to R \] 
by 
\[ \duality{\nu_1\otimes\cdots\otimes\nu_p}{v_1\otimes\cdots\otimes v_q} 
=\begin{cases}\sum_{\sigma\in S_p}\prod_{k=1}^p\duality{\nu_k}{v_{\sigma(k)}} 
& \text{if } p=q , \\ 0 & \text{otherwise.} \end{cases} \] 

\subsubsection*{Multi-indices}
Let $E\to M$ be a smooth vector bundle of finite rank $r$,
let $(\partial_i)_{i\in\{1,\dots,r\}}$ be a local frame of $E$ 
and let $(\chi_j)_{j\in\{1,\dots,r\}}$ be the dual local frame of $E\dual$.
Thus, we have $\duality{\chi_i}{\partial_j}=\delta_{i,j}$.
Given a multi-index $I=(I_1,I_2,\cdots,I_r)\in\NO^r$, 
we adopt the following multi-index notations:
\begin{gather*}
I! = I_1! \cdot I_2! \cdots I_r! \\ 
\abs{I} = I_1+I_2+\cdots+I_r \\ 
\partial^I=\underset{I_1 \text{ factors}}{\underbrace{\partial_1\odot\cdots\odot \partial_1}}
\odot\underset{I_2 \text{ factors}}{\underbrace{\partial_2\odot\cdots\odot\partial_2}}
\odot\cdots\odot\underset{I_r \text{ factors}}{\underbrace{\partial_r\odot\cdots\odot\partial_r}} \\ 
\chi^I=\underset{I_1 \text{ factors}}{\underbrace{\chi_1\odot\cdots\odot\chi_1}}
\odot\underset{I_2 \text{ factors}}{\underbrace{\chi_2\odot\cdots\odot\chi_2}}
\odot\cdots\odot\underset{I_r \text{ factors}}{\underbrace{\chi_r\odot\cdots\odot\chi_r}}
\end{gather*}
We use the symbol $e_k$ to denote the multi-index all of whose 
components are equal to $0$ except for the $k$-th which is equal to $1$. 
Thus $\chi^{e_k}=\chi_k$.

\subsubsection*{Shuffles} 
A $(p,q)$-shuffle is a permutation $\sigma$ of the set $\{1,2,\cdots,p+q\}$ 
such that $\sigma(1)<\sigma(2)<\cdots<\sigma(p)$ 
and $\sigma(p+1)<\sigma(p+2)<\cdots<\sigma(p+q)$.
The symbol $\shuffle{p}{q}$ denotes the set of $(p,q)$-shuffles. 

\subsubsection*{Graduation shift} 
Given a graded vector space $V=\bigoplus_{k\in\ZZ}V_k$, 
we write $V[i]$ to denote the graded vector space obtained by shifting the grading on $V$ 
according to the rule $(V[i])_{k}=V_{i+k}$. Accordingly, if $E=\bigoplus_{k\in\ZZ}E_{k}$
is a graded vector bundle over $M$, $E[i]$ denotes
the graded vector bundle obtained by shifting the degree in the fibers
of $E$ according to the above rule.

\subsubsection*{Koszul sign}
The Koszul sign $\sgn(\sigma; v_1, \cdots, v_{n})$ 
of a permutation $\sigma$ of homogeneous vectors $v_1$, $v_2$, \dots, $v_n$ 
of a $\ZZ$-graded vector space $V=\bigoplus_{k\in\ZZ}V_k$ is determined by
the equality \[ v_{\sigma(1)}\odot v_{\sigma(2)}\odot\cdots\odot v_{\sigma(n)}
= \sgn(\sigma; v_1,\cdots,v_n) \cdot v_1\odot v_2\odot\cdots\odot v_n \] 
in the graded commutative algebra $S(V)$.

\subsubsection*{Lie algebroid}
In this paper `Lie algebroid' always means `Lie $\KK$-algebroid' unless specified otherwise. 

\subsubsection*{Contraction}
Let $(C,\eth)$ and $(K,d)$ be two cochain complexes. 
A contraction of $(K,d)$ onto $(C,\eth)$ consists of a pair of chain maps $\tau: C\to K$ and $\sigma: K\to C$ 
together with a chain homotopy operator $h: K\to K[-1]$ satisfying 
\begin{gather*} \sigma\tau=\id_C; \qquad \id_K-\tau\sigma=dh+hd ; \\ 
\sigma h =0; \qquad h^2=0; \qquad \text{and} \qquad h\tau=0 .\end{gather*} 
We symbolize such a contraction by a diagram 
\[ \begin{tikzcd} 
(C,\eth) \arrow[r, "\tau", shift left] & (K,d) \arrow[l, "\sigma", shift left] \arrow[loop right, "h"] 
\end{tikzcd} .\] 

\subsection*{Acknowledgements} 

We would like to thank 
Ruggero Bandiera, Damien Broka, Martin Bordemann, Vasily Dolgushev,
 Olivier Elchinger,
Camille Laurent-Gengoux, Kirill Mackenzie, Do\-mi\-ni\-que Manchon,
 Marco Manetti, Rajan Mehta, Michael Pevzner,	 Boris Shoikhet, Jim Stasheff, 
Dmitry Tamarkin, and Thomas Will\-wa\-cher for fruitful discussions and useful comments. 
Stiénon is grateful to Université Paris~7 for its hospitality during his sabbatical leave in 2015--2016. Liao would like to thank National Center
for Theoretical Science for its hospitality in August--September 2018.

\section{Preliminaries}

\subsection{Connections and representations for Lie algebroids}

Let $M$ be a smooth manifold, let $L\to M$ be a Lie $\KK$-algebroid with anchor map $\rho:L\to T_M\otimes_{\RR}\KK$, 
and let $E\to M$ be a vector bundle over $\KK$. The algebra of smooth functions on $M$ with values in $\KK$ will be denoted $R$. 

The traditional description of a (linear) $L$-connection on $E$ is in terms of a \emph{covariant derivative} 
\[ \sections{L}\times\sections{E}\to\sections{E}: (l,e)\mapsto \nabla_l e \] 
characterized by the following two properties: 
\begin{gather}
\nabla_{f\cdot l} e=f\cdot \nabla_l e , \label{Molenbeek} \\ 
\nabla_l (f\cdot e)=\rho(l)f\cdot e+f\cdot\nabla_l e \label{Anderlecht}
,\end{gather}
for all $l\in\sections{L}$, $e\in\sections{E}$, and $f\in R$.

\begin{remark}
A covariant derivative 
$\nabla:\sections{L}\times\sections{E}\to\sections{E}$ 
induces a covariant derivative 
$\nabla:\sections{L}\times\sections{S(E)}\to\sections{S(E)}$ 
through the relation
\[ \nabla_l (e_1\odot\cdots\odot e_n) = \sum_{k=1}^n e_1\odot\cdots\odot\nabla_l e_k\odot\cdots\odot e_n ,\] 
for all $l\in\sections{L}$ and $e_1,\dots,e_n\in\sections{E}$.
\end{remark}

\begin{remark}
A covariant derivative 
$\nabla:\sections{L}\times\sections{S(E)}\to\sections{S(E)}$ 
induces a covariant derivative 
$\nabla:\sections{L}\times\sections{\hat{S}(E\dual)}\to\sections{\hat{S}(E\dual)}$ 
through the relation
\[ \rho(l)\duality{\sigma}{s}=\duality{\nabla_l\sigma}{s}+\duality{\sigma}{\nabla_l s} \] 
for all $l\in\sections{L}$, $s\in\sections{S(E)}$, 
and $\sigma\in\sections{\hat{S}(E\dual)}$. 
\end{remark}

A \emph{representation of a Lie algebroid} $L$ on a vector bundle $E\to M$ is 
a flat $L$-connection $\nabla$ on $E$, i.e.\ a covariant derivative 
$\nabla:\sections{L}\times\sections{E}\to\sections{E}$ satisfying 
\begin{equation}\label{Jette}
\nabla_{l_1}\nabla_{l_2} e-\nabla_{l_2}\nabla_{l_1} e=\nabla_{\lie{l_1}{l_2}}e ,
\end{equation} 
for all $l_1,l_2\in\sections{L}$ and $e\in\sections{E}$. 
A vector bundle endowed with a representation of the Lie algebroid $L$ is called an \emph{$L$-module}. 

\begin{example}\label{Auderghem}
Let $(L,A)$ be a Lie pair, i.e.\ an inclusion $A\into L$ of Lie algebroids. 
The \emph{Bott representation} of $A$ on the quotient $L/A$ is the flat connection defined by 
\[ \nabla^{\Bott}_a q(l)=q\big(\lie{a}{l}\big), \quad\forall a\in\sections{A},l\in\sections{L} ,\] 
where $q$ denotes the canonical projection $L\onto L/A$. 
Thus the quotient $L/A$ of a Lie pair $(L,A)$ is an $A$-module. 
\end{example}

Let $L$ be a Lie algebroid over a smooth manifold $M$, and $R$ be the algebra of smooth functions on $M$ valued in $\KK$.
The Chevalley--Eilenberg differential 
\[ d_L:\sections{\Lambda^k L\dual}\to\sections{\Lambda^{k+1} L\dual} \] 
defined by 
\begin{multline*} 
\big(d_L \omega\big)(l_0,l_1,\cdots,l_k)=
\sum_{i=0}^{n} (-1)^{i} \rho(l_i)\big(\omega(l_0,\cdots,\widehat{l_i},\cdots,l_k)\big) \\ 
+\sum_{i<j} (-1)^{i+j} \omega(\lie{l_i}{l_j},l_0,\cdots,\widehat{l_i},\cdots,\widehat{l_j},\cdots,l_k) 
\end{multline*}
and the exterior product make $\bigoplus_{k\geqslant 0}\sections{\Lambda^k L\dual}$ 
into a differential graded commutative $R$-algebra.

Given a Lie algebroid $L$ of rank $n$, and
an $L$-connection $\nabla$
 on a vector bundle $E\to M$,
the covariant differential is the operator 
\[ d_L^{\nabla}: \sections{\Lambda^k L\dual\otimes E}\to\sections{\Lambda^{k+1} L\dual\otimes E} \] 
that takes a section $\omega \otimes e$ of 
$\Lambda^k L\dual\otimes E$ to 
\[ d_L^{\nabla}(\omega\otimes e)=(d_L\omega)\otimes e
+\sum_{j=1}^{n}(\nu_j\wedge\omega)\otimes \nabla_{v_j}e ,\]
where $v_1,v_2,\dots,v_n$ and $\nu_1,\nu_2,\dots,\nu_n$ 
are any pair of dual local frames for the vector bundles $L$ and $L\dual$. 
If the connection $\nabla$ is flat, then
 $d_L^{\nabla}$ is a coboundary 
operator: $d_L^{\nabla}\circ d_L^{\nabla}=0$.

Let $L$ be a Lie algebroid and let
\[ \begin{tikzcd}
0 \arrow{r} & \EE^{-1} \arrow{r}{d} & \EE^{0} \arrow{r}{d} & \EE^{1} \arrow{r}{d} & \EE^{2} \arrow{r}{d} & \cdots
\end{tikzcd} \]
be a  complex of $\enveloping{L}$-modules.
The Chevalley--Eilenberg hypercohomology group $\hypercohomology_{\CE}^k(L,\EE^\bullet)$ 
is the degree~$k$ total cohomology  of the double complex 
\[ \begin{tikzcd}[row sep=small]
& \vdots & \vdots & \vdots & \\ 
\cdots \arrow[r] & \sections{\Lambda^{p-1}L\dual}\otimes_R\EE^{q+1} \arrow[r, "d_L^{\EE}"] \arrow[u] & 
\sections{\Lambda^{p}L\dual}\otimes_R\EE^{q+1} \arrow[r, "d_L^{\EE}"] \arrow[u] & 
\sections{\Lambda^{p+1}L\dual}\otimes_R\EE^{q+1} \arrow[r] \arrow[u] & \cdots \\ 
\cdots \arrow[r] & 
\sections{\Lambda^{p-1}L\dual}\otimes_R\EE^{q} \arrow[r, "d_L^{\EE}"] \arrow[u, "\id\otimes (-1)^{p-1}d_{\EE}"] & 
\sections{\Lambda^{p}L\dual}\otimes_R\EE^{q} \arrow[r, "d_L^{\EE}"] \arrow[u, "\id\otimes (-1)^{p}d_{\EE}"] & 
\sections{\Lambda^{p+1}L\dual}\otimes_R\EE^{q} \arrow[r] \arrow[u, "\id\otimes (-1)^{p+1}d_{\EE}"] & \cdots \\ 
& \vdots \arrow[u] & \vdots \arrow[u] & \vdots \arrow[u] & 
\end{tikzcd} \]
When we say that the above diagram is a double complex, we mean in particular that each square of the grid commutes.
Hence the hypercohomology $\hypercohomology_{\CE}^k(L,\EE^\bullet)$ is the cohomology of the total complex
\[ \Big(\tot\big(\sections{\Lambda^\bullet L\dual}\otimes_R\EE^{\bullet}\big), d_L^{\EE}+ \id\otimes d_{\EE}\Big) .\]
Recall that, the degree of the operator $d_{\EE}$ being $+1$, the usual sign convention for the tensor product of linear 
maps in the presence of gradings dictates that 
\[ \big(\id\otimes d_\EE\big)(\omega\otimes e) = (-1)^p\omega\otimes d_\EE(e),\quad 
\forall\omega\in\sections{\Lambda^p L\dual},\ \forall e\in\EE^\bullet .\]

\subsection{Atiyah class and Todd class of a Lie pair}\label{atiyah-todd}

Let $(L,A)$ be a pair of Lie algebroids over $\KK$. 
We write $B$ to denote the quotient vector bundle $L/A$. 
Consider the short exact sequence of vector bundles 
\[ \begin{tikzcd}
0 \arrow[r] & A \arrow[r, "i"] & L \arrow[r, "q"] & B \arrow[r] & 0
\end{tikzcd} .\]

Given $L$-connection $\nabla$ on $B$, 
we define a bundle map $\torsion:\Lambda^2 L\to B$ by 
\[ \torsion(x,y)=\nabla_{x}q(y)-\nabla_{y}q(x)-q\big(\lie{x}{y}\big), 
\quad\forall x,y\in\sections{L} .\] 

An $L$-connection $\nabla$ on $B$ is said to extend 
the Bott $A$-connection on $B$ (see Example~\ref{Auderghem}) if 
\[ \nabla_{i(a)} q(l)=\nabla^{\Bott}_a q(l)
=q\big(\lie{i(a)}{l}\big) ,\quad\forall a\in\sections{A},l\in\sections{L} .\]

\begin{lemma}\label{eq:dog}
The following assertions are equivalent: 
\begin{enumerate}
\item The $L$-connection $\nabla$ on $B$ extends the Bott $A$-connection on $B$. 
\item For all $a\in\sections{A}$ and $l\in\sections{L}$, we have $\torsion\big(i(a),l\big)=0$. 
\item There exists a unique bundle map $\beta^\nabla:\Lambda^2(B)\to B$ 
such that the diagram 
\[ \begin{tikzcd}
\Lambda^2 L \arrow[d,"q", swap] \arrow[r, "\torsion"] & B \\ 
\Lambda^2 B \arrow[ru ,"\beta^\nabla", swap] &
\end{tikzcd} \]
commutes. 
\end{enumerate}
\end{lemma}

\begin{proof}
Since $q\circ i=0$, we have \[ 0=\torsion\big(i(a),l\big)=\nabla_{i(a)} q(l)-q\big(\lie{i(a)}{l}\big) ,\] 
for all $a\in\sections{A}$ and $l\in\sections{L}$. 
\end{proof} 

Hence a torsion-free $L$-connection on $B$ is necessarily an extension of the Bott $A$-connection. 

An $L$-connection $\nabla$ on $B$ is said to be torsion-free if $\torsion=0$ (and hence $\beta^\nabla=0$). 

\begin{lemma}
Given a Lie pair $(L,A)$ with quotient $B=L/A$, there exist torsion-free $L$-connections on $B$. 
\end{lemma}

\begin{proof}[Sketch of proof]
First, construct an $L$-connection $\nabla$ on $B$ using the usual partition of unity argument. 
Then, tweak $\nabla$ so as to obtain an extension $\nabla':\sections{L}\times\sections{B}\to\sections{B}$ of the Bott $A$-connection: 
choose a splitting $i\circ p+j\circ q=\id_L$ of the short exact sequence 
\begin{equation} \label{cryptography} \begin{tikzcd} 
0 \arrow[r] & A \arrow[r, "i"] & L \arrow[l, "p", bend left, dashed] \arrow[r, "q"] & B \arrow[r] \arrow[l, "j", bend left, dashed] & 0 
\end{tikzcd} \end{equation} 
and set \[ \nabla'_l b = q\big(\lie{i\circ p(l)}{j(b)}\big)+\nabla_{j\circ q(l)} b .\] 
Finally, obtain a torsion-free connection $\nabla'':\sections{L}\times\sections{B}\to\sections{B}$ from $\nabla'$ 
by setting \[ \nabla''_l b=\nabla'_l b -\frac{1}{2}\beta^{\nabla'}\big(q(l),b\big) .\qedhere\] 
\end{proof}

Given a Lie pair $(L,A)$ with quotient $B=L/A$, let $\nabla$ be an $L$-connection on $B$ extending the Bott $A$-connection. 
The curvature of $\nabla$ is the bundle map $R^\nabla:\Lambda^2 L\to\End(B)$ defined by
\begin{equation*}
R^\nabla(x,y)=\nabla_{x}\nabla_{y}-\nabla_{y}\nabla_{x}
-\nabla_{\lie{x}{y}}, \quad \forall x,y\in\sections{L}.
\end{equation*}
Since $B$ is an $A$-module, the restriction of $R^\nabla$ to $\Lambda^2 A$ vanishes.
Hence the curvature induces a section
$\atiyahcocycle\in\sections{A^\vee\otimes A^\perp\otimes\End (B)}$
or, equivalently, a bundle map $\atiyahcocycle:A\otimes (L/A)\to\End(B)$
given by
\[ \atiyahcocycle\big(a;q(l)\big)=R^\nabla(a,l)=\nabla_{a}\nabla_{l}-\nabla_{l}\nabla_{a}-\nabla_{\lie{a}{l}},
\quad\forall a\in\sections{A},l\in\sections{L} .\]

\begin{proposition} [\cite{MR3439229}]
\label{Thm:atiyahclass}
\begin{enumerate}
\item The section $\atiyahcocycle\in \sections{A\dual\otimes A^\perp\otimes\End(L/A)}$
is a $1$-cocycle for the Lie algebroid $A$ with values in the $A$-module $A^\perp\otimes\End(L/A)$.
\item The cohomology class $\atiyahclass_{L/A}\in 
H_{\CE}^1(A, A^\perp\otimes\End(L/A))$ of the $1$-cocycle $\atiyahcocycle$
does not depend on the choice of $L$-connections extending the Bott $A$-connection.
\end{enumerate}
\end{proposition}

We call $\atiyahcocycle$ the \emph{Atiyah cocycle}
associated with the $L$-connection
$\nabla$. Its cohomology class
\[ \atiyahclass_{L/A}\in H_{\CE}^1\big(A, A^\perp\otimes\End(L/A)\big) = H_{\CE}^1\big(A,B\dual\otimes\End(B)\big) \]
is called the \emph{Atiyah class} of the Lie pair.

Choosing a splitting $i\circ p+j\circ q=\id_L$ of the short exact sequence \eqref{cryptography}, 
we can identify $\Lambda^2 L\dual$ with the Whitney sum $\Lambda^2 A\dual \oplus (A\dual\otimes B\dual) \oplus \Lambda^2 B\dual$.

The following lemma will be needed later on.

\begin{lemma}
Under the identification above,
the curvature of $\nabla$ decomposes as 
\[ R^\nabla = \widetilde{\atiyahcocycle} + R^\nabla_{0,2} \] 
where 
$\widetilde{\atiyahcocycle}\in \sections{\Lambda^2 L\dual\otimes\End(B)}$ 
denotes the skew-symmetrization
of $\atiyahcocycle \in \sections{A\dual\otimes B\dual\otimes\End(B)}$,
and $R^\nabla_{0, 2}: \Lambda^2 L\to \End(B)$ is the 
bundle map defined by 
\[ R^\nabla_{0, 2}(x,y) = R^\nabla\big(j\circ q(x),j\circ q(y)\big), \quad \forall x,y\in\sections{L} .\]
\end{lemma}

The Todd cocycle of a Lie pair $(L,A)$ is the 
Chevalley-Eilenberg cocycle
\begin{gather*} 
\todd_{L/A}^{\nabla}=\det\left(\frac{\atiyahcocycle}{1-e^{-\atiyahcocycle}}\right)
\in \bigoplus_{k=0}\sections{\Lambda^kA\dual\otimes\Lambda^k A^\perp} ,\\
\ttodd_{L/A}^{\nabla}=\det \left(\frac{\atiyahcocycle}
{e^{\half\atiyahcocycle}-e^{-\half\atiyahcocycle}}\right)
\in\bigoplus_{k=0}\sections{\Lambda^kA\dual\otimes\Lambda^k A^\perp} 
.\end{gather*}

The Todd class of a Lie pair $(L,A)$ is the cohomology class
\begin{gather*} 
\Todd_{L/A}=\det\left(\frac{\atiyahclass_{L/A}}{1-e^{-\atiyahclass_{L/A}}}\right)
\in \bigoplus_{k=0} H_{\CE}^{k}(A,\Lambda^k A^\perp ) ,\\ 
\tTodd_{L/A}=\det \left(\frac{\atiyahclass_{L/A}}
{e^{\half\atiyahclass_{L/A}}-e^{-\half\atiyahclass_{L/A}}}\right)
\in\bigoplus_{k=0} H_{\CE}^{k}(A,\Lambda^k A^\perp) 
.\end{gather*}

In the particular case of the Lie pair comprised of the Lie $\CC$-algebroids $L=T_X \otimes \CC$ and $A=T_X^{0,1}$ 
associated with a complex manifold $X$, the quotient of the pair is $T_X^{1,0}$ and the Atiyah class 
and the Todd class of the pair are the classical Atiyah class of $T_X$ and the classical Todd class of the complex manifold $X$.

\subsection{Polydifferential operators}

The universal enveloping algebra $\enveloping{L}$ of the Lie algebroid $L$ 
is a coalgebra over $R$ --- see~\cite{MR1815717}. 
Its comultiplication
\[ \Delta:\enveloping{L}\to\enveloping{L}\otimes_R\enveloping{L} \] 
is characterized by the identities 
\begin{gather*}
\Delta(1)=1\otimes 1; \\ 
\Delta(x)=1\otimes x+x\otimes 1, \quad \forall x\in \sections{L}; \\ 
\Delta(u\cdot v)=\Delta(u)\cdot\Delta(v), \quad \forall u,v\in\enveloping{L} ,
\end{gather*}
where $1\in R$ denotes the constant function on $M$ with value $1$ while the symbol $\cdot$ denotes 
the multiplication in $\enveloping{L}$. We refer the reader to~\cite{MR1815717} for the precise meaning of the last equation above. 
Explicitly, we have
\begin{multline*} 
\Delta(l_1\cdot l_2\cdot\cdots\cdot l_n)=
1\otimes(l_1\cdot l_2\cdot\cdots\cdot l_n) 
+ \sum_{\substack{p+q=n \\ p,q\in\NN}}\sum_{\sigma\in\shuffle{p}{q}} 
(l_{\sigma(1)}\cdot\cdots\cdot l_{\sigma(p)}) \otimes
(l_{\sigma(p+1)}\cdot\cdots\cdot l_{\sigma(n)}) \\
+ (l_1\cdot l_2\cdot\cdots\cdot l_n)\otimes 1
,\end{multline*} 
for all $l_1,\dots,l_n\in\sections{L}$. 

Let $(L,A)$ be a pair of Lie algebroids over $\KK$. 
Writing $\enveloping{L}\sections{A}$ for the left ideal of $\enveloping{L}$ generated by $\sections{A}$, 
the quotient $\frac{\enveloping{L}}{\enveloping{L}\sections{A}}$ is automatically an $R$-coalgebra since 
\[ \Delta\big(\enveloping{L}\sections{A}\big)\subseteq \enveloping{L}\otimes_R \big(\enveloping{L}\sections{A}\big) 
+ \big(\enveloping{L}\sections{A}\big) \otimes_R \enveloping{L} . \]

Let $\Dpoly{-1}$ denote the algebra $R$ of smooth functions on the manifold $M$, 
let $\Dpoly{0}$ denote the left $\enveloping{A}$-module 
$\frac{\enveloping{L}}{\enveloping{L}\sections{A}}$, 
let $\Dpoly{k}$ denote the tensor product 
$\Dpoly{0}\otimes_R\cdots\otimes_R\Dpoly{0}$ of $(k+1)$ copies of 
the left $R$-module $\Dpoly{0}$, and set 
$\Dpoly{\bullet}=\bigoplus_{k=-1}^{\infty}\Dpoly{k}$. 
Since $\enveloping{A}$ is a Hopf algebroid, it follows that
for each $k\geqslant -1$, $\Dpoly{k}$ is also naturally a $\enveloping{A}$-module
\cite{MR1815717}. 

\begin{lemma}
The $\enveloping{A}$-module $\Dpoly{0}$ is a cocommutative coassociative coalgebra over $R$ 
whose comultiplication $\Delta:\Dpoly{0}\to\Dpoly{0}\otimes_R\Dpoly{0}$ is
a morphism of $\enveloping{A}$-modules.
\end{lemma}

Since the comultiplication $\Delta$ is coassociative,
the Hochschild operator $\hochschild:\Dpoly{k-1}\to\Dpoly{k}$ defined by 
\begin{multline*} 
\hochschild(u_1\otimes\cdots\otimes u_k) = 1\otimes u_1\otimes\cdots\otimes u_k 
+ \sum_{i=1}^{k} (-1)^i u_1\otimes\cdots\otimes\Delta(u_i)\otimes\cdots\otimes u_k \\ 
+ (-1)^{k+1} u_1\otimes\cdots\otimes u_k\otimes 1 
,\end{multline*}
for all $u_1,u_2,\dots,u_k\in\Dpoly{0}$, 
is a coboundary operator, i.e.\ $\hochschild^2=0$.

Moreover, $\hochschild:\Dpoly{k-1}\to\Dpoly{k}$ is 
a morphism of $\enveloping{A}$-modules, 
since the comultiplication 
$\Delta:\Dpoly{0}\to\Dpoly{0}\otimes_R\Dpoly{0}$ is 
a morphism of $\enveloping{A}$-modules.
Therefore, the Hochschild complex 
\[ \begin{tikzcd} 
0 \arrow[r] & \Dpoly{-1} \arrow[r, "\hochschild"] & \Dpoly{0} \arrow[r, "\hochschild"] & 
\Dpoly{1} \arrow[r, "\hochschild"] & \Dpoly{2} \arrow[r, "\hochschild"] & \cdots
\end{tikzcd} \]
is a complex of $\enveloping{A}$-modules.

The Chevalley--Eilenberg hypercohomology  $\hypercohomology^k_{\CE}(A,\Dpoly{\bullet})$ is 
the degree~$k$ total cohomology of the double complex 
\begin{equation}\label{double cx: CE Dpoly}
\begin{tikzcd}[row sep=small]
\vdots & \vdots & \vdots & \\ 
\sections{\Lambda^0 A\dual}\otimes_R\Dpoly{0} \arrow[u, "\id\otimes\hochschild"] \arrow[r, "\dau"] & 
\sections{\Lambda^1 A\dual}\otimes_R\Dpoly{0} \arrow[u, "-\id\otimes\hochschild"] \arrow[r, "\dau"] & 
\sections{\Lambda^2 A\dual}\otimes_R\Dpoly{0} \arrow[u, "\id\otimes\hochschild"] \arrow[r, "\dau"] & \cdots \\
\sections{\Lambda^0 A\dual}\otimes_R\Dpoly{-1} \arrow[u, "\id\otimes\hochschild"] \arrow[r, "\dau"] & 
\sections{\Lambda^1 A\dual}\otimes_R\Dpoly{-1} \arrow[u, "-\id\otimes\hochschild"] \arrow[r, "\dau"] & 
\sections{\Lambda^2 A\dual}\otimes_R\Dpoly{-1} \arrow[u, "\id\otimes\hochschild"] \arrow[r, "\dau"] & \cdots 
\end{tikzcd} 
\end{equation}
The horizontal coboundary operator 
$\dau:\sections{\Lambda^p A\dual}\otimes
\Dpoly{q}\to\sections{\Lambda^{p+1} A\dual}\otimes
\Dpoly{q}$ is defined by 
\[ \begin{split} \dau(\omega\otimes u_0\otimes\cdots\otimes u_q) 
=&\ (d_A\omega)\otimes u_0\otimes\cdots\otimes u_q \\ 
&\ +\sum_{j=1}^{\rk(A)}\sum_{k=0}^{q} (\alpha_j\wedge\omega) 
\otimes u_0\otimes \cdots\otimes u_{k-1} \otimes 
a_j\cdot u_k\otimes u_{k+1} \otimes \cdots\otimes u_q 
,\end{split} \] 
for all $\omega\in\sections{\Lambda^p A\dual}$ and $u_1,u_2,\dots,u_k\in\Dpoly{0}$. 
Here $(a_i)_{i\in\{1,\dots,r\}}$ designates any local frame of $A$ 
and $(\alpha_j)_{j\in\{1,\dots,r\}}$ the corresponding dual local frame of $A\dual$.

Thus, $\hypercohomology^k_{\CE}(A,\Dpoly{\bullet})$ is the cohomology 
of the total complex 
\[ \Big( \tot\big(\sections{\Lambda^\bullet A\dual}\otimes_R\Dpoly{\bullet}\big),\dau+\dHH \big) ,\]
where we use the abbreviated symbol $\dHH$ to denote the operator $\id\otimes\hochschild$. 
Recall that, the degree of the operator $\hochschild$ being $+1$, the usual sign convention for the tensor product of linear 
maps in the presence of gradings dictates that 
\[ \big(\id\otimes\hochschild\big)(\omega\otimes u) = (-1)^p\omega\otimes\hochschild(u),\quad 
\forall\omega\in\sections{\Lambda^p A\dual},\ \forall u\in\Dpoly{\bullet} .\]

In the sequel, we will refer to $\hypercohomology^\bullet_{\CE}(A,\Dpoly{\bullet})$ 
as the Hochschild cohomology of the Lie pair $(L,A)$.

\begin{lemma}
For any Lie pair $(L, A)$, the Hochschild cohomology $\hypercohomology^\bullet_{\CE}(A,\Dpoly{\bullet})$ 
is an associative algebra, whose multiplication stems from
the tensor product of left $R$-modules $\otimes_R$ in $\Dpoly{\bullet}$.
\end{lemma}

\begin{remark}
Since, unlike the universal enveloping algebra $\enveloping{L}$ of a Lie algebroid $L$, 
the space $\Dpoly{0}$ is not a Hopf algebroid, 
the Gerstenhaber bracket \eqref{hazmat} does not extend to $\Dpoly{\bullet}$ 
(and $\sections{\Lambda^\bullet A\dual}\otimes_R\Dpoly{\bullet}$). 
Therefore, unlike $\hypercohomology^\bullet_{\CE}(A,\enveloping{L}^{\otimes\bullet +1})$, 
the Hochschild cohomology $\hypercohomology^\bullet_{\CE}(A,\Dpoly{\bullet})$ 
does not, a priori, admit a Gerstenhaber algebra structure. 
However, it turns out that $\tot(\sections{\Lambda^\bullet A\dual}\otimes_R\Dpoly{\bullet})$ does actually admit an $L_\infty$ structure %#liao103
and that its cohomology admits a Gerstenhaber algebra structure. These arise from what we call a Fedosov dg Lie algebroid 
associated with the Lie pair $(L,A)$ --- see Corollary~\ref{thm:Naples}.
\end{remark}

\subsection{Polyvector fields}

Given a Lie pair $(L,A)$, let $\Tpoly{-1}$ denote the algebra $R$ of smooth functions on the manifold $M$, 
set $\Tpoly{k}:=\sections{\Lambda^{k+1}(L/A)}$ 
for $k\geqslant 0$, and 
consider $\Tpoly{\bullet}=\bigoplus_{k=-1}\Tpoly{k}$ as a complex of $\enveloping{A}$-modules with trivial differential: 
\begin{equation*} \begin{tikzcd} 
0 \arrow[r] & \Tpoly{-1} \arrow[r, "0"] & 
\Tpoly{0} \arrow[r, "0"] & \Tpoly{1} \arrow[r, "0"] & 
\Tpoly{2} \arrow[r, "0"] & \cdots
\end{tikzcd} \end{equation*} 
The Chevalley--Eilenberg hypercohomology $\hypercohomology^k_{\CE}(A,\Tpoly{\bullet})$
is the degree~$k$ total cohomology of the double complex
\begin{equation}\label{double cx: CE Tpoly}
\begin{tikzcd}[row sep=small]
\vdots & \vdots & \vdots & \\ 
\sections{\Lambda^0 A\dual}\otimes_R\Tpoly{0} \arrow[u, "0"] \arrow[r, "d_A^{\Bott}"] & 
\sections{\Lambda^1 A\dual}\otimes_R\Tpoly{0} \arrow[u, "0"] \arrow[r, "d_A^{\Bott}"] & 
\sections{\Lambda^2 A\dual}\otimes_R\Tpoly{0} \arrow[u, "0"] \arrow[r, "d_A^{\Bott}"] & \cdots \\
\sections{\Lambda^0 A\dual}\otimes_R\Tpoly{-1} \arrow[u, "0"] \arrow[r, "d_A^{\Bott}"] & 
\sections{\Lambda^1 A\dual}\otimes_R\Tpoly{-1} \arrow[u, "0"] \arrow[r, "d_A^{\Bott}"] & 
\sections{\Lambda^2 A\dual}\otimes_R\Tpoly{-1} \arrow[u, "0"] \arrow[r, "d_A^{\Bott}"] & \cdots 
\end{tikzcd}
\end{equation}

The coboundary operator 
$d_A^{\Bott}:\sections{\Lambda^p A\dual}\otimes
\Tpoly{q}\to\sections{\Lambda^{p+1} A\dual}\otimes
\Tpoly{q}$ is defined by 
\[ \begin{split} d_A^{\Bott}(\omega\otimes b_0\wedge\cdots\wedge b_q) 
=&\ (d_A\omega)\otimes b_0\wedge\cdots\wedge b_q \\ 
&\ +\sum_{j=1}^{\rk(A)}\sum_{k=0}^{q} (\alpha_j\wedge\omega) 
\otimes b_0\wedge \cdots\wedge b_{k-1} \wedge 
\nabla^{\Bott}_{a_j} b_k\wedge b_{k+1} \wedge\cdots\wedge b_q 
,\end{split} \] 
for all $\omega\in\sections{\Lambda^p A\dual}$ and $b_0,b_1,\dots,b_q\in\sections{L/A}$. 
Here $(a_i)_{i\in\{1,\dots,r\}}$ designates any local frame of $A$ 
and $(\alpha_j)_{j\in\{1,\dots,r\}}$ the corresponding dual local frame of $A\dual$.

\begin{lemma}
For any Lie pair $(L, A)$, the cohomology $\hypercohomology^\bullet_{\CE}(A,\Tpoly{\bullet})$
is an associative algebra, whose multiplication stems from the wedge product on $\Tpoly{\bullet}$.
\end{lemma}

\begin{remark}
Again, $\hypercohomology^\bullet_{\CE}(A,\Tpoly{\bullet})$ does not, a priori, admit a Gerstenhaber algebra structure. 
However, it turns out that $\tot(\sections{\Lambda^\bullet A\dual}\otimes_R\Tpoly{\bullet})$ does actually admit an $L_\infty$ structure
and that its cohomology admits a Gerstenhaber algebra structure. These arise from what we call a Fedosov dg Lie algebroid 
associated with the Lie pair $(L,A)$ --- see Corollary~\ref{Bari}.
\end{remark}

\subsection{Hochschild--Kostant--Rosenberg isomorphism}\label{section:HKR}

The natural inclusion $\sections{L/A}\hookrightarrow\Dpoly{0}$
extends to a morphism of complexes of $\enveloping{A}$-modules
\[ \hkr:\Tpoly{\bullet}\to\Dpoly{\bullet} \] by skew-symmetrization:
\[ \hkr(b_1\wedge\cdots\wedge b_n)=\frac{1}{n!}\sum_{\sigma\in S_n}\sgn(\sigma)
b_{\sigma(1)}\otimes b_{\sigma(2)}\otimes\cdots\otimes b_{\sigma(n)},
\qquad\forall b_1,\cdots,b_n\in\sections{L/A} .\]
Furthermore, $\hkr$ induces a morphism of double complexes:

\begin{lemma}
The map 
\[ \id \otimes \hkr: \big(\sections{\Lambda^\bullet A\dual} \otimes_R \Tpoly{\bullet}, d_A^{\Bott}, 0 \big) 
\to \big( \sections{\Lambda^\bullet A\dual} \otimes_R \Dpoly{\bullet}, \dau , \pm \dHH \big) \] 
is a morphism of double complexes from~\eqref{double cx: CE Tpoly} to~\eqref{double cx: CE Dpoly}.
\end{lemma}

The induced map between total cohomologies is called \emph{Hochschild--Kostant--Rosenberg} map.
Abusing notations, we will denote it $\hkr$ instead of $\id\otimes \hkr$. 

\begin{proposition}\label{thm:HKR}
For any Lie pair $(L, A)$, the Hochschild--Kostant--Rosenberg map
\[ \hkr: \hypercohomology^\bullet_{\CE}(A,\Tpoly{\bullet})\to \hypercohomology^\bullet_{\CE}(A,\Dpoly{\bullet}) \] 
is an isomorphism.
\end{proposition}

Proposition~\ref{thm:HKR} can be proved by a spectral sequence argument. 
Repeating the argument in \cite[Theorem~4.10]{MR2062626}, one can prove the following lemma:

\begin{lemma}\label{lem:HKR}
For each $p$, the map
\[ \id\otimes\hkr: \big(\sections{\Lambda^p A\dual} \otimes_R \Tpoly{\bullet}, 0 \big) 
\to \big( \sections{\Lambda^p A\dual} \otimes_R \Dpoly{\bullet}, (-1)^p \dHH \big) \] 
is a quasi-isomorphism.
\end{lemma}

\begin{proof}[Proof of Proposition~\ref{thm:HKR}]
Consider the spectral sequences associated with the filtrations 
\begin{gather*}
F^k\Big(\bigoplus_{\substack{p\geqslant 0 \\ q \geqslant -1}} \big(\sections{\Lambda^p A\dual} \otimes_R \Tpoly{q}) \Big) 
= \bigoplus_{\substack{p\geqslant k \\ q \geqslant -1}} \big(\sections{\Lambda^p A\dual} \otimes_R \Tpoly{q}) \\ 
F^k\Big(\bigoplus_{\substack{p \geqslant 0 \\ q \geqslant -1}} \big(\sections{\Lambda^p A\dual} \otimes_R \Dpoly{q}) \Big) 
= \bigoplus_{\substack{p \geqslant k \\ q \geqslant -1}} \big(\sections{\Lambda^p A\dual} \otimes_R \Dpoly{q})
\end{gather*} 
on the double complexes \eqref{double cx: CE Tpoly} and~\eqref{double cx: CE Dpoly}. 
The map induced by $\id\otimes\hkr$ between the $E_0$-terms of these two spectral sequences is precisely 
the quasi-isomorphism of Lemma~\ref{lem:HKR}. 
Therefore $\id\otimes\hkr$ induces isomorphism between the $E_k$-terms of the two spectral sequences for each $k$ larger than or equal to 1. 
Since both filtrations are complete and exhaustive, it follows from the Eilenberg--Moore comparison theorem 
that $\hkr: \hypercohomology^\bullet_{\CE}(A,\Tpoly{\bullet})\to \hypercohomology^\bullet_{\CE}(A,\Dpoly{\bullet})$ is an isomorphism. 
\end{proof}

\subsection{Atiyah and Todd cocycles/classes of a dg Lie algebroid}

A dg vector bundle \cite{MR2534186,MR3319134,MR3293862} is a vector bundle in
the category of dg manifolds. 
Given a vector bundle $\cE\xto{\pi}\cM$ of graded
manifolds, its space of sections, denoted
$\sections{\cE}$, is defined to be $\bigoplus_{j\in \ZZ}\Gamma (\cE)_j$,
where $\Gamma (\cE)_j$ consists of sections of degree $j$, i.e.\ maps 
$s\in \Hom (\cM, \cE[-j])$ such that $(\pi[-j])\circ s=\id_\cM$. 
Here $\pi[-j]: \cE[-j]\to \cM$ is the natural map induced by $\pi$; 
see~\cite{MR2534186} for more details.
When $\cE\to \cM$ is a dg vector bundle, the homological
vector fields $Q^{\cE}$ and $Q^{\cM}$ on $\cE$ and $\cM$ naturally determine
an operator $\mathcal{Q}$ of degree $+1$ on $\sections{\cE}$, 
making $\sections{\cE}$ a dg module over the dg algebra $C^\infty(\cM)$, 
i.e.\ $\mathcal{Q}$ satisfies 
\[ \mathcal{Q}(f\cdot e)=Q^{\cM}(f)\cdot e+f\cdot\mathcal{Q}(e), 
\quad\forall f\in C^\infty(\cM), e\in\sections{\cE} .\] 
Indeed, $\cE\xto{\pi}\cM$ is a dg vector bundle if and only if 
``the vector field $Q^{\cE}$ projects onto $Q^{\cM}$'' i.e.\ 
\[ Q^{\cE}\big(\pi^*(f)\big)=\pi^*\big(Q^{\cM}(f)\big), \quad \forall f\in C^{\infty}(\cM) \] 
and ``the flow of $Q^{\cE}$ preserves the linear structure of the fibers of $\pi$'' i.e.\ 
the submodule $\sections{\cE\dual}$ of $C^{\infty}(\cE)$ comprised of all 
smooth functions on $\cE$ ``linear along the fibers of $\pi$'' is stable under the derivation $Q^{\cE}$. 
The restriction of $Q^{\cE}$ to $\sections{\cE\dual}$ determines an operator $\mathcal{Q}$ on $\sections{\cE}$ 
through the relation 
\[ Q^{\cM}\big(\duality{\zeta}{e}\big) = 
\duality{Q^{\cE}(\zeta)}{e}+(-1)^{\abs{\zeta}}\duality{\zeta}{\mathcal{Q}(e)} 
,\quad \forall \zeta\in\sections{\cE\dual},e\in\sections{\cE} .\] 
Since $\pi^*\big(C^{\infty}(\cM)\big)$ and $\sections{\cE\dual}$ together generate 
the algebra $C^{\infty}(\cE)$ multiplicatively, knowledge of the vector field $Q^{\cM}$ 
and the operator $\mathcal{Q}$ suffices to recover the homological vector field $Q^{\cE}$. 

The degree $+1$ operator 
$\mathcal{Q}$ on $\Gamma(\cE)$ gives rise to a cochain complex
\[ \cdots\to\Gamma (\cE)_i \xto{\mathcal{Q}}\Gamma (\cE)_{i+1}\to\cdots ,\]
whose cohomology group will be denoted 
by $H^\bullet(\Gamma(\cE),\mathcal{Q})$.

A dg Lie algebroid is a Lie algebroid object
in the category of dg manifolds. For more details, we refer the reader 
to~\cite{MR2534186,MR2709144}, where dg Lie algebroids are called $Q$-algebroids.
It is simple to see that if $\cM$ is a dg manifold, then $T_\cM$ is
naturally a dg Lie algebroid.

The notion of Atiyah class of dg Lie algebroids was introduced
and studied by Mehta--Stiénon--Xu \cite{MR3319134}. It extends
the notion of Atiyah class of a dg manifold, which 
was first investigated by Shoikhet~\cite{arXiv:math/9812009} in relation with 
Kontsevich's formality theorem and Duflo's formula.

Let $\cA$ be a dg Lie algebroid with anchor $\rho:\cA\to T_\cM$.
An $\cA$-connection on $\cA$ is a map
\[ \nabla:\sections{\cA}\times\sections{\cA}\to\sections{\cA} \] 
of degree $0$ satisfying
\begin{gather*}
\nabla_{fX} Y = f \nabla_X Y \\
\nabla_X (fY) = \big(\rho(X) f\big) Y + (-1)^{|X||f|} f \nabla_X Y
\end{gather*}
for all $f\in C^\infty(\cM)$ and all (homogeneous) $X, Y\in\sections{\cA}$. 
The notation $|\argument|$ is used to denote the degree of the argument. 
Here the degree of $\nabla$ is its degree as a map from $\sections{\cA} \otimes \sections{\cA}$ to $\sections{\cA}$. 
More precisely, by saying that $\nabla$ is a map of degree $0$ we mean that 
$|\nabla_X Y| = |X| + |Y|$ for every pair of homogeneous elements $X$ and $Y$.
Connections always exist since the standard partition of unity argument 
holds in the context of graded manifolds.

Given a dg Lie algebroid $\cA$, the associated operator $\mathcal{Q}$ of 
degree $+1$ on $\sections{\cA}$, 
and an $\cA$-connection $\nabla$ on $\cA$,
one defines a bundle map $\At^\nabla_\cA:\cA\otimes\cA\to\cA$ of degree $+1$ by
\[ \At^\nabla_\cA (X,Y) := \mathcal{Q}(\nabla_X Y) -\nabla_{\mathcal{Q}(X)}Y -(-1)^{|X|}\nabla_X\big(\mathcal{Q}(Y)\big),
\quad \forall X, Y\in\sections{\cA} .\]
Alternatively, we may think of $\At^\nabla_\cA$ as a section of degree $+1$ in $\sections{\cA^\vee\otimes\End\cA}$.
It is immediate that $\mathcal{Q}(\At^\nabla_\cA)=0$. 
Since $\mathcal{Q}^2=0$, we may thus regard $\At^\nabla_\cA$ as a 1-cocycle in 
the cochain complex $\big(\sections{\cA^\vee\otimes\End\cA}_\bullet, \mathcal{Q}\big)$. 

\begin{definition}
The 1-cocycle $\At^\nabla_\cA\in Z^1(\sections{\cA^\vee\otimes\End\cA}, \mathcal{Q})$
is called the Atiyah 1-cocycle of the dg 
Lie algebroid $\cA$ with respect to the
$\cA$-connection $\nabla$ on $\cA$.
\end{definition}

It is simple to check that its cohomology class 
$\atiyahclass_\cA := [\At^\nabla_\cA] \in H^1\big(\Gamma(\cA^\vee\otimes\End\cA),\mathcal{Q}\big)$ 
is independent of the choice of the connection $\nabla$. 
The class $\atiyahclass_\cA$ is called the \emph{Atiyah class} of the dg Lie algebroid $\cA$. 
It is the obstruction class to the existence
of a dg compatible $\cA$-connection on $\cA$. 
(See \cite{MR3319134} for more details.)

The \emph{Todd cocycle} $\todd^\nabla_\cA$ (or $\ttodd^\nabla_\cA$)
and \emph{Todd class} $\Todd_\cA$ (or $\tTodd_\cA$) of a dg Lie algebroid $\cA$ 
are defined as follows: 
\begin{gather*}
\todd^\nabla_\cA:=\Ber\left(\frac{\At^\nabla_\cA}{1-e^{-\At^\nabla_\cA}}\right)
\in\prod_{k\geqslant 0} \sections{\Lambda^k\cA\dual}_k ,\\
\ttodd^\nabla_\cA:=\Ber\left(\frac{\At^\nabla_\cA}{e^{\half\At^\nabla_\cA}-e^{-\half \At^\nabla_\cA}}\right)
\in\prod_{k\geqslant 0} \sections{\Lambda^k\cA\dual}_k\\
\Todd_\cA:=\Ber\left(\frac{\atiyahclass_\cA}{1-e^{-\atiyahclass_\cA}}\right)
\in\prod_{k\geqslant 0} H^k\big(\sections{\Lambda^k\cA\dual},\mathcal Q\big) ,\\
\tTodd_\cA:=\Ber\left(\frac{\atiyahclass_\cA}{e^{\half\atiyahclass_\cA}-e^{-\half \atiyahclass_\cA}}\right)
\in\prod_{k\geqslant 0} H^k\big(\sections{\Lambda^k\cA\dual},\mathcal Q\big) 
,\end{gather*}
where $\Lambda^k \cA^\vee$ denotes the dg vector bundle $S^k (\cA^\vee[-1])[k]\to\cM$. 
The definition of the Berezinian $\Ber$ can be found in~\cite{MR1701597}. 
It is known that $\Todd_{\cA}$ and $\tTodd_\cA$ can be expressed in terms of the scalar Atiyah classes 
$c_k:=\frac{1}{k!}(\frac{i}{2\pi})^k \supertrace \atiyahclass_\cA^k \in H^k\big(\sections{\Lambda^k\cA\dual},\mathcal{Q}\big)$. 
Here $\supertrace: \End \cA\to C^\infty (\cM)$ denotes the supertrace 
and $\supertrace(\At^\nabla_\cA)^k\in\sections{\Lambda^k\cA\dual}$ 
since $(\At^\nabla_\cA)^k\in\sections{\Lambda^k\cA\dual}\otimes_{C^\infty(\cM)}\End\cA$.

\begin{example}\cite[Example~3.4]{MR3319134}
\label{ex:trivial cntn}
Consider the tangent dg Lie algebroid $T_\cM$
of a dg manifold $\cM=(\cV,Q)$, where $\cV=\RR^m\times\vf$ 
for some finite dimensional $\ZZ$-graded vector space $V$ over $\KK$.
The algebra of functions on the graded manifold $\cV$ is $C^\infty (\cV)=C^\infty (\RR^m )\otimes \hat{S}(V\dual)$. 
Let $(z_1,\cdots,z_N)$ be a choice of coordinate functions on $\cV$. 
Writing the homological vector field $Q$ as $Q=\sum_k Q_k\frac{\partial}{\partial z_k}$, 
the Atiyah 1-cocycle associated with the trivial connection 
$\nabla^{\trivial}_{\frac{\partial}{\partial z_i}}\frac{\partial}{\partial z_j}=0$ 
admits the simple expression 
\[ \At^{\trivial}_{T_\cV}\left(\frac{\partial}{\partial z_i},\frac{\partial}{\partial z_j}\right)
=(-1)^{|z_i|+|z_j|}\sum_k\frac{\partial^2 Q_k}{\partial z_i\partial z_j}
\frac{\partial}{\partial z_k} .\]
Hence the Atiyah 1-cocycle $\At^{\trivial}_{T_\cV}$ captures the second- and higher-order information 
contained in the homological vector field. 
\end{example}

In this paper, we are particularly interested in
an important class of dg Lie algebroids, namely
the Fedosov dg Lie algebroids associated with a Lie pair $(L, A)$. 
See the Appendix or \cite{arXiv:1901.04602} for more details.

\subsection{Atiyah and Todd cocycles/classes of the Fedosov dg Lie algebroid}
\label{Angers}

Let $(L,A)$ be a Lie pair over a smooth manifold $M$. 
Each choice of, firstly, a splitting of the short exact sequence of vector bundles $0\to A\to L\to B\to 0$ and, secondly, a torsion-free $L$-connection on $B$ determines 
a homological vector field $Q$ on the graded manifold $\cM=L[1]\oplus B$ --- see Theorem~\ref{strawberry} in the Appendix. 
Any such dg manifold $(\cM,Q)$ is called a Fedosov dg manifold associated with the Lie pair $(L,A)$. 
The pullback $\cF\to\cM$ of the quotient bundle $B\to M$ through the canonical projection $\cM\to M$ is a dg Lie subalgebroid of the tangent dg Lie algebroid 
$T_{\cM}\to\cM$ --- see Proposition~\ref{pro:Rome} in the Appendix. 
Any such dg Lie algebroid $\cF\to\cM$ is called a Fedosov dg Lie algebroid associated with the Lie pair $(L,A)$. 

Since $\cF\to\cM$ is both the pullback of the vector bundle $B\to M$ through the canonical map $\cM\onto M$ 
and a vector subbundle of $T_{\cM}\to\cM$, we have the inclusions 
\[ \begin{tikzcd}[row sep=small] 
\sections{B} \arrow[hook]{r} & C^\infty(\cM)\otimes_{C^\infty(M)}\sections{B} \arrow[leftrightarrow]{r}{\cong}
& \sections{\cF\to\cM} \arrow[hook]{r} & \XX(\cM) \\ 
b \arrow[mapsto]{r} & 1\otimes b \arrow[mapsto]{rr} & & \hat{b} 
\end{tikzcd} \]
Indeed, $\sections{\cF\to\cM}$ is the $C^\infty(\cM)$-submodule of $\XX(\cM)$ generated by $\sections{B}$. 
In particular, if $\partial_1,\dots,\partial_r$ is a local frame for $B$ and $\chi_1,\dots,\chi_r$ is the dual local frame for $B\dual$, 
then $\hat{\partial}_k$ is the vector field $\frac{\partial}{\partial\chi_k}$ on $\cM$, 
i.e.\ the derivation $\lambda\otimes\chi^M\mapsto\lambda\otimes M_k\chi^{M-e_k}$ 
of $C^\infty(\cM)\cong\sections{\Lambda L\dual\otimes\hat{S}(B\dual)}$.

There exists a \emph{canonical} $\cF$-connection on $\cF$ characterized by the relation 
\[ \nablacan_{\hat{b}}\hat{c}=0 ,\quad\forall b,c\in\sections{B} .\] 

\begin{definition}
The Atiyah 1-cocycle $\At^{\can}_\cF\in Z^1(\sections{\cM;\cF^\vee\otimes\End\cF}, \mathcal{Q})$ 
corresponding to the canonical connection $\nablacan$ is called the canonical Atiyah 1-cocycle.
\end{definition}

Since $\Gamma ({\cM;\cF^\vee\otimes\End\cF})$ can be identified canonically 
with $\sections{\Lambda^\bullet L^\vee \otimes \hat{S}B^\vee \otimes B^\vee \otimes \End B}$,
the canonical Atiyah cocycle is essentially the tensor product of an endomorphism of $B$ 
and an element of $\sections{\Lambda^1 L^\vee \otimes \hat{S}B^\vee \otimes \Lambda^1 B^\vee}$.
Therefore, the Berezinian appearing in the expression for the Todd cocycle of $\nabla^{\can}$ is simply the classical determinant
and the \emph{canonical Todd cocycle} is
\begin{equation}\label{corrected Todd}
\todd^{\can}_\cF := \det\left(\frac{\At^{\can}_\cF}{1-e^{- \At^{\can}_\cF}}\right)
\in\prod_{k\geqslant 0} \sections{\Lambda^k L\dual \otimes \hat{S}B\dual \otimes \Lambda^k B\dual}
.\end{equation}
Similarly,
\begin{equation}\label{twisted Todd}
\ttodd^{\can}_\cF := \det\left(\frac{\At^{\can}_\cF}{e^{\half\At_\cF^{\can}}-e^{-\half \At^{\can}_\cF}}\right)
\in\prod_{k\geqslant 0} \sections{\Lambda^k L\dual \otimes \hat{S}B\dual \otimes \Lambda^k B\dual}
.\end{equation}

\begin{lemma}\label{lem:local Atiyah}
Given any local frame $\partial_1,\cdots,\partial_r$ for $B$, 
the canonical Atiyah 1-cocycle $\At^{\can}_{\cF}:\cF\otimes\cF\to\cF$ of the Fedosov dg Lie algebroid $\cF\to\cM$ 
admits the local expression 
\begin{equation}\label{eq:local Atiyah matrix}
\At^{\can}_{\cF}(\hat{\partial}_i,\hat{\partial}_j)
=\sum_{k=1}^r\hat{\partial}_i(\hat{\partial}_j f_k)\cdot\hat{\partial}_k 
,\end{equation} 
where the functions $f_k\in C^\infty(\cM)$ are the components of the vector field $X^\nabla=\sum_{k=1}^r f_k\cdot\hat{\partial}_k$ 
of Theorem~\ref{strawberry} relative to the chosen frame. 
\end{lemma}

\begin{proof}
Recall that the coboundary operator $\mathcal{Q}$ on the sections of the Fedosov dg Lie algebroid is the restriction 
of the Lie derivative $\mathcal{Q}= \fedosova$ in $\XX(\cM)$ with the homological vector field $Q\in\XX(\cM)$ appearing in Theorem~\ref{strawberry}. 
We have
\begin{align*}
\At^{\can}_\cF (\hhb{i}, \hhb{j}) &=
 \mathcal{Q}(\nablacan_{\hhb{i}} \hhb{j}) -\nablacan_{\mathcal{Q}(\hhb{i})}\hhb{j} -(-1)^{0}\nablacan_{\hhb{i}}\big(\mathcal{Q}(\hhb{j})\big) \\
&= -\nablacan_{\hhb{i}}\big(\mathcal{Q}(\hhb{j})\big)
.\end{align*}
Now $\mathcal{Q}(\hhb{j})=[Q, \hhb{j}]=-[\delta, \hhb{j}] + [d_L^{\nabla}, \hhb{j}]+[X^\nabla,\hhb{j}]$ and it is easy to show that $\lie{\delta}{\hat{\partial}_j}=0$. 
Therefore, \[ \At^{\can}_\cF (\hhb{i}, \hhb{j}) = - \nabla_{\hat{\partial}_i}^{\can}(\lie{d_L^\nabla}{\hat{\partial}_j}) - \nabla_{\hat{\partial}_i}^{\can}(\lie{X^\nabla}{\hat{\partial}_j}) .\] 
Given local coordinates $(x_1,\dots,x_m)$ on the base manifold $M$; 
a local frame $\eta_1,\dots,\eta_l$ for $L$; the dual local frame $\lambda_1,\dots,\lambda_l$ for $L\dual$; 
the local frame $\partial_1,\dots,\partial_r$ for $B$; and the dual local frame $\chi_1,\dots,\chi_r$ for $B\dual$, 
the vector field $d_L^\nabla$ on $\cM$ decomposes as the sum 
\[ d_L^\nabla= \sum_{j} a_i^j\lambda_i\frac{\partial}{\partial x_j} -\frac{1}{2}\sum_{i,j,k}c_{ij}^k\lambda_i\lambda_j\frac{\partial}{\partial \lambda_k} 
-\sum_{i,j,k}\Gamma_{ij}^k \lambda_i\chi_j\frac{\partial}{\partial\chi_k} ,\] 
where $a_i^j = \duality{dx_j}{\rho(\eta_i)}$, $c_{ij}^k=\duality{\lambda_k}{\lie{\eta_i}{\eta_j}}$, and $\Gamma_{ij}^k=\duality{\chi_k}{\nabla_{\eta_i}\partial_j}$
are functions in $C^\infty(M)$ encoding the anchor and the Lie bracket of the Lie algebroid $L$ and the $L$-connection on $B$. 
Therefore, since $\hat{\partial}_j=\frac{\partial}{\partial\chi_j}$, we have $\lie{d_L^\nabla}{\hat{\partial}_j}=\sum_{ik} \Gamma_{ij}^k \lambda_i \cdot\hat{\partial}_k$ 
and thus $\nabla_{\hat{\partial}_i}^{\can}(\lie{d_L^\nabla}{\hat{\partial}_j})=0$. 
Likewise, we have $\lie{X^\nabla}{\hat{\partial}_j}=-\sum_k\hat{\partial}_j(f_k)\cdot\hat{\partial}_k$ and thus 
$\nabla_{\hat{\partial}_i}^{\can}(\lie{X^\nabla}{\hat{\partial}_j})=-\sum_k\hat{\partial}_i(\hat{\partial}_j f_k)\cdot\hat{\partial}_k$.
\end{proof}

\begin{proposition}\label{sigma-atiyah-todd}
The quasi-isomorphism 
\[ \etendu{\sigma}: \Big(\sections{\Lambda^\bullet L\dual}\otimes_R\verticalTpoly{r,s}, \liederivative{Q}\Big) 
\to \Big(\sections{\Lambda^\bullet A\dual}\otimes_R\Tpoly{r,s},\dabott\Big) \] 
of Proposition~\ref{sigma for general tensor} takes 
$\At^{\can}_\cF$, $\todd^{\can}_\cF$, and $\ttodd^{\can}_\cF$ 
to $\atiyahcocycle$, $\todd^\nabla_{L/A}$, and $\ttodd^\nabla_{L/A}$, respectively: 
\[ \etendu{\sigma}(\At^{\can}_\cF)=\atiyahcocycle ,\qquad \qquad
\etendu{\sigma}(\todd^{\can}_\cF)=\todd^\nabla_{L/A} ,\qquad \qquad
\etendu{\sigma} (\ttodd^{\can}_\cF)=\ttodd^\nabla_{L/A} .\] 
\end{proposition} 

\begin{proof}
It suffices to prove the first statement; the second
and third statements are immediate consequences of the first. 
Lemma~\ref{lem:local Atiyah}, together with the natural identification 
of $\Gamma ({\cM;\cF^\vee\otimes\End\cF})_1$ with 
$C^\infty(\cM)\otimes_{C^\infty(M)}\sections{B^\vee\otimes B^\vee\otimes B}$, implies that 
\[ \At^{\can}_\cF = \sum^r_{i,j,k=1} \hhb{i}(\hhb{j} f_k) \otimes (\chi_i \otimes \chi_j \otimes \partial_k) ,\]
where $\hhb{i}(\hhb{j} f_k) \in C^\infty(\cM) = \sections{L\dual \otimes \hat{S}(B\dual)}$ 
and $\chi_i \otimes \chi_j \otimes \partial_k \in \sections{B\dual \otimes B\dual \otimes B}$. 

According to Theorem~\ref{strawberry}, we have $X^\nabla = \sum_{t=2}^\infty X_t$ 
with $X_t=\sum_{k=1}^r f^{(t)}_k\otimes \partial_k$ 
and $f^{(t)}_k\in\sections{L\dual\otimes S^t(B\dual)}$. 
Moreover, the term $X_2$ decomposes as the sum of 
$\etendu{h}( \widetilde{\atiyahcocycle})=\sum_k f_k^{(1,1)} \otimes \partial_k$ 
and $\etendu{h}(R^\nabla_{0,2})=\sum_k f_k^{(0,2)}\otimes\partial_k$, 
where $f_k^{(1,1)} \in \sections{p^\top(A\dual)\otimes S^2 B\dual}$ 
and $f_k^{(0,2)} \in \sections{q^\top(B\dual)\otimes S^2 B\dual}$.

Since $\hhb{i} (\hhb{j} f_k^{(t)}) \in \sections{L\dual \otimes S^{t-2} B\dual}$, 
we have $\sigma\big( \hhb{i} (\hhb{j} f_k^{(t)}) \big) =0$ for $t\geqslant 3$. 
Since $\hhb{i} (\hhb{j} f_k^{(0,2)}) \in \sections{q^\top(B\dual)}$, 
we also have $\sigma\big( \hhb{i} (\hhb{j} f_k^{(0,2)}) \big) =0$. 

Therefore, we obtain 
\begin{align*} \etendu{\sigma}(\At^{\can}_\cF) 
= {}& \etendu{\sigma}\Big( \sum_{i,j,k=1}^r \hhb{i}(\hhb{j} f_k) \otimes (\chi_i \otimes \chi_j \otimes \partial_k) \Big) \\ 
= {}& \sum_{i,j,k=1}^r \sigma\Big( \sum_{t=2}^{\infty} \hhb{i}(\hhb{j} f^{(t)}_k) \Big) \otimes (\chi_i \otimes \chi_j \otimes \partial_k) \\ 
= {}& \sum_{i,j,k=1}^r \hhb{i}(\hhb{j} f^{(1,1)}_k) \otimes (\chi_i \otimes \chi_j \otimes \partial_k) \\ 
= {}& {2} \sum_{k=1}^r f^{(1,1)}_k \otimes \partial_k \\
= {}& {2} \etendu{h}( \widetilde{\atiyahcocycle})\\
= {}&\atiyahcocycle
.\qedhere \end{align*}
\end{proof}	

\begin{corollary}
\begin{enumerate}
\item 
The isomorphism 
$\etendu{\sigma}: H^1\big(\Gamma(\cM;\cF^\vee\otimes\End\cF), \mathcal{Q} \big) \to H_{\CE}^1(A, B\dual \otimes\End B)$ 
induced by the quasi-isomorphism of 
Proposition~\ref{sigma for general tensor} (for $r=2$ and $s=1$) 
takes the Atiyah class 
$\atiyahclass_{\cF}$ of the Fedosov dg Lie algebroid $\cF$ to the Atiyah class $\alpha_{L/A}$ of the Lie pair $(L,A)$: 
\[ \etendu{\sigma} ( \atiyahclass_\cF)=\alpha_{L/A} .\] 
\item
The isomorphism 
$\etendu{\sigma}: H^k\big(\sections{\cM;\Lambda^k\cF\dual}, \mathcal{Q}\big)
\to H^k_{\CE} (A, \Lambda^k B\dual)$
induced by the quasi-isomorphism of Proposition~\ref{sigma for general tensor} 
takes the Todd class of the Fedosov dg Lie algebroid $\cF$ 
to the Todd class of the Lie pair $(L,A)$: 
\[ \etendu{\sigma}( \Todd_\cF)=\Todd_{L/A}, \qquad \qquad 
\etendu{\sigma}( \tTodd_\cF)=\tTodd_{L/A} .\] 
\end{enumerate}
\end{corollary}

\section{Formality theorem for Lie pairs}
\label{Kabul}

\subsection{Statements of main theorems}

We are ready to state the main theorems of the paper.
Let $(L,A)$ be a Lie pair. 
According to Theorem~\ref{Bari2} and Theorem~\ref{thm:Naples2}, both
$\tot\big(\sections{\Lambda^\bullet A\dual}\otimes_R \Tpoly{\bullet}\big)$
and $\tot\big(\sections{\Lambda^\bullet A\dual}\otimes_R \Dpoly{\bullet}\big)$
are $L_\infty$ algebras with $d_A^{\Bott}$ and $\dau+\dHH$ as their respective unary brackets. 
Moreover, these $L_\infty$ algebra structures are canonical up to $L_\infty$ isomorphisms 
having the identity map as linear part. 
Furthermore, their cohomologies $\hypercohomology^\bullet_{\CE}(A,\Tpoly{\bullet})$ and
$\hypercohomology^\bullet_{\CE}(A,\Dpoly{\bullet})$ carry canonical Gerstenhaber algebra structures. 
The main result of the paper is the following

\begin{theorem}[Formality theorem for Lie pairs] 
\label{thm:main}
Let $(L,A)$ be a Lie pair. 
Endow the associated graded vector spaces 
$\tot\big(\sections{\Lambda^\bullet A^\vee}\otimes_R\Tpoly{\bullet}\big)$
and $\tot\big(\sections{\Lambda^\bullet A^\vee}\otimes_R\Dpoly{\bullet}\big)$ 
with their inherited $L_\infty$ algebras --- see Theorems~\ref{Bari2} and~\ref{thm:Naples2}. 
Then, there exists an $L_\infty$ quasi-isomorphism
\[ \Phii: \tot\big(\sections{\Lambda^\bullet A^\vee}\otimes_R\Tpoly{\bullet}\big)
\to \tot\big(\sections{\Lambda^\bullet A^\vee}\otimes_R\Dpoly{\bullet}\big) \]
with first Taylor coefficient 
$\Phii_1: \tot\big( \sections{\Lambda^\bullet A^\vee}\otimes_R \Tpoly{\bullet}\big)
\to \tot\big( \sections{\Lambda^\bullet A^\vee}\otimes_R \Dpoly{\bullet}\big)$
satisfying the following two properties: 
\begin{enumerate}
\item $\Phii_1$ preserves the associative algebra structures (wedge and cup product, respectively) up to homotopy;
\item $\Phii_1 = \hkr\circ(\todd_{L/A}^{\nabla})^{\frac{1}{2}}$, where 
$(\todd_{L/A}^{\nabla})^{\frac{1}{2}}\in\bigoplus_{k=0}\sections{\Lambda^kA\dual\otimes\Lambda^k A^\perp}$
acts on $\tot\big( \sections{\Lambda^\bullet A^\vee}\otimes_R \Tpoly{\bullet}\big)$ by contraction.
\end{enumerate}
\end{theorem}

As an immediate consequence, we have the following

\begin{theorem}[Kontsevich-Duflo type theorem for Lie pairs]
\label{KD-thm}
Given a Lie pair $(L,A)$, the map 
\[ \hkr\circ\Todd_{L/A}^{\frac{1}{2}}: \hypercohomology^\bullet_{\CE}(A,\Tpoly{\bullet}) \to \hypercohomology^\bullet_{\CE}(A,\Dpoly{\bullet}) \]
is an isomorphism of Gerstenhaber algebras --- the square root of the Todd class 
\[ \Todd_{L/A}^{\frac{1}{2}} \in\bigoplus_{k=0} H_{\CE}^{k}(A,\Lambda^k A^\perp) \] 
acts on $\hypercohomology^\bullet_{\CE}(A,\Tpoly{\bullet})$ by contraction. 
\end{theorem}

To prove Theorem~\ref{thm:main}, we make use of a Fedosov dg Lie algebroid $\cF\to\cM$ 
associated with the Lie pair $(L,A)$ --- see~\cite{arXiv:1901.04602} and the Appendix for details. 
As recalled earlier in Section~\ref{Angers}, a Fedosov dg Lie algebroid $\cF\to\cM$
encodes a dg foliation of the Fedosov dg manifold $(\cM,Q)$ associated with the Lie pair $(L,A)$

Since a Lie algebroid can be thought of as an extension of the
tangent bundle of a manifold, the notions of polyvector fields 
and polydifferential operators admit extensions to the context of a Lie algebroid
and these each carry a natural dgla structure \cite{MR1675117,MR1815717}.
Likewise, the notions of polyvector fields and polydifferential operators can
be extended in a appropriate sense to the context of a dg Lie algebroid.
The ``polyvector fields'' and ``polydifferential operators'' associated to
a Fedosov dg Lie algebroid $\cF$
can be viewed geometrically as polyvector fields and polydifferential operators
tangent to the dg foliation on the Fedosov dg manifold $(\cM,Q)$.
In fact, one can identify the dglas of ``polyvector fields'' and of ``polydifferential operators'' on $\cF$
to $\big(\tot(\sections{\Lambda^\bullet L^\vee}\otimes_R\verticalTpoly{\bullet}),\schouten{Q}{\argument},\schouten{\argument}{\argument}\big)$
and $\big(\tot(\sections{\Lambda^\bullet L^\vee}\otimes_R\verticalDpoly{\bullet}),\gerstenhaber{Q+m}{\argument},\gerstenhaber{\argument}{\argument}\big)$, respectively,
where $\verticalTpoly{\bullet}$ denotes the formal polyvector fields and $\verticalDpoly{\bullet}$
the formal polydifferential operators tangent to the fibers of the vector bundle $L/A\to M$.

In fact, according to Corollary~\ref{Bari} and Corollary~\ref{thm:Naples}, 
the $L_\infty$ structures on $\tot\big( \sections{\Lambda^\bullet A^\vee}\otimes_R \Tpoly{\bullet}\big)$
and $\tot\big( \sections{\Lambda^\bullet A^\vee}\otimes_R \Dpoly{\bullet}\big)$
are indeed obtained by the homotopy transfer from the dgla structures on 
$\big(\tot(\sections{\Lambda^\bullet L^\vee}\otimes_R\verticalTpoly{\bullet}),\schouten{Q}{\argument},\schouten{\argument}{\argument}\big)$
and $\big(\tot(\sections{\Lambda^\bullet L^\vee}\otimes_R\verticalDpoly{\bullet}),\gerstenhaber{Q+m}{\argument},\gerstenhaber{\argument}{\argument}\big)$, 
respectively (see \cite{arXiv:1901.04602}). 
Therefore, as a key step, we apply Kontsevich formality theorem
to the Fedosov dg Lie algebroid $\cF$ and establish the following

\begin{theorem}
\label{thm:main0}
There exists an $L_\infty$ quasi-isomorphism
\[ \Phia : \big(\tot(\sections{\Lambda^\bullet L^\vee}\otimes_R\verticalTpoly{\bullet}),\schouten{Q}{\argument},\schouten{\argument}{\argument}\big)
\to \big(\tot(\sections{\Lambda^\bullet L^\vee}\otimes_R\verticalDpoly{\bullet}),\gerstenhaber{Q+m}{\argument},\gerstenhaber{\argument}{\argument}\big) \] %#liao112
from the dgla of ``polyvector fields'' on $\cF$ to the dgla of ``polydifferential operators'' on $\cF$ 
with first Taylor coefficient 
$\Phia_1: \tot(\sections{\Lambda^\bullet L^\vee}\otimes_R\verticalTpoly{\bullet})
\to \tot(\sections{\Lambda^\bullet L^\vee}\otimes_R\verticalDpoly{\bullet})$
satisfying the following two properties: 
\begin{enumerate}
\item $\Phia_1$ preserves the associative algebra structures (wedge and cup product, respectively) up to homotopy;
\item 
$\Phia_1=\hkr\circ(\todd_\cF^{\can})^{\frac{1}{2}}$, where $\hkr$ denotes the natural extension of the fiberwise Hochschild--Kostant--Rosenberg map
$\verticalTpoly{\bullet}\to\verticalDpoly{\bullet}$ and $(\todd_\cF^{\can})^{\frac{1}{2}}$ the action 
of the square root of the canonical Todd cocycle $\todd^{\can}_\cF$ of the Fedosov dg Lie algebroid $\cF$ 
on $\tot(\sections{\Lambda^\bullet L^\vee}\otimes_R\verticalTpoly{\bullet})$ by contraction.
\end{enumerate}
\end{theorem}

\subsection{Kontsevich formality morphism for Lie pairs}

\subsubsection{Tangent $L_\infty$ algebras}

Let $\frakg$ and $\frakg'$ be two $L_\infty$ algebras \cite{Manetti:book,MR1235010,MR1183483,MR2440258} and let 
$Q$ and $Q'$ denote the corresponding 
homological vector fields on the 
associated dg manifolds $\frakg[1]$ and $\frakg'[1]$, respectively. 
An $L_\infty$ morphism $\kont:\frakg\to\frakg'$ is, by definition, a morphism of dg manifolds $\frakg[1] \to \frakg'[1]$, 
which means that the homomorphism of algebras 
$\kont^*:C^\infty(\frakg'[1])\to C^\infty(\frakg[1])$ 
intertwines the derivations: 
$\kont^*\circ Q' = Q\circ\kont^*$. 
Such an $L_\infty$ morphism $\kont$ is entirely determined by its so-called `Taylor coefficients,' 
which are a sequence $(\kont_n)_{n=1,2,\cdots}$ of morphisms of graded vector spaces \[ \kont_n: \Lambda^n \frakg \to \frakg'[1-n] .\]

A Maurer--Cartan (MC) element of an $L_\infty$ algebra $(\frakg,Q)$ 
is an element $\omega\in\frakg_1$ (of degree 1) satisfying 
\begin{equation}\label{eq:Paris}
\sum_{j=1}^\infty\frac{1}{j!}Q_j(\omega^j)=0 ,
\end{equation}
where $Q_j:\Lambda^j\frakg\to\frakg$ is the $j$-th bracket 
and $\omega^j=\omega\wedge\cdots\wedge\omega\in\Lambda^j\frakg$. 
In particular, the MC elements $\omega$ of a dgla satisfy the classical Maurer--Cartan equation
\[ d\omega+\half [\omega,\omega] =0 .\]
Given a MC element $\omega$ of an $L_\infty$ structure $Q$ on a graded vector space $\frakg$, 
there is a new $L_\infty$ algebra structure $Q_{\omega}$ on $\frakg$ called tangent $L_\infty$ algebra \cite{MR2062626}: 
the Taylor coefficients of $Q_{\omega}$ satisfy 
\begin{equation}\label{eq:Athens}
(Q_{\omega})_n (\gamma) = \sum_{j=0}^\infty \frac{1}{j!} Q_{n+j}(\omega^j\wedge\gamma) 
,\quad\forall\gamma\in\Lambda^n\frakg 
.\end{equation}
In general, the convergence of the summations in Equations~\eqref{eq:Paris} and~\eqref{eq:Athens} is an issue that has to be addressed. 
For instance, the summations converge for $L_\infty$ algebras equipped with complete (descending) filtrations --- see \cite{MR3323983} for more details on this issue.
However, if $\frakg$ is a dgla with differential $d$ and Lie bracket $[-,-]$, 
the sums are finite and the tangent $L_\infty$ algebra is again a dgla with the same bracket 
but with the modified differential $d_\omega = d+[\omega,-]$. 

We use the symbol $T_\omega\frakg$ to distinguish the tangent $L_\infty$ algebra at $\omega$ from the original $L_\infty$ algebra $\frakg$. 

Given an $L_\infty$ morphism of dglas $\kont:\frakg\to\frakg'$ and a MC element $\omega$ of $\frakg$, 
consider the element $\kont(\omega)$ of $\frakg'$ defined by 
\begin{equation}\label{Taipei}
\kont(\omega) = \sum_{j=1}^\infty \frac{1}{j!} \kont_j(\omega^j)
\end{equation}
assuming the summation converges. 
Then $\kont(\omega)$ is a Maurer--Cartan element of $\frakg'$ and therefore both $T_\omega\frakg$ and $T_{\kont(\omega)}\frakg'$ are dglas. 
There is a tangent $L_\infty$ morphism 
\[ \kont_\omega: T_\omega\frakg \to T_{\kont(\omega)}\frakg' \]
defined through the relations 
\begin{equation}\label{tangent}
(\kont_{\omega})_n(\gamma) = \sum_{j=0}^\infty \frac{1}{j!} \kont_{n+j}(\omega^j \wedge\gamma)
,\quad \forall \gamma \in \Lambda^n \frakg
.\end{equation}
Provided the summations in the r.h.s.\ of Equation~\eqref{tangent} converge, $\kont_\omega$ is a well defined $L_\infty$ morphism
--- see~\cite{MR3318161,MR3323983,MR3669170,MR2062626,MR2265541}) for details on the issue of convergence.

\subsubsection{Kontsevich formality morphism for $\KK^d$}

In this section, we briefly recall the definition of Kontsevich's formality morphism for $\KK^d$ 
(where $\KK$ is either $\RR$ or $\CC$), which we need later on. 
For more details, the reader may want to refer to Kontsevich's original paper \cite{MR2062626}.

Kontsevich's formality morphism is an $L_\infty$ quasi-isomorphism 
\[ \kont: \Tpoly{\bullet}(\KK^d) \to \Dpoly{\bullet}(\KK^d) \] 
between the two dglas $\big(\Tpoly{\bullet}(\KK^d),0,\schouten{\argument}{\argument}\big)$ 
and $\big(\Dpoly{\bullet}(\KK^d),\gerstenhaber{m}{\argument},\gerstenhaber{\argument}{\argument}\big)$.
Its `Taylor coefficients' are of the form
\begin{equation}\label{def:Kontsevich} 
\kont_n = \sum_{m \geqslant 0} \sum_{\Gamma \in \graphs_{n,m}} W_\Gamma \kont_\Gamma 
,\end{equation}
where $\graphs_{n,m}$ denotes the set of admissible graphs of type $(n,m)$, $W_\Gamma$ is a number called Kontsevich weight of the graph $\Gamma$, and 
$\kont_\Gamma$ is a map which assembles $n$ polyvector fields 
into a single polydifferential operator in a way determined by the graph $\Gamma$. 

We will now describe $\graphs_{n,m}$, $\kont_\Gamma$, and $W_\Gamma$ successively.

\subsubsection*{Admissible graphs}

A directed graph $\Gamma$ is a pair of (finite) sets $V_\Gamma$ and $E_\Gamma$ together with two maps $s,\; t: E_\Gamma\to V_\Gamma$. 
The elements of $V_\Gamma$ are called vertices. The elements of $E_\Gamma$ are called edges. 
Each edge $e \in E_\Gamma$ starts at its source $s(e)\in V_\Gamma$ and ends at its target $t(e)\in V_\Gamma$. 
Given a vertex $v \in V_\Gamma$, we use the symbol $\outedge{v}$ to denote the set $s\inv(v)$ of all edges starting at $v$ 
and we use the symbol $\inedge{v}$ to denote the set $t\inv(v)$ of all edges ending at $v$. 

An admissible graph of type $(n,m)$ is a directed graph $\Gamma=(V_\Gamma, E_\Gamma)$ 
with labels on its vertices and edges satisfying the following requirements.
\begin{enumerate}
\item The set of vertices is partitioned into two subsets: $V_\Gamma = V_\Gamma^1 \sqcup V_\Gamma^2$. 
The elements of $V_\Gamma^1$ are called vertices of the first type or aerial vertices. 
The elements of $V_\Gamma^2$ are called vertices of the second type or terrestrial vertices. 
\item For all $e\in E_\Gamma$, $s(e)\in V_\Gamma^1$.
\item For all $e\in E_\Gamma$, $s(e)\neq t(e)$. 
\item No two edges have the same source and the same target. 
\item The aerial vertices are labelled by the symbols 
$1, 2, 3, \dots, n$ while the terrestrial vertices are labelled 
by the symbols $\bar{1}, \bar{2}, \bar{3}, \dots, \bar{m}$. 
\item For every vertex $k \in V_\Gamma^1$ of the first type, 
the elements of $\outedge{k}$ are labelled by the symbols 
$e_k^1, e_k^2, e_k^3, \dots$. 
\end{enumerate}

\subsubsection*{Assembling a polydifferential operator from polyvector fields according to an admissible graph.} 

Fix an admissible graph $\Gamma \in \graphs_{n,m}$. 
Each choice of a vertex $v\in V_\Gamma$ 
and a map $I: E_\Gamma\to\{1,\cdots,d\}$ 
determines a constant differential operator 
\[ D_I^v := \prod_{e\in\inedge{v}} \tfrac{\partial}{\partial x_{I(e)}} \] on $\KK^d$.
Furthermore, each choice of an aerial vertex $k\in V^1_\Gamma$ 
and a map $I: E_\Gamma\to\{1,\cdots,d\}$ 
determines a map 
\[ \Tpoly{\bullet}(\KK^d)\ni \gamma\mapsto \gamma^{I(\outedge{k})}\in C^\infty(\KK^d) \] 
through the relation 
\[ \gamma^{I(\outedge{k})}= \langle dx_{I(e_k^1)} \otimes \cdots 
\otimes dx_{I(e_k^{|\outedge{k}|})} | \; \alt(\gamma)\rangle .\] 
The bundle map $\alt:\Lambda^\bullet T_{\KK^d}
\to \bigotimes^\bullet T_{\KK^d}$ is the antisymmetrization 
\[ \xi_1 \wedge \cdots \wedge \xi_r \xmapsto{\alt} \sum_{\sigma \in S_r} \sgn(\sigma)\ \xi_{\sigma(1)} \otimes \cdots \otimes \xi_{\sigma(r)} .\]
The admissible graph $\Gamma$ with $n$ aerial and $m$ terrestrial vertices provides a recipe 
for assembling $n$ polyvector fields $\gamma_1, \gamma_2, \dots, \gamma_n$ on $\KK^d$
into an $m$-differential operator 
$\kont_\Gamma (\gamma_1, \cdots, \gamma_n)$ on $\KK^d$, as given by 
\begin{equation}\label{Taichung}
(f_1, \cdots, f_m) \xmapsto{\kont_\Gamma (\gamma_1, \cdots, \gamma_n)} \sum_{I:E_\Gamma \to \{1,\cdots, d\}} 
\left(\prod_{k=1}^n D_I^k \left( \gamma_k^{I(\outedge{k})} \right) \right)
\left(\prod_{l=1}^m D_I^{\bar{l}} \big(f_l\big)\right) 
,\end{equation}
for all $f_1,\dots,f_m\in C^\infty(\KK^d)$. 
We note that $\gamma^{I(\outedge{k})}=0$ if $\gamma\in\Tpoly{r}(\KK^d)$ with 
$r+1\neq\abs{\outedge{k}}$. 
Therefore, $\kont_\Gamma (\gamma_1, \cdots, \gamma_n)=0$ 
if $\gamma_1, \cdots, \gamma_n$ are homogeneous elements of $\Tpoly{\bullet}(\KK^d)$ 
and $\abs{\gamma_1}+\cdots+\abs{\gamma_n}+n\neq |E_{\Gamma}|$. 

\subsubsection*{Configuration spaces and their compactifications}

The Kontsevich weights are obtained from integrals over compactified configuration spaces. 

Let $\HH^+ = \{z \in \CC \mid \im(z)>0\}$ denote the hyperbolic plane and let $\overline{\HH^+}$ 
denote its closure in $\CC$. The group $G_2:=\RR^+ \ltimes \RR =\{ z \mapsto az+b \mid a,b \in \RR, \; a >0\}$ acts on the configuration space 
\[ \conf_{n,m}^+ := \{ (z_1, \cdots, z_n, q_1, \cdots, q_m) \in (\HH^+)^n \times \RR^m \mid z_i\neq z_j \text{ if } i\neq j \ ;\ q_1 < \cdots < q_m \} .\] 
The quotient $C_{n,m}^+ = \conf_{n,m}^+ / G_2$ is a manifold of dimension $2n+m-2$. Fixing $z_1$ at $i=\sqrt{-1}$, 
we may identify $C_{n,m}^+$ with an open subset of $\CC^{n-1}\times\RR^m$ and transfer the standard orientation of the affine space to $C_{n,m}^+$.

We now proceed with the compactification of $C_{n,m}^+$. 

Let $\conf_n$ be the space of configurations of $n$ distinct points $z_1,z_2,\dots,z_n$ in $\CC$. 
The group $G_3:= \RR^+ \ltimes \CC$ acts on $\conf_n$ by dilations and translations. 
The quotient $C_n := \conf_n /G_3$ is a manifold which we embed into $(S^1)^{n(n-1)} \times (\RP^2)^{n(n-1)(n-2)}$ 
by recording all possible angles $\arg(z_i-z_j)$ and homogeneous coordinate triples $\homogeneoustriple{|z_i-z_j|}{|z_j-z_k|}{|z_k-z_i|}$. 

Since $C_{n,m}^+$ is itself embedded into $C_{2n+m}$ by the map 
\[ (z_1, \cdots, z_n, q_1, \cdots, q_m) \mapsto (z_1, \cdots, z_n, \overline{z_1}, \cdots, \overline{z_n}, q_1, \cdots, q_m) ,\] 
we obtain an embedding \[ C_{n,m}^+ \into C_{2n+m} \into (S^1)^{N_1} \times (\RP^2)^{N_2} \] 
with $N_1=(2n+m)(2n+m-1)$ and $N_2=(2n+m)(2n+m-1)(2n+m-2)$. 
The desired compactification of $C_{n,m}^+$ is the closure $\overline {C_{n,m}^+}$ of the image of the above embedding. 

\subsubsection*{Kontsevich weight of an admissible graph.}

Consider the hyperbolic angle function 
$\varphi: \overline{\HH^+} \times \overline{\HH^+} \to S^1$ defined by 
$\varphi(z,w) = \frac{1}{2\pi} \arg\left(\frac{z-w}{\overline{z}-w}\right)$. 

Given an admissible graph $\Gamma \in \graphs_{n,m}$, 
define a function 
$\varphi_e: C_{n,m}^+ \to S^1$ 
for each edge $e \in E_\Gamma$ by 
\[ \varphi_e(z_1, \cdots, z_n, z_{\bar 1}, \cdots, z_{\bar m}) =\varphi(z_{s(e)},z_{t(e)}) \] 
and a differential form $\kappa_\Gamma$ of degree 
$|E_\Gamma|$ on $C_{n,m}^+$ by 
\[ \kappa_\Gamma = \bigwedge_{e \in E_\Gamma} d \varphi_e ,\] 
where $d \varphi_e$ denotes the pullback of the standard volume on $S^1$ through $\varphi_e$. 
In the exterior product, the 1-forms are multiplied according to the lexicographic order 
$e_1^1, e_1^2, \dots, e_2^1, e_2^2, \dots$ of the edges of the graph.
Note that $\kappa_\Gamma$ extends smoothly to $\overline {C_{n,m}^+}$. 
Integrating $\kappa_\Gamma$ over the (oriented) compactified configuration space, we obtain the Kontsevich weight of the graph $\Gamma$: 
\[ W_\Gamma = \prod_{k=1}^n \frac{1}{|\outedge{k}|!} \int_{\overline {C_{n,m}^+}} \kappa_\Gamma .\]
Obviously, the Kontsevich weight of a graph $\Gamma\in\graphs_{n,m}$ is zero if $|E_\Gamma|\neq \dim(\overline {C_{n,m}^+})=2n+m-2$. 

\subsubsection*{Kontsevich formality theorem for $\KK^d$}

For all graphs $\Gamma\in\graphs_{n,m}$ and all homogeneous polyvector fields $\gamma_1$, $\dots$, $\gamma_n$ on $\KK^d$, 
we know that $\kont_\Gamma(\gamma_1,\cdots,\gamma_n)$ 
is a homogeneous element of degree $m-1$ in $\Dpoly{\bullet}(\KK^d)$ and that 
$W_\Gamma \kont_\Gamma(\gamma_1, \cdots, \gamma_n) \neq 0$ only when $|\gamma_1|+ \cdots +|\gamma_n| +n = |E_{\Gamma}|=2n+m-2$. 
It follows that $\kont_n : \bigotimes^n\Tpoly{\bullet}(\KK^d) \to \Dpoly{\bullet}(\KK^d)$ is a map of degree $1-n$.

Consider the case $n=1$. There are $m!$ distinct graphs $\Gamma\in\graphs_{1,m}$ satisfying the property $|E_{\Gamma}|=2\cdot 1+m-2=m$. 
In each such graph, each one of the $m$ terrestrial vertices is the target of a single edge starting from the unique aerial vertex. 
Any two such graphs only differ by the labelling of the $m$ edges. 
Moreover, all such graphs have the same weight $W_{\Gamma}=(m!)^{-2}$. 
It follows that the `first Taylor coefficient' $\kont_1$ of the formality map is precisely the Hochschild--Kostant--Rosenberg map: 
\[ \Tpoly{\bullet}(\KK^d)\xto{\kont_1=\hkr}\Dpoly{\bullet}(\KK^d) .\]

\begin{theorem}[Kontsevich formality theorem \cite{MR2062626}]\label{kontsevich}
The maps $(\kont_n)_{n=1}^\infty$ defined above are the `Taylor coefficients' of an $L_\infty$ quasi-isomorphism 
\[ \kont : \Tpoly{\bullet}(\KK^d) \to \Dpoly{\bullet}(\KK^d) \]
satisfying the following additional properties. 
\begin{enumerate}
\item \label{property:first_Taylor} The first Taylor coefficient of $\kont$ is the Hochschild--Kostant--Rosenberg map $\hkr$. 
\item \label{property:equivariance} The formality morphism $\kont$ is $GL(\KK^d)$-equivariant.
\item \label{property:vf_only} For all $n\geqslant 2$ and $\xi_1,\cdots,\xi_n\in\Tpoly{0}(\KK^d)$, we have \[ \kont_n(\xi_1,\cdots,\xi_n)=0 .\]
\item \label{property:lin vf} Provided $\xi$ is a linear vector field on $\KK^d$ and $n \geqslant 2$, we have
\[ \kont_n(\xi, \eta_2, \cdots, \eta_n) =0 \]
for all $\eta_2, \cdots, \eta_n \in \Tpoly{\bullet}(\KK^d)$.
\end{enumerate}
Furthermore, the formality morphism $\kont$ can be defined for $\kf{d}$ as well.
\end{theorem}

\begin{remark}
With suitable sign adjustments, Kontsevich's formality theorem was generalized 
to $\ZZ$-graded manifolds by Cattaneo--Felder \cite{MR2304327}. 
Later, these sign adjustments were given a simple operadic explanation \cite{MR3522653}.
\end{remark}

\subsubsection{Fiberwise formality map}
\label{section: fiberwise formality}

Let $(L,A)$ be a Lie pair over a smooth manifold $M$. As before, set $R=C^\infty(M)$. 
The quotient $B=L/A$ is a vector bundle over $M$ whose fibers are all (noncanonically) isomorphic to $\KK^d$. 
Next we apply Kontsevich's formality theorem (essentially) fiberwisely to a Fedosov dg Lie algebroid. 
See Section~\ref{Fedosov dg abd} for the construction of Fedosov dg Lie algebroids.

Since the formality morphism $\kont:\Tpoly{\bullet}(\kf{d})\to\Dpoly{\bullet}(\kf{d})$
is $GL(\KK^d)$-equivariant (see Theorem~\ref{kontsevich}~(\ref{property:equivariance})), there exist $R$-linear maps 
\[ \kont^f_n:\Lambda^n\verticalTpoly{\bullet}\to\verticalDpoly{\bullet}[1-n] \] 
whose restrictions to each fiber of $B\to M$ coincide with the Taylor coefficients 
\[ \kont_n:\Lambda^n\Tpoly{\bullet}(\kf{d})\to \Dpoly{\bullet}(\kf{d})[1-n] \] 
of $\kont$. 
Extending them $\sections{\Lambda^{\bullet}L\dual}$-linearly, we obtain $R$-linear maps 
\[ \kont^f_n:\Lambda^n\tot\big(\sections{\Lambda^{\bullet}L\dual}\otimes_R\verticalTpoly{\bullet}\big)
\to\tot\big(\sections{\Lambda^{\bullet}L\dual}\otimes_R\verticalDpoly{\bullet}\big)[1-n] .\] 

Consider the difference \[ \omega=Q-d_L^\nabla\in\sections{L\dual}\otimes_R\verticalTpoly{0}
\subset \Tpoly{0}(L[1]\oplus B) \] of the derivations $Q=-\delta+d_L^\nabla+\xnabla$ 
and $d_L^\nabla$ of the algebra \[ \sections{\Lambda^\bullet L\dual\otimes\hat{S}(B\dual)}=C^\infty(L[1]\oplus B) \] 
appearing in Theorem~\ref{strawberry}. 
We do \emph{not} claim that $\omega$ is a MC element for any dgla structure on 
$\tot(\sections{\Lambda^\bullet L\dual}\otimes_R\verticalTpoly{\bullet})$. 
Nevertheless, we define a sequence $(\Phi_n)_{n=1,2,\cdots}$ of $R$-linear maps 
\[ \Phi_n:\Lambda^n\big(\tot(\sections{\Lambda^\bullet L\dual}\otimes_R\verticalTpoly{\bullet})\big)
\to \tot(\sections{\Lambda^\bullet L\dual}\otimes_R\verticalDpoly{\bullet})[1-n] \] by 
\[ \Phi_n(\gamma)=\sum_{j=0}^{\infty}\frac{1}{j!}
\kont^f_{n+j}(\omega^j\wedge\gamma) ,\quad\forall\gamma\in
\Lambda^n\big(\tot(\sections{\Lambda^\bullet L\dual}\otimes_R\verticalTpoly{\bullet})\big) .\]

\begin{lemma}\label{Goma}
The maps $(\Phi_n)_{n=1}^\infty$ are well defined. 
\end{lemma}

\begin{proof}
Suppose $\gamma_k\in\sections{\Lambda^\bullet L\dual}\otimes_R\verticalTpoly{r_k}$ 
for $k\in\{1,2,\cdots,n\}$. Since 
\[ \kont_{n+j}:\Lambda^{n+j}\Tpoly{\bullet}(\kf{d})
\to\Dpoly{\bullet}(\kf{d}) \]
is a map of degree $1-(n+j)$ and $\omega\in\sections{L\dual}\otimes\verticalTpoly{0}$, we have 
\[ \kont^f_{n+j}(\underset{j \text{ factors}}{\underbrace{\omega\wedge\cdots\wedge\omega}}
\wedge\gamma_1\wedge\cdots\wedge\gamma_n) \in 
\sections{\Lambda^\bullet L\dual}\otimes_R\verticalDpoly{r_1+\cdots+r_n+1-(n+j)} .\] 
As $j$ increases, $r_1+\cdots+r_n+1-(n+j)$ eventually becomes smaller than $-1$ 
forcing $\kont^f_{n+j}(\omega\wedge\cdots\wedge\omega
\wedge\gamma_1\wedge\cdots\wedge\gamma_n)$ to vanish. 
Therefore, only finitely many of the terms of 
$\Phi_n(\gamma_1\wedge\cdots\wedge\gamma_n)$ are not zero. 
\end{proof}

Although $\omega$ is not a MC element, the maps $(\Phi_n)_{n=1}^\infty$ still define an $L_\infty$ morphism. 

\begin{proposition}\label{thm:renoPhi}
The maps $(\Phi_n)_{n=1}^\infty$ are the Taylor coefficients of an $L_\infty$ morphism of dglas
\[ \Phi : \big( \tot(\sections{\Lambda^\bullet L\dual}\otimes_R\verticalTpoly{\bullet}) , \fedosova , \schouten{\argument}{\argument} \big) 
\to \big( \tot(\sections{\Lambda^\bullet L\dual}\otimes_R \verticalDpoly{\bullet}) , \gerstenhaber{Q+m}{\argument} ,  \gerstenhaber{\argument}{\argument}) .\]
\end{proposition} 

We will need the following well known lemma. 
\begin{lemma}\label{bidule}
Let $(C,\bar{d})$ be a cdga, and $\kont:(\frakg, d,[\argument,\argument]) \to (\frakg', d', [\argument,\argument]')$ be an $L_\infty$ morphism of dglas
\begin{enumerate}
\item
Then $(C \otimes \frakg, \bar{d}\otimes \id + \id \otimes d,[\argument,\argument])$ and $(C \otimes \frakg', \bar{d}\otimes \id + \id \otimes d', [\argument,\argument]')$ are dglas
\item
and the $C$-linear extension of $\kont$ 
\[ \widehat{\kont}: (C \otimes \frakg, \bar{d}\otimes \id + \id \otimes d,[\argument,\argument]) \to (C \otimes \frakg', \bar{d}\otimes \id + \id \otimes d', [\argument,\argument]') \] 
is an $L_\infty$ morphism of dglas.
\end{enumerate}
\end{lemma}

\begin{proof}[Proof of Proposition~\ref{thm:renoPhi}]
Choosing a local trivialization $B|_U\cong U\times\KK^d$ of the vector bundle $B$ 
over an open subset $U$ of $M$ yields identifications 
\begin{gather*} 
\big(\sections{\Lambda^\bullet L\dual}\otimes_R\verticalTpoly{\bullet}\big)|_U
\cong \sections{U;\Lambda^\bullet L\dual}\otimes_{\KK}\Tpoly{\bullet}(\kf{d}) \\ 
\big(\sections{\Lambda^\bullet L\dual}\otimes_R\verticalDpoly{\bullet}\big)|_U
\cong \sections{U;\Lambda^\bullet L\dual}\otimes_{\KK}\Dpoly{\bullet}(\kf{d}) 
.\end{gather*}
According to Lemma~\ref{bidule}, 
the restrictions to $U$ of the maps $(\kont^f_n)_{n=1,2,\cdots}$ constructed earlier 
are the Taylor coefficients of an $L_\infty$ morphism of dglas
\[ \begin{tikzcd} 
\big( \tot(\sections{U;\Lambda^\bullet L\dual}\otimes_{\KK}\Tpoly{\bullet}(\kf{d})) , 
d_L\otimes\id_{\Tpoly{\bullet}(\kf{d})} , \schouten{\argument}{\argument} \big) 
\arrow[d, "\kont^f_U"] \\ 
\big( \tot(\sections{U;\Lambda^\bullet L\dual}\otimes_{\KK}\Dpoly{\bullet}(\kf{d})) , 
d_L\otimes\id_{\Dpoly{\bullet}(\kf{d})} +\id\otimes\hochschild , \gerstenhaber{\argument}{\argument} \big)
.\end{tikzcd} \]
In the chosen local trivialization $B|_U\cong U\times\KK^d$ of the vector bundle $B$ 
over the open subset $U$ of $M$, 
we may compare $Q$ with the derivation $d_L\otimes\id_{\KK [[ \chi_1,\cdots,\chi_d ]]}$ of the algebra 
\[ \sections{U;\Lambda^\bullet L\dual\otimes\hat{S}B\dual} 
\cong \sections{U;\Lambda^\bullet L\dual}\otimes_{\KK} \KK [[ \chi_1,\cdots,\chi_d ]] .\]
Since $Q^2=0$, the difference \[ \varpi=Q-d_L\otimes\id\in 
\sections{U;L\dual}\otimes_{\KK} \Tpoly{0}(\kf{d}) \] 
is a MC element of the dgla 
\[\tot\big(\sections{\Lambda^\bullet L\dual}\otimes_R\verticalTpoly{\bullet}\big)|_U
\cong \tot\big(\sections{U;\Lambda^\bullet L\dual}\otimes_{\KK}\Tpoly{\bullet}(\kf{d})\big) \]
endowed with the differential $d_L\otimes\id_{\Tpoly{\bullet}(\kf{d})}$ 
and the $\sections{U;\Lambda L\dual}$-multilinear extension of the Schouten bracket 
on $\Tpoly{\bullet}(\kf{d})$. 

Since 
\begin{align*} 
\kont^f_U(\varpi) &= \sum_{j=1}^{\infty} \frac{1}{j!}\kont^f_j(\varpi^j) 
&& \text{by Equation~\eqref{Taipei}} \\ 
&= \kont^f_1(\varpi) && \text{by Theorem~\ref{kontsevich}~(\ref{property:vf_only})} \\ 
&= \hkr(\varpi) && \text{by Theorem~\ref{kontsevich}~(\ref{property:first_Taylor})} \\ 
&= \varpi ,&& 
\end{align*}
we obtain the tangent $L_\infty$ morphism of $\kont^f_U$ at $\varpi$: 
\begin{multline*} \kont^f_{U,\varpi}: 
\big( \tot(\sections{\Lambda^\bullet L\dual}\otimes_R\verticalTpoly{\bullet})|_U , 
\schouten{Q}{\argument} , \schouten{\argument}{\argument} \big) \\
 \to  
\big( \tot(\sections{\Lambda^\bullet L\dual}\otimes_R\verticalDpoly{\bullet})|_U , 
\gerstenhaber{Q+m}{\argument} , \gerstenhaber{\argument}{\argument} \big) .\end{multline*}
Adapting the argument used for $\Phi$ in the proof of Lemma~\ref{Goma}, one can show that the tangent 
$L_\infty$ morphism $\kont^f_{U,\varpi}$ is well defined. 

Since the map $\Phi_n$ depends only locally on its arguments, 
we may consider its restriction to the open subset $U$ of $M$. 
We claim that the $n$-th Taylor coefficient of the $L_\infty$ morphism 
$\kont^f_{U,\varpi}$ is the restriction of $\Phi_n$ to $U$. 
Indeed, one easily checks that $\omega-\varpi$ is 
(the tensor product of a section of $L\dual$ over $U$ with) 
a \emph{linear} vertical vector field on $\kf{d}$ 
and it then follows from Theorem~\ref{kontsevich}~(\ref{property:lin vf}) that, for all 
$\gamma\in\Lambda^n\big(\tot(\sections{\Lambda^\bullet L\dual}
\otimes_R\verticalTpoly{\bullet})|_U\big)$, 
\[ \Phi_n(\gamma)=\sum_{j=0}^{\infty}\frac{1}{j!}
\kont^f_{n+j}(\omega^j\wedge\gamma) 
=\sum_{j=0}^{\infty}\frac{1}{j!}\kont^f_{n+j}(\varpi^j\wedge\gamma)
=\big(\kont^f_{U,\varpi}\big)_n(\gamma) .\]
This shows that $(\Phi_n)_{n=1,2,\cdots}$ is the sequence of Taylor coefficients of an $L_\infty$ morphism 
\[ \Phi: \big( \tot(\sections{\Lambda^\bullet L\dual}\otimes_R\verticalTpoly{\bullet}) , \fedosova , \schouten{\argument}{\argument} \big) \to 
\big( \tot(\sections{\Lambda^\bullet L\dual}\otimes_R\verticalDpoly{\bullet}) , 
\gerstenhaber{Q+m}{\argument} , \gerstenhaber{\argument}{\argument} \big) \]
defined globally on $M$. 
\end{proof}

Our construction of the $L_\infty$ morphism $\Phi$ is essentially the same as the one given by Dolgushev in \cite{MR2102846} 
except that we define its Taylor coefficients $\Phi_n$ globally from the get-go rather than by glueing local data. 

\subsection{Algebraic homomorphism property}

In this section, we sketch a proof why $\Phi_1$ is a morphism of associative algebras up to homotopy. 
For more details, the reader may want to consult \cite{MR1990011,MR2077241,MR1872382,MR2062626,MR2646112}.

\subsubsection{Kontsevich's eye}

The compactified configuration space $\overline{C_{2,0}^+}$, 
which is customarily called `Kontsevich's eye,' is represented in Figure~\ref{oeil}. 
\begin{figure}[ht] 
\caption{`Kontsevich's eye'}
\label{oeil}
\centering
\includegraphics[width=7cm]{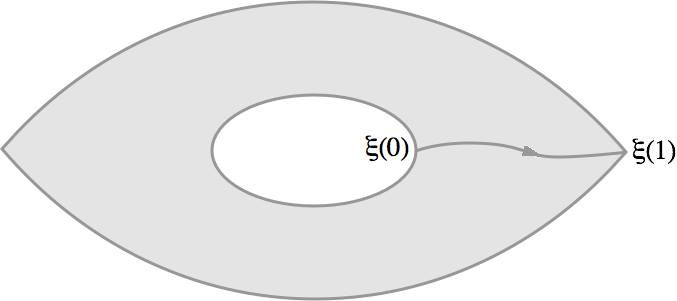} 
\end{figure}
Its boundary admits the following decomposition in strata: 
\[ \partial(\overline{C_{2,0}^+})=C_{2} \sqcup C_{1,1} \sqcup C_{1,1} \sqcup C_{0,2} .\]
The stratum $C_{2}$ --- the pupil of the eye --- is reached when the two aerial vertices 
$z_1$ and $z_2$ merge. 
The first copy of $C_{1,1}$ --- the upper eyelid --- is reached when the aerial vertex $z_1$ approaches the real line. 
The second copy of $C_{1,1}$ --- the lower eyelid --- is reached when the aerial vertex $z_2$ approaches the real line. 
The stratum $C_{0,2}$ is made of two points --- the corners of the eye. 
The left corner is reached when the vertices $z_1$ and $z_2$ each approach a distinct point of the real line simultaneously and $z_1$ is the leftmost of the two points. 
The right corner is reached when the vertices $z_1$ and $z_2$ each approach a distinct point of the real line simultaneously and $z_1$ is the rightmost of the two points. 

\subsubsection{Vanishing lemma}

Given a configuration space $C_{n,m}^+$ with $n\geqslant 2$, consider the projection $\pi: C_{n,m}^+ \to C_{2,0}^+$, 
which forgets all but the first two of the $n$ aerial points in $\HH^+$ and all $m$ points on the real line. 
More precisely, consider its continuous extension $\overline{\pi}: \overline{C_{n,m}^+} \to \overline{C_{2,0}^+}$ to the compactified configuration spaces. 

Now choose a smooth path $\xi: [0,1] \to \overline{C_{2,0}^+}$ starting from a point on the inner boundary of Kontsevich's eye and ending at the right corner. 
The inverse image of $\xi([0,1])$ under $\pi$ in $\overline{C_{n,m}^+}$ is a compact subspace denoted $Z_{n,m}$. 
Kontsevich assigns a weight 
\[ \widetilde{W}_\Gamma = \prod_{k=1}^n \frac{1}{|\outedge{k}|!} \int_{Z_{n,m}} j^*(\kappa_\Gamma) \]
to each admissible graph $\Gamma\in\graphs_{n,m}$. The symbol $j$ denotes the embedding of $Z_{n,m}$ into $\overline{C_{n,m}^+}$. 
Since $\dim Z_{n,m} = 2n+m-3$ and $\kappa_\Gamma$ is an $|E_\Gamma|$-form, the weight $\widetilde{W}_\Gamma$ is zero unless $|E_\Gamma| = 2n+m-3$.

The following vanishing lemma is analogous to Theorem~\ref{kontsevich}~(\ref{property:lin vf}). 

\begin{lemma}\label{lem:lin vf for hpt op}
If $\xi\in\Tpoly{0}(\KK^d)$ is a vector field \emph{linear on $\KK^d$} 
and $\Gamma\in\graphs_{n,m}$ is an admissible graph with $n \geqslant 3$, then 
\[ \widetilde{W}_\Gamma \kont_\Gamma(\eta_1,\cdots,\eta_{n-1},\xi) =0 \]
for all $\eta_1,\cdots,\eta_{n-1}\in\Tpoly{\bullet}(\KK^d)$. 
\end{lemma}

All ingredients of the proof can be found in Kontsevich's original paper~\cite{MR2062626}. 

Recall the hyperbolic angle function $\varphi: \overline{\HH^+} \times \overline{\HH^+} \to S^1$ 
defined by $\varphi(z,w)=\frac{1}{2\pi}\arg\left(\frac{z-w}{\overline{z}-w}\right)$. 
Given a point $z_0$ of $\overline{\HH^+}$, let $d\varphi(z,z_0)$ denote the pullback to $\HH^+$ 
of the standard volume form on $S^1$ through the function $\HH^+\ni z\mapsto \varphi(z,z_0)\in S^1$. 
Likewise let $d\varphi(z_0,z)$ denote the pullback to $\HH^+$ of the standard volume form on $S^1$ 
through the function $\HH^+\ni z\mapsto \varphi(z_0,z)\in S^1$. 

\begin{lemma}[{\cite[Lemmas 7.3, 7.4, and 7.5]{MR2062626}}]\label{lem:vanishing}
\begin{enumerate}
\item
For every pair of distinct points $z_1$ and $z_2$ in $\HH^+$, we have 
\[ \int_{z \in \HH^+ \smallsetminus \{z_1, z_2\}} d\varphi(z_1,z) \wedge d\varphi(z,z_2) =0 .\]
\item
For every pair of points $z_1\in\HH^+$ and $z_2\in\RR=\partial(\HH^+)$, we have 
\[ \int_{z \in \HH^+ \smallsetminus \{z_1\}} d\varphi(z_1,z) \wedge d\varphi(z,z_2) =0 .\]
\item
For every point $z_0$ in $\HH^+$, we have 
\[ \int_{z \in \HH^+ \smallsetminus \{z_0\}} d\varphi(z_0,z) \wedge d\varphi(z,z_0) =0 .\]
\end{enumerate}
\end{lemma}

\begin{proof}[Proof of Lemma~\ref{lem:lin vf for hpt op}]
Since $\xi\in\Tpoly{0}(\KK^d)$, we have $\xi^{I(\outedge{n})}=0$ for all 
maps $I:E_\Gamma\to\{1,\cdots,d\}$ unless $\abs{\outedge{n}}=1$.
Moreover, if $\abs{\inedge{n}}>1$, the order of the differential operator $D^n_I$ is at least two, 
no matter which map $I:E_\Gamma\to\{1,\cdots,d\}$ is considered.
Therefore  the function $D^n_I \xi^{I(\outedge{n})}$ vanishes
since $\xi$ is linear. 
Hence $\kont_\Gamma(\eta_1,\cdots,\eta_{n-1},\xi)=0$ unless 
$\abs{\outedge{n}}=1$ and $\abs{\inedge{n}}\leqslant 1$.
We may thus assume without loss of generality 
that $\outedge{n}=\{e_n^1\}$ and $\abs{\inedge{n}}\in\{0,1\}$. 

Let's assume for now that $\outedge{n}=\{e_n^1\}$ and $\inedge{n}=\{e'\}$ 
--- we will treat the other case later. 
Consider the graph $\Delta\in\graphs_{n-1,m}$ obtained 
from $\Gamma\in\graphs_{n,m}$ by removing the $n$-th aerial vertex $n$ and all the edges 
starting or ending at it. We have 
\[ \kappa_\Gamma=\pm d\varphi_{e'}\wedge d\varphi_{e_n^1}\wedge F_n^*(\kappa_{\Delta}) ,\] 
where $F_n:C_{n,m}^+\onto C_{n-1,m}^+$
is the projection which forgets the $n$-th aerial point $z_n$ 
of a configuration $(z_1,\cdots,z_n;q_1,\cdots,q_m)$. 
Making use of Fubini's theorem, we obtain 
\[ \widetilde{W}_\Gamma 
=\int_{Z_{n,m}}j^*(\kappa_\Gamma)
=\pm\int_{Z_{n,m}}j^*\big(d\varphi_{e'}\wedge d\varphi_{e_n^1}\wedge F_n^*(\kappa_{\Delta})\big)
=\pm\int_{F_n(Z_{n,m})}f\cdot j^*(\kappa_{\Delta}) ,\] 
where $f$ denotes the function on $C_{n-1,m}^+$ obtained by integration 
of the 2-form $d\varphi_{e'}\wedge d\varphi_{e_n^1}$ along the fibers of 
$F_n:C_{n,m}^+\onto C_{n-1,m}^+$. Since 
\[ f(z_1,\cdots,z_{n-1};z_{\bar{1}},\cdots,z_{\bar{m}}) 
=\int_{z_n\in\HH^+\smallsetminus\{z_1,\cdots,z_{n-1}\}} 
d\varphi(z_{s(e')},z_n)\wedge d\varphi(z_n,z_{t(e_n^1)}) ,\] 
it follows from Lemma~\ref{lem:vanishing} that $\widetilde{W}_\Gamma=0$.

Finally, we turn our attention to the situation where $\outedge{n}=\{e_n^1\}$ and 
$\inedge{n}=\varnothing$. We start with making two observations. 
\begin{enumerate}
\item For all $e\in E_\Gamma$ with $e\neq e_n^1$, the aerial vertex $n$ is neither the source 
nor the target of the edge $e$ and, consequently, the function $\varphi_e:C_{n,m}^+\to S^1$ 
is constant along the fibers of the projection $F_n:C_{n,m}^+\onto C_{n-1,m}^+$
which forgets the $n$-th aerial point $z_n$ of a configuration $(z_1,\cdots,z_n;q_1,\cdots,q_m)$. 
\item Each fiber of the projection $F_n:C_{n,m}^+\onto C_{n-1,m}^+$ is diffeomorphic to $\HH^+$ 
punctured at $n-1$ points and is foliated by its intersections with the level sets of the function 
$\varphi_{e_n^1}:C_{n,m}^+\to S^1$. In other words, $C_{n,m}^+$ is foliated by the fibers of 
$F_n$, which are themselves foliated by curves along which the function $\varphi_{e_n^1}$ 
is constant. Obviously, the subspace $Z_{n,m}$ of $C_{n,m}^+$ is a union of such curves. 
\end{enumerate}
It follows from these observations that $Z_{n,m}$ is foliated by curves along which 
all functions $\varphi_e$ for all edges $e\in E_\Gamma$ are constant. 
Therefore, the component of the form $j^*(\bigwedge_{e\in E_\Gamma}d\varphi_e)$ 
of degree $\dim(Z_{n,m})=2n+m-3$ vanishes. 
Hence $\widetilde{W}_\Gamma=\int_{Z_{n,m}}j^*(\kappa_\Gamma)=0$. 
\end{proof}

\subsubsection{Homotopy operator}

For every admissible graph $\Gamma \in \graphs_{n,m}$, the operator 
\[ \kont_\Gamma : \Tpoly{\bullet}(\kf{d})^{\otimes n} \to \Dpoly{m-1}(\kf{d}) \]
defined by Equation~\eqref{Taichung} is $GL_d(\KK)$-equivariant. 
Therefore there exists an $R$-linear map 
\[ \kont^f_\Gamma:\big(\verticalTpoly{\bullet}\big)^{\otimes n}\to\verticalDpoly{m-1} \] 
whose restrictions to each fiber of $B\to M$ coincide with $\kont_\Gamma$. 
Extending the latter $\sections{\Lambda^{\bullet}L\dual}$-multilinearly, 
we obtain an $R$-linear operator 
\[ \kont^f_\Gamma: (\Gamma(\Lambda^\bullet L\dual) \otimes_R \verticalTpoly{\bullet})^{\otimes n} 
\to \Gamma(\Lambda^\bullet L\dual) \otimes_R \verticalDpoly{m-1} .\]

Using the maps $\kont^f_\Gamma$, the weights $\widetilde{W}_\Gamma$, and 
the difference \[ \omega=Q-d_L^\nabla\in\sections{L\dual}\otimes_R\verticalTpoly{0} \] 
of the derivations $Q$ and $d_L^\nabla$ appearing in Theorem~\ref{strawberry}, 
we define an operator 
\[ H: \big(\Gamma(\Lambda^\bullet L\dual) \otimes_R \verticalTpoly{\bullet}\big) 
\times \big(\Gamma(\Lambda^\bullet L\dual) \otimes_R \verticalTpoly{\bullet}\big) 
\to \Gamma(\Lambda^\bullet L\dual) \otimes_R \verticalDpoly{\bullet} \]
by
\begin{equation}\label{Agbogbloshie} 
H(\alpha,\beta)=\sum_{\substack{n\geqslant 0 \\ m\geqslant 0}} \frac{1}{n!} \sum_{\Gamma\in\graphs_{2+n,m}} 
(-1)^{m+1} \widetilde{W}_\Gamma \kont^f_\Gamma(\alpha,\beta,
\underset{n \text{ arguments}}{\underbrace{\omega,\cdots,\omega}})
.\end{equation}

Note that since $\omega$ has ``exterior degree'' $1$, the polydifferential operator
$\kont^f_\Gamma(\alpha,\beta,\overset{n \text{ arguments}}{\overbrace{\omega,\cdots,\omega}})$ 
vanishes when $n>\rk(L)$. 
Furthermore, the coefficient $\widetilde{W}_\Gamma$ vanishes for $m$ large enough. 
Therefore, the summation \eqref{Agbogbloshie} involves only finitely many nonzero terms 
and the operator $H$ is indeed well-defined  

\begin{proposition}
For all $\alpha\in(\Gamma(\Lambda^\bullet L)\otimes\verticalTpoly{a})$ 
and $\beta\in(\Gamma(\Lambda^\bullet L)\otimes\verticalTpoly{b})$, we have 
\[ \Phi_1(\alpha\cup\beta) - \Phi_1(\alpha)\cup\Phi_1(\beta)
= \gerstenhaber{Q+m}{H(\alpha,\beta)} - H\big(\schouten{\fedosov}{\alpha},\beta\big) 
- (-1)^{a} H\big(\alpha,\schouten{\fedosov}{\beta}\big) .\]
\end{proposition}

\begin{proof}[Sketch of proof]
Recall from the proof of Theorem~\ref{thm:renoPhi} that, although $\omega$ is not a Maurer--Cartan element, 
its restriction to any open subset of the manifold $M$ over which the vector bundle $B$ is trivial 
is equal to the sum of a Maurer--Cartan element and a linear vector field. 
It follows from Lemma~\ref{lem:lin vf for hpt op}, the definitions and locality of $\Phi_1$ and $H$ that, 
for the purpose of this proof, $\omega$ may be treated as if it were a Maurer--Cartan element. 
The rest of the proof is then virtually identical to a difficult computation due to Manchon and Torossian 
\cite[Théorème~4.6]{MR1990011,MR2077241} --- see also Mochizuki's work \cite[Equation 56]{MR1872382}. 
There is only one significant difference with \cite[Théorème~4.6]{MR1990011,MR2077241}: 
their Poisson bivector $\hbar\gamma$ must be replaced by our vector field $\omega$. 
This is responsible for the discrepancy in the number of edges of the admissible graphs appearing here and in \cite[Théorème~4.6]{MR1990011,MR2077241}. 
\end{proof}

\begin{remark}
The operator $H$ defined above, which is implicit in \cite[Théorème~4.6]{MR1990011,MR2077241} 
and \cite[Equation 56]{MR1872382}, was made explicit in \cite[Proposition~9.1]{MR2646112}.
\end{remark}

\subsection{Explicit formula for \texorpdfstring{$\Phi_1$}{Phi1}}

Consider the first `Taylor coefficient' 
\[ \tot\big(\sections{\Lambda^\bullet L\dual}\otimes_R\verticalTpoly{\bullet}\big) 
\xto{\Phi_1} \tot\big(\sections{\Lambda^\bullet L\dual}\otimes_R\verticalDpoly{\bullet}\big) \] 
of the $L_\infty$ morphism $\Phi$ constructed in Section~\ref{section: fiberwise formality}. 

In this section, we prove the following

\begin{proposition}
\label{thm:Florence}
The map $\Phi_1: \tot(\sections{\Lambda^\bullet L^\vee}\otimes_R \verticalTpoly{\bullet}) \to \tot(\sections{\Lambda^\bullet L^\vee}\otimes_R \verticalDpoly{\bullet})$ 
is the modification 
of the Hochschild--Kostant--Rosenberg map by (the square root of) the canonical Todd cocycle: 
\[ \Phi_1 = \hkr \circ (\ttodd_\cF^{\can})^{\frac{1}{2}} .\]
\end{proposition}

Suppose that an open subset $U$ of $M$ diffeomorphic to $\RR^m$ is the domain of a coordinate chart of $M$ over which the vector bundles $L$ and $B$ are trivial. 
The algebra of functions of the graded manifold $\cV=\RR^m\oplus\kkf{r}{l}$ obtained by restriction of the Fedosov dg manifold $\cM$ to the support $U$ is 
\[ C^\infty(\cV)=C^\infty(\RR^m)\otimes\hat{S}\big((\KK^r\oplus\KK^l[1])\dual\big) .\] 
Here $m$ is the dimension of the manifold $M$ while $r$ is the rank of the vector bundle $B$ and $l$ is the rank of the vector bundle $L$. 

There are natural injections
\[ \begin{tikzcd}[column sep=small] 
& C^\infty(\RR^m\times\kkf{0}{l})\otimes\Tpoly{0}\kkf{r}{0} \arrow[d, , hook] \\ 
\XX(\RR^m\times\kkf{0}{l}) \arrow[r, hook] & \XX(\cV) 
.\end{tikzcd} \]
The restriction of the Fedosov homological vector field $Q\in\XX(\cM)$ to $U$ is the sum 
\[ Q=d_L + \varpi \] 
of $d_L\in \XX(\RR^m\times\kkf{0}{l})$ and $\varpi\in C^\infty(\RR^m\times\kkf{0}{l})\otimes\Tpoly{0}\kkf{r}{0}$
as observed in Section~\ref{section: fiberwise formality}. 

\begin{lemma}\label{lem:Osaka}
The sum $Q=Q_1 + Q_2$ 
of two vector fields $Q_1\in\XX(\RR^m\times\kkf{0}{l})\subset\XX(\cV)$ 
and $Q_2\in C^\infty(\RR^m\times\kkf{0}{l})\otimes\Tpoly{0}\kkf{r}{0}\subset\XX(\cV)$
is a homological vector field on $\cV=\RR^m\otimes\kkf{r}{l}$ if and only if
\textbf{(1)} $Q_1$ is a homological vector field on $\RR^m\times\kkf{0}{l}$ 
(i.e.\ $Q_1$ is of degree $+1$ and $Q^2_1=0$)
and \textbf{(2)} $Q_2$ satisfies the Maurer--Cartan equation 
$\liederivative{Q_1} Q_2+\half [Q_2, Q_2]=0$. \newline
Moreover, in this case, $(C^\infty(\RR^m\times\kkf{0}{l}),Q_1)$ is a cdga 
and $Q_2$ is a MC element in 
$C^\infty(\RR^m\times\kkf{0}{l})\otimes\Tpoly{\bullet}(\kf{r})$ 
endowed with its $L_\infty$ algebra structure determined by $Q_1$ and $\Tpoly\bullet(\kf{r})$.
\end{lemma}
According to Va{\u\i}ntrob~\cite{MR1480150}, Condition (1) in Lemma~\ref{lem:Osaka} above 
means that the trivial vector bundle $\RR^m\times\KK^l\to\RR^m$ carries a Lie algebroid structure. 

It follows from Lemma~\ref{lem:Osaka} that 
$\big(C^\infty(\RR^m\times\kkf{0}{l}),d_L\big)$ is a cdga 
and $\varpi$ is a Maurer--Cartan element of the dgla
$C^\infty(\RR^m\times\kkf{0}{l})\otimes\Tpoly{\bullet}\kkf{r}{0}$ determined by 
the differential $d_L\otimes\id_{\Tpoly{0}\kkf{r}{0}} =\liederivative{d_L}$
and the restriction of the Schouten bracket in $\Tpoly{\bullet}\cV$.

In Section~\ref{section: fiberwise formality}, we proved that $\Phi_1$ 
depends only locally on its arguments and that its restriction to $U$ is 
\[ C^\infty(\RR^m\times\kkf{0}{l})\otimes\Tpoly{\bullet}(\kf{r}) 
\xto{\big(\kont^f_{U,\varpi}\big)_1} 
C^\infty(\RR^m\times\kkf{0}{l})\otimes\Dpoly{\bullet}(\kf{r}) ,\] 
which is the first Taylor coefficient of the tangent $L_\infty$ morphism to $\kont^f_U$ at the Maurer--Cartan element $\varpi$. 

Therefore, to establish Proposition~\ref{thm:Florence}, it suffices to prove that 
\[ \big(\kont^f_{U,\varpi}\big)_1= \hkr\circ \big(\ttodd^{\can}_{\cF|_U}\big)^{\frac{1}{2}} \] 
in every coordinate chart $U$ of $M$ over which the vector bundles $B$ and $L$ are trivial. 

Note that the dglas $C^\infty(\RR^m\times\kkf{0}{l})\otimes\Tpoly{\bullet}(\kf{r})$ 
and $C^\infty(\RR^m\times\kkf{0}{l})\otimes\Dpoly{\bullet}(\kf{r})$, 
the restrictions of $\tot(\sections{\Lambda^\bullet L^\vee}\otimes_R \verticalTpoly{\bullet})$ 
and $\tot(\sections{\Lambda^\bullet L^\vee}\otimes_R \verticalDpoly{\bullet})$ over $U$, 
are dg Lie subalgebras of $\big(\Tpoly{\bullet}\cV,\schouten{Q}{\argument},\schouten{\argument}{\argument}\big)$ 
and $\big(\Dpoly{\bullet}\cV,\gerstenhaber{Q+m}{\argument},\gerstenhaber{\argument}{\argument}\big)$, respectively.

Now, consider the Kontsevich formality $L_\infty$ quasi-isomorphism 
\[ (\Tpoly{\bullet}\cV, 0, \schouten{\argument}{\argument})
\xto{\kont^{\cV}} (\Dpoly{\bullet}\cV, \gerstenhaber{m}{\argument}, \gerstenhaber{\argument}{\argument}) \] 
devised for $\cV=\RR^m\times\kkf{r}{l}$ in~\cite[Appendix]{MR2304327}.

Since $[Q, Q]=0$, the vector field $Q\in \Tpoly{0}\cV$ is a Maurer--Cartan element of the dgla $\Tpoly{\bullet}\cV$ 
and we can consider the tangent $L_\infty$ morphism $\kont^{\cV}_Q$ defined by Equation~\eqref{tangent}.
Since $Q$ is a vector field, it follows from Theorem~\ref{kontsevich}~(\ref{property:lin vf}) 
that $\kont^{\cV}(Q)=Q\in\Dpoly{0}\cV$, where $\kont^{\cV}(Q)$ is given by the graded version of Equation~\eqref{Taipei} as in \cite{MR2304327}. 
Hence we obtain the $L_\infty$ morphism
\[ (\Tpoly{\bullet}\cV, \schouten{Q}{\argument}, \schouten{\argument}{\argument})
\xto{\kont^{\cV}_Q} (\Dpoly{\bullet}\cV,  \gerstenhaber{Q+m}{\argument}, \gerstenhaber{\argument}{\argument}) .\]

\begin{lemma}\label{lem:Tokyo}
In the category of cochain complexes of $\KK$-modules, the diagram 
\[ \begin{tikzcd}
\big(\Tpoly{\bullet}(\cV),\schouten{Q}{\argument}\big) \arrow[r, "\big(\kont^{\cV}_Q\big)_1"] & 
\big(\Dpoly{\bullet}(\cV),\gerstenhaber{Q+m}{\argument}\big) \\ 
C^\infty(\RR^m\times\kkf{0}{l})\otimes\Tpoly{\bullet}\kf{r} \arrow[u,"\cI" , hook] \arrow[r, "\big(\kont^f_{U,\varpi}\big)_1", swap] & 
C^\infty(\RR^m\times\kkf{0}{l})\otimes\Dpoly{\bullet}\kf{r} \arrow[u,"\cI" , hook]
\end{tikzcd} \]  
is commutative. 
\end{lemma}

\begin{proof}
Let $\gamma\in C^\infty(\RR^m\times\kkf{0}{l})\otimes\Tpoly{k}\kf{r}$ be a $(k+1)$-vector field. 
It follows from Equations~\eqref{tangent}, \eqref{def:Kontsevich}, and~\eqref{Taichung} that 
\[ \begin{split} 
\big(\kont^{\cV}_Q\big)_1(\gamma)
&{}= \sum_{j=0}^{\infty} \frac{1}{j!} \kont^{\cV}_{1+j}(Q\wedge\cdots\wedge Q\wedge\gamma) 
\\ 
&{}= \sum_{j=0}^{\infty} \frac{1}{j!} \sum_{m\geqslant 0} \sum_{\Gamma\in\graphs_{1+j,m}} W_\Gamma \kont^{\cV}_\Gamma(Q\wedge\cdots\wedge Q\wedge\gamma) 
\end{split} \]
and 
\[ \begin{split} 
\big(\kont^f_{U,\varpi}\big)_1(\gamma)
&{}= \sum_{j=0}^{\infty} \frac{1}{j!} \big(\kont^f_{U}\big)_{1+j}(\varpi\wedge\cdots\wedge\varpi\wedge\gamma) \\ 
&{}= \sum_{j=0}^{\infty} \frac{1}{j!} \kont^{\cV}_{1+j}(\varpi\wedge\cdots\wedge\varpi\wedge\gamma) 
\\ 
&{}= \sum_{j=0}^{\infty} \frac{1}{j!} \sum_{m\geqslant 0} \sum_{\Gamma\in\graphs_{1+j,m}} W_\Gamma \kont^{\cV}_\Gamma(\varpi\wedge\cdots\wedge\varpi\wedge\gamma) 
.\end{split} \]
Therefore, since $Q=d_L+\varpi$, it suffices to prove that, provided $X_1,X_2,\dots,X_j\in\{d_L,\varpi\}$ and $X_p=d_L$ for at least one $p\in\{1,2,\dots,j\}$, 
the expression $W_\Gamma\kont^{\cV}_\Gamma(X_1\wedge X_2\wedge\cdots\wedge X_j\wedge\gamma)$ vanishes for all $\Gamma\in\graphs_{j+1,m}$. 

Given a graph $\Gamma\in\graphs_{1+j,m}$, we know that 
\begin{itemize}
\item $W_\Gamma=0$ if $\abs{E_\Gamma}\neq 2j+m$; 
\item $\kont_\Gamma(X_1,\cdots,X_j,\gamma)=0$ if $\abs{E_\Gamma}\neq j+k+1$; 
\item and $\kont_\Gamma(X_1,\cdots,X_j,\gamma)=0$ if the number of edges starting from the $(j+1)$-th aerial vertex is different from $k+1$ 
(since $\gamma$ is a $(k+1)$-vector field). 
\end{itemize}
Therefore, if $\Gamma\in\graphs_{1+j,m}$, we have $W_\Gamma\kont_\Gamma(X_1,\cdots,X_j,\gamma)=0$ unless 
$\abs{\outedge{v_{j+1}}}=k+1=j+m$. 
In other words, since $(j+1)+m$ is the total number of vertices of the graph $\Gamma$, 
we have $W_\Gamma\kont_\Gamma(X_1,\cdots,X_j,\gamma)=0$ unless a single edge runs 
from the $(j+1)$-th aerial vertex of $\Gamma$ to each one of the other $j+m$ vertices of $\Gamma$. 

However, if an edge $e'$ of $\Gamma$ starts at the $(j+1)$-th aerial vertex $v_{j+1}$ (the aerial vertex corresponding to $\gamma$) 
and ends at $v_p$ (the aerial vertex corresponding to $X_p=d_L$), the factor $D_I^{v_p}(X_p^{I(\outedge{p})})$ 
appearing in each term of the expansion \eqref{Taichung} of $\kont_\Gamma(X_1,\cdots,X_j,\gamma)$ must vanish. 
Indeed, $D_I^{v_p}$ is a composition of one or more partial derivatives w.r.t.\ the coordinates on $\kf{r}$ 
containing $\frac{\partial}{\partial x_{I(e')}}$ at the very least, while $X_p^{I(\outedge{p})}=(d_L)^{I(\outedge{p})}$ 
is a function in the subalgebra $C^\infty(\RR^m\times\kkf{0}{l})$ of $C^\infty(\cV)$. 

The proof is complete. 
\end{proof}

The following theorem was first announced by Shoikhet in \cite{arXiv:math/9812009}. 
We refer the interested reader to \cite{MR1854132} for more details.

\begin{theorem} [Kontsevich--Shoikhet \cite{arXiv:math/9812009}] \label{thm:Shoikhet}
The first `Taylor coefficient' $\big(\kont^{\cV}_Q\big)_1:\Tpoly{\bullet}\cV\to\Dpoly{\bullet}\cV$ 
of the tangent $L_\infty$ quasi-iso\-mor\-phism $\kont^{\cV}_Q$ is the modification 
\begin{equation*}
\big(\kont^{\cV}_Q\big)_1=\hkr\circ\big(\ttodd^{\trivial}_{T_\cV}\big)^{\frac{1}{2}}
\end{equation*}
of the Hochschild--Kostant--Rosenberg map by (the square root of) the Todd cocycle 
$\ttodd^{\trivial}_{T_\cV}\in\prod_{k=0}^\infty\Omega^k(\cV)$ 
of the dg manifold $(\cV,Q)$ 
associated with the trivial connection as in Example~\ref{ex:trivial cntn}. 
The Todd cocycle acts on $\Tpoly{\bullet}\cV$ by contraction.
\end{theorem}

The Todd cocycle $\ttodd^{\trivial}_{T_\cV}$ can be expressed in terms of the scalar Atiyah cocycles 
$\supertrace\big((\At^{\trivial}_{T_{\cV}})^s\big) \in \sections{\Lambda^s T_\cV \dual}$, 
which are related to the scalar Atiyah cocycles $\trace\big((\At_{\cF|_U}^{\can})^s\big) \in \sections{\Lambda^s\cF|_U\dual}$ 
of the restriction $\cF|_U$ of the Fedosov dg Lie algebroid in the following way:

\begin{lemma}
For all $s\in\NN$, we have
\[ \trace\big((\At_{\cF|_U}^{\can})^s\big) = \cI\transpose \supertrace\big((\At^{\trivial}_{T_{\cV}})^s\big) ,\]
where $\cI\transpose$ is the transpose of the bundle map 
$\cI:\Lambda^s \cF|_U\cong\cV\times\Lambda^s\kkf{r}{0}\hookrightarrow\cV\times\Lambda^s\cV\cong\Lambda^s T_\cV$.
\end{lemma}

\begin{proof}
Let $U\xto{(x_1,\dots,x_m)}\RR^m$ be a local chart of $M$. 

Let $\partial_1,\dots,\partial_r$ be a local frame for $B\to M$ over $U$ 
and let $\chi_1,\dots,\chi_r$ be the dual local frame for $B\dual\to M$. 
Likewise, let $\eta_1,\dots,\eta_l$ be a local frame for $L\to M$ over $U$ 
and let $\lambda_1,\dots,\lambda_l$ be the dual local frame for $L\dual\to M$ with the degree shift: $|\lambda_j| = 1$. 

The restrictions to $U$ of the anchor map $\rho:L\to T_M$, 
the Lie bracket on $\sections{L}$, the bundle map $q:L\to B$, 
and the $L$-connection $\nabla$ on $B$ admit local expressions 
\begin{align*}
\rho(\eta_j) &{} =\sum_{i=1}^m \rho_{ij}\frac{\partial}{\partial x_i} &
\lie{\eta_i}{\eta_j} &{} =\sum_{k=1}^l c_{ij}^k \eta_k
\\
q(\eta_j) &{} =\sum_{i=1}^r q_{ij} \partial_i & 
\nabla_{\eta_i}\partial_j &{} =\sum_{k=1}^r\Gamma_{ij}^k\partial_k 
\end{align*}
where $\rho_{ij}$, $c_{ij}^k$, $q_{ij}$, and $\Gamma_{ij}^k$ 
are functions of the coordinates $x_1,\dots,x_m$. 

Then $(z_1,\dots,z_{m+l+r})=(x_1,\dots,x_m,\lambda_1,\dots,\lambda_l,\chi_1,\dots,\chi_r)$ are coordinates on $\cV=\RR^m\times\kkf{r}{l}$ whose degrees are
\[ |x_i| = 0, \qquad |\lambda_j| = 1, \qquad |\chi_k| = 0 ,\]
for $i\in\{1,\cdots,m\}$, $j\in\{1,\cdots,l\}$, and $k\in\{1,\cdots,r\}$. The homological vector field on $\cV$ is the sum 
$Q=-\delta+d_L^\nabla+X^\nabla$ of 
\begin{gather*} 
\delta=\sum_{k=1}^r\sum_{j=1}^l q_{kj}\lambda_j\frac{\partial}{\partial\chi_k} ,\\ 
d_L^\nabla=\sum_{j=1}^m\sum_{i=1}^l \rho_{ji}\lambda_i\frac{\partial}{\partial x_j} 
-\frac{1}{2}\sum_{i,j,k=1}^l c_{ij}^k \lambda_i\lambda_j
\frac{\partial}{\partial\lambda_k} 
-\sum_{i=1}^l\sum_{j,k=1}^r \Gamma_{ij}^k\lambda_i\chi_j\frac{\partial}{\partial \chi_k} ,\\ 
\intertext{and} 
X^\nabla=\sum_{k=1}^r f_k \frac{\partial}{\partial\chi_k}
.\end{gather*} 

Let $\hat{z}_1, \cdots, \hat{z}_{m+l+r}$ be the local frame of $T_\cV\dual$ dual to $\frac{\partial}{\partial z_1}, \cdots, \frac{\partial \;}{\partial z_{m+l+r}}$. 
This local frame is essentially $dz_1, \cdots, dz_{m+l+r}$, but they have different degrees: $|\hat{z}_i| = |dz_i| -1 = |z_i|$. It follows from Lemma~\ref{lem:local Atiyah} that 
\begin{equation}\label{crossbow}
\At^{\can}_{\cF|_U} = \sum_{i,j,k=1}^r \frac{\partial^2 f_k}{\partial\chi_i\partial\chi_j} 
\hat\chi_i\otimes\left(\hat\chi_j\otimes\frac{\partial}{\partial\chi_k}\right) 
\end{equation}
and from Example~\ref{ex:trivial cntn} that 
the Atiyah 1-cocycle of the dg Lie algebroid $T_{\cV}$ associated with 
the trivial connection $\nabla^{\trivial}_{\frac{\partial}{\partial z_i}}
\frac{\partial}{\partial z_j}=0$ is 
\[ \At^{\trivial}_{T_{\cV}} = \sum_{i,j,k=1}^{m+l+r} (-1)^{\abs{z_i}+\abs{z_j}} 
\frac{\partial^2 \big(Q(z_k)\big)}{\partial z_i\partial z_j} \hat z_i\otimes\left(\hat z_j\otimes\frac{\partial}{\partial z_k}\right) .\]

Then, we have 
\[ \big(\cI\transpose\otimes\id\big) (\At^{\trivial}_{T_{\cV}}) = \sum_{i=1}^r 
\sum_{j,k=1}^{m+l+r} (-1)^{\abs{z_j}} 
\frac{\partial^2 \big(Q(z_k)\big)}{\partial \chi_i\partial z_j} 
\hat\chi_i\otimes\left(\hat z_j\otimes\frac{\partial}{\partial z_k}\right) .\]

Since $Q(x_k)=\sum_{i=1}^l \rho_{ki}\lambda_i$ and the functions $\rho_{ki}$ 
depend on the $x$-coordinates only, we have 
\[ \frac{\partial^2 \big(Q(x_k)\big)}{\partial \chi_i\partial z_j} =0 .\] 

Since $Q(\lambda_k)=-\frac{1}{2}\sum_{i,j=1}^l c_{ij}^k 
\lambda_i\lambda_j$ and the functions $c_{ij}^k$ 
depend on the $x$-coordinates only, we have 
\[ \frac{\partial^2 \big(Q(\lambda_k)\big)}{\partial \chi_i\partial z_j} =0 .\] 

Since $Q(\chi_k)=-\sum_{j=1}^l q_{kj}\lambda_j -\sum_{i,j}\Gamma_{ij}^k\lambda_i\chi_j +f_k$, 
the functions $q_{kj}$ and $\Gamma_{ik}^j$ depend on the $x$-coordinates only, we have 
\[ \frac{\partial^2 \big(Q(\chi_k)\big)}{\partial \chi_i\partial \chi_j} =\frac{\partial^2 f_k}{\partial\chi_i\partial\chi_j} .\] 
Therefore, the matrix representation of $\big(\cI\transpose\otimes\id\big) (\At^{\trivial}_{T_{\cV}})$
with respect to the frame $(\frac{\partial}{\partial z_1}, \cdots, \frac{\partial \;}{\partial z_{m+l+r}})$ is 
\begin{equation}\label{archer}
\big(\cI\transpose\otimes\id\big) (\At^{\trivial}_{T_{\cV}}) =
\begin{bmatrix} 0 & 0 & \ast \\ 0 & 0 & \ast \\ 0 & 0 & \frac{\partial^2 f_k}{\partial\chi_i\partial\chi_j}\hat\chi_i \end{bmatrix}.
\end{equation}
It follows from Equations~\eqref{crossbow} and~\eqref{archer} that 
\begin{align*}
\trace\big((\At_{\cF|_U}^{\can})^s\big) &=
\supertrace\big(((\cI\transpose\otimes\id)(\At^{\trivial}_{T_{\cV}}))^s\big) \\
 &=\cI\transpose\supertrace\big((\At^{\trivial}_{T_{\cV}})^s\big) ,
\end{align*}
which concludes the proof.
\end{proof}

Since the Todd cocycle can be expressed in terms of scalar Atiyah cocycles, 
we have the following immediate corollary.

\begin{corollary}\label{cor:local td}
The diagram 
\[ \begin{tikzcd}
\Tpoly{\bullet}(\cV) \arrow[r, "\ttodd^{\trivial}_{T_{\cV}}"] & \Tpoly{\bullet}(\cV) \\ 
C^\infty(\RR^m\times\kkf{0}{l})\otimes\Tpoly{\bullet}\kf{r} \arrow[u, "\cI", hook] 
\arrow[r, "\ttodd^{\can}_{\cF|_U}", swap] & 
C^\infty(\RR^m\times\kkf{0}{l})\otimes\Tpoly{\bullet}\kf{r} \arrow[u, "\cI", hook]
\end{tikzcd} \]
commutes. 
\newline
Here $\ttodd^{\can}_{\cF|_U}\in\prod_{k=0}^{\infty} C^\infty(\RR^m\times\kkf{0}{l})\otimes\Omega^k(\kkf{r}{0})$ 
is the Todd cocycle of the restriction to $U$ of the Fedosov Lie algebroid $\cF$ associated with the canonical connection 
while $\ttodd^{\trivial}_{T_\cV}\in\prod_{k=0}^{\infty}\Omega^k(\cV)$ is the 
Todd cocycle of the dg manifold $(\cV,Q)$ associated with the trivial connection 
as in Example~\ref{ex:trivial cntn}. 
\newline
The Todd cocycles act by contraction on the spaces of polyvector fields. 
\end{corollary}

Finally we have 
\begin{align*}
\cI\circ\big(\kont^f_{U,\varpi}\big)_1 
={} & \big(\kont^{\cV}_Q\big)_1\circ\cI && \text{(by Lemma~\ref{lem:Tokyo})} \\ 
={} & \hkr\circ\big(\ttodd^{\trivial}_{T_\cV}\big)^{\frac{1}{2}}\circ\cI && \text{(by Theorem~\ref{thm:Shoikhet})} \\ 
={} & \hkr\circ\cI\circ\big(\ttodd^{\can}_{\cF|_U}\big)^{\frac{1}{2}} && \text{(by Corollary~\ref{cor:local td})} \\ 
={} & \cI\circ\hkr\circ\big(\ttodd^{\can}_{\cF|_U}\big)^{\frac{1}{2}} && 
\end{align*}
Therefore, in every coordinate chart $U$ of $M$ over which the vector bundles $B$ and $L$ are trivial, we have 
\[ \big(\kont^f_{U,\varpi}\big)_1= \hkr\circ \big(\ttodd^{\can}_{\cF|_U}\big)^{\frac{1}{2}} .\] 
The proof of Proposition~\ref{thm:Florence} is complete. 

\subsection{Proof of Theorem~\ref{thm:main0}}

The difference between Equations~\eqref{twisted Todd} and~\eqref{corrected Todd} is the factor $e^{\half\trace\At^{\can}_\cF}$. 
We start with considering $\trace\At^{\can}_\cF\in\Gamma(\cF^\vee)_1$.

\begin{lemma}
We have \[ \trace \At^{\can}_\cF =d_{\cF}(\divergence X^\nabla) ,\] 
where $\divergence X^\nabla\in C^\infty (\cM)= \sections{L^\vee \otimes \hat{S}B^\vee}$ is the divergence of the formal vertical vector field $X^\nabla$. 
More explicitly, $\divergence X^\nabla=\sum_k \hhb{k} f_k.$
\end{lemma}

\begin{proof}
By Equation~\eqref{eq:local Atiyah matrix}, 
\[ \trace \At^{\can}_\cF = \sum_{i,k=1}^r \hhb{i}(\hhb{k}f_k) \chi_i = d_{\cF} (\divergence X^\nabla) .\qedhere\] 
\end{proof}

Furthermore, one has the following lemma.

\begin{lemma}\label{onhafhankelijkheid}
Let $\cA\to \cM$ be a dg Lie algebroid. 
Let $\mathcal{Q}$ denote the endomorphism of $\sections{\cA}$ induced by the dg structure of $\cA$,
and let $d_{\cA}$ denote the Chevalley--Eilenberg differential. 
If $\xi\in\sections{\cA\dual}$ satisfies $d_{\cA} \xi=0$ and $\mathcal{Q} \xi=0$, 
then the contraction with $\xi$ is a derivation of the differential Gerstenhaber algebra
$\big(\sections{\Lambda^\bullet \cA},\lie{\argument}{\argument},\mathcal{Q}\big)$.
\end{lemma}

Applying Lemma~\ref{onhafhankelijkheid} to the Fedosov dg Lie algebroid $\cF$ 
and the section $\trace\big(\At^{\can}_\cF\big)$ of $\cF\dual$ and noting that 
$\sections{\Lambda^\bullet\cF}\cong\tot\big( \sections{\Lambda^\bullet L\dual}\otimes_R\verticalTpoly{\bullet}\big)$ 
and $\mathcal{Q}=\schouten{\fedosov}{\argument}$, we obtain

\begin{corollary}\label{cor:premainthm}
\begin{enumerate}
\item The contraction by $\trace \At^{\can}_\cF$
is a derivation of the differential Gerstenhaber algebra $\tot\big(\sections{\Lambda^\bullet L\dual}\otimes_R
\verticalTpoly{\bullet}\big)$.
\item The contraction by $e^{\half \trace \At^{\can}_\cF}$
is an automorphism of the differential Gerstenhaber algebra $\tot\big(\sections{\Lambda^\bullet L\dual}\otimes_R
\verticalTpoly{\bullet}\big)$.
\end{enumerate}
\end{corollary}

Let $\Phia$ be the composition of the contraction operator $e^{\half \trace \At^{\can}_\cF}$, 
which is an automorphism of the dgla $\tot(\sections{\Lambda^\bullet L\dual}\otimes_R\verticalTpoly{\bullet})$ according to Corollary~\ref{cor:premainthm}, 
with the $L_\infty$ morphism
\[ \Phi : \tot(\sections{\Lambda^\bullet L\dual}\otimes_R\verticalTpoly{\bullet}) 
\to \tot(\sections{\Lambda^\bullet L\dual}\otimes_R\verticalDpoly{\bullet}) \]
constructed in Proposition~\ref{thm:renoPhi}. It follows from Proposition~\ref{thm:Florence} that the first Taylor coefficient of the $L_\infty$ morphism $\Phia$ is 
\[ \Phia_1 = \Phi_1 \circ e^{\half \trace \At^{\can}_\cF}
= \hkr \circ (\ttodd_\cF^{\can})^{\frac{1}{2}}\circ e^{\half \trace \At^{\can}_\cF} 
= \hkr\circ (\todd_\cF^{\can})^{\frac{1}{2}}. \]

To conclude the proof, we need the following lemma which follows from a straightforward computation.

\begin{lemma}\label{lem:s-hkr}
The diagram 
\[ \begin{tikzcd}
\tot\big(\sections{\Lambda^\bullet L\dual}\otimes_R\verticalTpoly{\bullet}\big)
\arrow[r, "\hkr"] \arrow[d, "\etendu{\sigma}", swap] &
\tot\big(\sections{\Lambda^\bullet L\dual}\otimes_R\verticalDpoly{\bullet}\big)
\arrow[d, "\etendu{\sigma}"] \\ 
\tot\big(\sections{\Lambda^\bullet A\dual}\otimes_R\Tpoly{\bullet}\big) \arrow[r, "\hkr", swap] &
\tot\big(\sections{\Lambda^\bullet A\dual}\otimes_R\Dpoly{\bullet}\big) 
\end{tikzcd} \]
commutes.
\end{lemma}

Together with Theorem~\ref{thm:HKR}, Theorem~\ref{thm:contractionTpol} and Theorem~\ref{thm:contractionDpol}, Lemma~\ref{lem:s-hkr} implies that the cochain map
\[ \hkr: \big(\tot(\sections{\Lambda^\bullet L\dual}\otimes_R\verticalTpoly{\bullet}), \schouten{Q}{\argument} \big) 
\to \big(\tot(\sections{\Lambda^\bullet L\dual}\otimes_R\verticalDpoly{\bullet}), \gerstenhaber{Q+m}{\argument} \big) \]
and thus $\Phia_1 = \hkr\circ (\todd_\cF^{\can})^{\frac{1}{2}}$ as well are quasi-isomorphisms.
The proof of Theorem~\ref{thm:main0} is thus complete.

\subsection{Proof of Theorem~\ref{thm:main}}

The following result is an immediate consequence of Proposition~\ref{sigma-atiyah-todd}.

\begin{corollary}
\label{cor:s-td}
The diagram  
\[ \begin{tikzcd}
\tot\big(\sections{\Lambda^{\bullet} L\dual}\otimes_R\verticalTpoly{\bullet}\big) \arrow[r, "(\todd^{\can}_\cF)^{\frac{1}{2}}"] \arrow[d, "\etendu{\sigma}", swap] 
& \tot\big(\sections{\Lambda^{\bullet} L\dual}\otimes_R\verticalTpoly{\bullet}\big) \arrow[d, "\etendu{\sigma}"] \\ 
\tot\big(\sections{\Lambda^{\bullet} A\dual}\otimes_R\Tpoly{\bullet}\big) \arrow[r, "(\todd^{\nabla}_{L/A})^{\frac{1}{2}}", swap] 
& \tot\big(\sections{\Lambda^{\bullet} A\dual}\otimes_R\Tpoly{\bullet}\big)
\end{tikzcd} \]
commutes.
\end{corollary}

According to Theorem~\ref{thm:contractionTpol}, we have a contraction
\[ \begin{tikzcd}[cramped]
\Big(\tot\big(\sections{\Lambda^\bullet A\dual}\otimes_R\Tpoly{\bullet}\big),d_A^{\Bott}\Big)
\arrow[r, "\etendu{\perturbed{\tau}}", shift left] &
\Big(\tot\big(\sections{\Lambda^\bullet L\dual}\otimes_R\verticalTpoly{\bullet}\big), \liederivative{Q}\Big)
\arrow[l, "\etendu{\sigma}", shift left]
\arrow["\etendu{\perturbed{h}}", loop,out=5,in=-5,looseness = 3]
\end{tikzcd} .\]
The r.h.s.\ $\tot\big(\sections{\Lambda^\bullet L\dual}\otimes_R\verticalTpoly{\bullet}\big)$ 
is a dgla while the l.h.s.\ $\tot\big(\sections{\Lambda^\bullet A\dual}\otimes_R \Tpoly{\bullet}\big)$ 
inherits an $L_\infty$ structure from the dgla structure of the r.h.s.\ by homotopy transfer theorem of $L_\infty$ algebras \cite{MR3276839,MR1932522,MR1950958, arXiv:1705.02880,MR2361936,arXiv:1807.03086, MR3318161,MR3323983}.

\begin{lemma}[Homotopy transfer of $L_\infty$ structures {\cite[Theorem~1.9]{arXiv:1705.02880}}]
\label{lem:ext} 
Let $(C,\eth)$ and $(K,d)$ be two cochain complexes and let
\[ \begin{tikzcd} 
(C,\eth) \arrow[r, "\tau", shift left] & (K,d) \arrow[l, "\sigma", shift left] \arrow[loop right, "h"] 
\end{tikzcd} \] 
be a contraction of $(K,d)$ onto $(C,\eth)$.
Given an $L_\infty$ algebra structure on $K$ (with $d$ as unary bracket), 
there exists a `transferred' $L_\infty$ algebra structure on $C$ and a pair of $L_\infty$ quasi-i\-so\-mor\-phisms 
$T:C\to K$ and $\Sigma:K\to C$ having the chain maps $\tau$ and $\sigma$ as respective first Taylor coefficients. 
\end{lemma}

Lemma~\ref{lem:ext} asserts the existence of an $L_\infty$ quasi-isomorphism $\Tau$ having $\etendu{\perturbed{\tau}}$ as first Taylor coefficient.

According to Theorem~\ref{thm:contractionDpol}, there is also a contraction
\[ \begin{tikzcd}[cramped]
\Big(\tot\big(\sections{\Lambda^\bullet A\dual}\otimes_R\Dpoly{\bullet}\big),\dau + \dHH \Big)  
\arrow[shift left]{r}{ \etendu{\perturbed{\tau}}} &
\Big(\tot\big(\sections{\Lambda^\bullet L\dual}\otimes_R\verticalDpoly{\bullet}\big), \gerstenhaber{Q + m}{\argument}\Big)
\arrow[shift left]{l}{ \etendu{\sigma} }
\arrow["\etendu{\perturbed{h}}", loop, out=5,in=-5,looseness = 3] 
\end{tikzcd} .\]
 
Again, the l.h.s.\ inherits an $L_\infty$ structure from the dgla structure of the r.h.s.\ by homotopy transfer. 
Lem\-ma~\ref{lem:ext} asserts the existence of an $L_\infty$ quasi-isomorphism $\etendu{\Sigma}$ having $\etendu{\sigma}$ as first Taylor coefficient. 
 
Consider the $L_\infty$ quasi-isomorphism 
\[ \Phii: \tot\big(\sections{\Lambda^\bullet A^\vee}\otimes_R\Tpoly{\bullet}\big)
\to \tot\big(\sections{\Lambda^\bullet A^\vee}\otimes_R\Dpoly{\bullet}\big) \]
obtained as the composition 
\[ \Phii=\etendu{\Sigma}\circ\Phia\circ\Tau \]
of the $L_\infty$ quasi-isomorphism  
$\Phia : \tot\big(\sections{\Lambda^\bullet L^\vee}\otimes_R\verticalTpoly{\bullet}\big)
\to \tot\big(\sections{\Lambda^\bullet L^\vee}\otimes_R\verticalDpoly{\bullet}\big)$ 
of Theorem~\ref{thm:main0} with the $L_\infty$ quasi-isomorphisms $\Tau$ and $\etendu{\Sigma}$: 
\[ \begin{tikzcd}
\tot\big(\sections{\Lambda^\bullet L\dual}\otimes_R\verticalTpoly{\bullet}\big) 
\arrow[r, "\Phia"] & 
\tot\big(\sections{\Lambda^\bullet L\dual}\otimes_R\verticalDpoly{\bullet}\big) 
\arrow[d, "\etendu{\Sigma}"] \\ 
\tot\big(\sections{\Lambda^\bullet A\dual}\otimes_R\Tpoly{\bullet}\big) 
\arrow[u, "\Tau"] \arrow[r, "\Phii", swap, dashed] & 
\tot\big(\sections{\Lambda^\bullet A\dual}\otimes_R\Dpoly{\bullet}\big) 
.\end{tikzcd} \]
Its first Taylor coefficient is 
\begin{align*}
\Phii_1 =&\ (\etendu{\Sigma})_1 \circ \Phia_1 \circ(\Tau)_1 & \\
=&\ \etendu{\sigma} \circ (\hkr\circ (\todd_\cF^{\can})^{\frac{1}{2}}) \circ \etendu{\perturbed{\tau}} && \text{(by Theorem~\ref{thm:main0})} \\
=&\ \hkr\circ \etendu{\sigma}\circ (\todd_\cF^{\can})^{\frac{1}{2}} \circ \etendu{\perturbed{\tau}} && \text{(by Lemma~\ref{lem:s-hkr})} \\
=&\ \hkr\circ (\todd^\nabla_{L/A})^{\frac{1}{2}} \circ \etendu{\sigma} \circ \etendu{\perturbed{\tau}} && \text{(by Corollary~\ref{cor:s-td})} \\
=&\ \hkr\circ (\todd^\nabla_{L/A})^{\frac{1}{2}} & 
\end{align*} 
This concludes the proof of Theorem~\ref{thm:main}. 

\section{Applications}
\label{applications}

\subsection{Complex manifolds}

Let $X$ be a complex manifold. 
Then $T_X\otimes\CC\cong T^{0,1}_X\bowtie T^{1,0}_X$ is a matched pair of Lie algebroids 
(see Section~\ref{section:dgla} or~\cite{MR1460632} for the definition of matched pairs).
Hence $(T_X\otimes\CC,T^{0,1}_X)$ is a Lie pair with quotient $T^{1,0}_X$. 
The flat $T^{0,1}_X$-connection on $T^{1,0}_X$ --- the Bott connection ---  
encodes the holomorphic vector bundle structure of $T^{1,0}_X$; 
the sections of $T^{1,0}_X$ which are flat w.r.t.\ the $T^{0,1}_X$-connection 
are precisely the holomorphic sections of $T^{1,0}_X$. 
In other words, the Chevalley--Eilenberg differential associated with the Bott representation of the Lie pair $(T_X\otimes\CC,T^{0,1}_X)$ 
is the Dolbeault operator
\[ \bar{\partial}: \Omega^{0, \bullet }(T^{1,0}_X)\to \Omega^{0, \bullet+1 }(T^{1,0}_X) .\]

\subsubsection{Atiyah and Todd classes of complex manifolds}

A torsion-free $T_X\otimes\CC$-connection $\nabla$ on $T^{1,0}_X$ extending the Bott $T^{0,1}_X$-connection 
is necessarily the sum $\nabla=\bar{\partial}+\nabla^{1,0}$ --- more precisely $d^\nabla=\bar{\partial}+d^{\nabla^{1,0}}$ 
--- of the Dolbeault operator and a torsion-free $T^{1,0}_X$-connection $\nabla^{1,0}$ on $T^{1,0}_X$, i.e.\ a $\CC$-bilinear map 
$\nabla^{1,0}:\sections{T^{1,0}_X}\times\sections{T^{1,0}_X}\to\sections{T^{1,0}_X}$ 
satisfying the usual connection axioms and the condition 
\[ \nabla_X Y-\nabla_Y X = \lie{X}{Y} ,\quad\forall\; X,Y\in\sections{T^{1,0}_X} .\] 

The Atiyah cocycle associated with such a connection $\nabla$ 
is the element $\atiyahcocycle\in\OO^{1,1}\big(\End(T_X^{1,0})\big)$ defined by 
\[ \atiyahcocycle(a; b) = \nabla_a \nabla_{b} - \nabla_b \nabla_a - \nabla_{\lie{a}{b}} ,\quad\forall\; a\in\sections{T_{X}^{0,1}},\; b\in\sections{T_X^{1,0}} .\]
Its cohomology class $\alpha_X\in H^{1,1}\big(X,\End(T_X^{1,0})\big)$ is independent of the choice of the connection $\nabla$ 
and is precisely the Atiyah class of the complex manifold $X$. 

The Todd cocycle associated with the connection $\nabla$ is 
\[ \todd^\nabla_{X}=\det\left(\frac{\atiyahcocycle}{1-e^{-\atiyahcocycle}}\right)
\in\bigoplus_{k=0}^\infty\Omega^{k,0}\big(\Lambda^k(T_X^{0,1})\dual\big) 
\cong\bigoplus_{k=0}^\infty\Omega^{k,k}(X) .\]
Its cohomology class 
\[ \Todd_{X}=\det\left(\frac{\atiyahclass_{X}}{1-e^{-\atiyahclass_{X}}}\right)
\in\bigoplus_{k=0}^\infty H^{k}\big(X,\Lambda^k(T_X^{0,1})\dual\big) 
\cong\bigoplus_{k=0}^\infty H^{k,k}(X) .\]
is independent of the choice of the connection $\nabla$ 
and is called the Todd class of the complex manifold $X$.

\subsubsection{Polyvector fields and polydifferential operators on 
complex manifolds}

Since $T^{0,1}_X\bowtie T^{1,0}_X$ is a matched pair, 
it follows from Corollary~\ref{Bari} that $\OO^{0,\bullet}\big(X,\Tpoly{\bullet}(X)\big)$ 
is a differential Gerstenhaber algebra with the Dolbeault operator 
$\bar{\partial}:\OO^{0,\bullet}\big(X,\Tpoly{\bullet}(X)\big)\to\OO^{0,\bullet+1}\big(X,\Tpoly{\bullet}(X)\big)$
as differential, the wedge product as associative multiplication 
and the natural extension (see Equation~\eqref{eq:Brussels1})
\begin{multline*} [\xi_1\otimes b_1, \xi_2\otimes b_2]
=\xi_1\wedge \xi_2 \otimes [b_1, b_2]
+\xi_1\wedge \nabla^{\Bott}_{b_1} \xi_2 \otimes b_2
- \nabla^{\Bott}_{b_2} \xi_1\wedge \xi_2 \otimes b_1 ,\\ 
\forall\; \xi_1,\xi_2\in\OO^{0,\bullet}(X),\; b_1,b_2\in\sections{T_X^{1,0}} 
,\end{multline*}
of the Lie bracket on $\sections{T_X^{1,0}}$ as graded Lie bracket. 

Similarly, $\OO^{0,\bullet}\big(X,\Dpoly{\bullet}(X)\big)$ is a dgla 
and its cohomology is a Gerstenhaber algebra (Corollary~\ref{thm:Naples}). 
Here the differential is $\bar{\partial}+\dHH$, 
where $\bar{\partial}:\OO^{0,\bullet}(X,\Dpoly{\bullet}(X))\to\OO^{0,\bullet+1}(X,\Dpoly{\bullet}(X))$ is the Dolbeault operator, 
while the associative multiplication and the graded Lie bracket are given by Proposition~\ref{pro:zurich}.
 
Note that $\big(\OO^{0,\bullet}(X,\Tpoly{\bullet}(X)),\bar{\partial}\big)$ is the Dolbeault resolution of the complex of sheaves 
\[ 0 \to \mathcal{O}_X \xto{0} \Theta_X \xto{0} \Lambda^2\Theta_X \xto{0} \Lambda^3\Theta_X \to \cdots \] 
of holomorphic polyvector fields over $X$, while $\big(\OO^{0,\bullet}(X,\Dpoly{\bullet}(X)),\bar{\partial}+\dHH\big)$ 
is the Dolbeault resolution of the complex of sheaves 
\[ 0 \to \mathcal{O}_X \to \mathscr{D}_X \xto{\hochschild} \mathscr{D}_X^{\otimes 2} \xto{\hochschild} \mathscr{D}_X^{\otimes 3} \to \cdots \] 
of holomorphic polydifferential operators over $X$.

As a result, for the Lie pair $(L=T_X\otimes\CC,A=T^{0,1}_X)$, the cohomology 
\[ \hypercohomology_{\CE}^{\bullet}(A,\Tpoly{\bullet})=\hypercohomology(\big(\OO^{0,\bullet}(X,\Tpoly{\bullet}(X)),\bar{\partial}\big) \] 
is isomorphic to the sheaf cohomology $\hypercohomology_{\sheaf}^\bullet(X,\Lambda^\bullet\Theta_X)$
while the cohomology
\[ \hypercohomology_{\CE}^{\bullet}(A,\Dpoly{\bullet})=\hypercohomology^{\bullet}\big(\OO^{0,\bullet}(X,\Dpoly{\bullet}(X)),\bar{\partial}+\dHH\big) \] 
is isomorphic to the Hochschild cohomology $HH^{\bullet}(X)\cong\Ext_{\cO_{X\times X}}^\bullet(\cO_{\Delta},\cO_{\Delta})$ 
(see~\cite{MR2141853, MR2923978, MR2732577}) of the complex manifold $X$.

\subsubsection{Formality theorem for complex manifolds}

Theorems~\ref{thm:main} and~\ref{KD-thm} imply the following two theorems. 

\begin{theorem}[Formality theorem for complex manifolds]
Let $X$ be a complex manifold. Choose a torsion-free $T^{1, 0}_X$-connection $\nabla^{1,0}$ on $T^{1, 0}_X$.
There exists an $L_\infty$ quasi-isomorphism
\[ \Phii: \OO^{0,\bullet}\big(X,\Tpoly{\bullet}(X)\big) \to \OO^{0,\bullet}\big(X,\Dpoly{\bullet}(X)\big) \] 
with first Taylor coefficient $\Phii_1$ satisfying the following two properties:
\begin{itemize}
\item $\Phii_1$ preserves the associative algebra structures up to homotopy;
\item $\Phii_1=\hkr\circ\big(\todd^{\bar{\partial}+\nabla^{1,0}}_X\big)^{\frac{1}{2}}$, 
where the square root of the Todd cocycle $\todd^{\bar{\partial}+\nabla^{1,0}}_X\in\bigoplus_{k=0}\OO^{k, k}(X)$ 
acts on $\OO^{0,\bullet}\big(X,\Tpoly{\bullet}(X)\big)$ by contraction.
\end{itemize}
\end{theorem}

\begin{theorem}[Kontsevich-Duflo theorem for complex manifolds]
For every complex manifold $X$, the composition
\[ \hkr\circ(\Todd_X)^{\frac{1}{2}}:\hypercohomology_{\sheaf}^\bullet(X,\Lambda^\bullet T_X)\to HH^\bullet(X) \]
is an isomorphism of Gerstenhaber algebras. 
It is understood that the square root of the Todd class 
\[ \Todd_X\in\bigoplus_{k=0}H^{k, k}(X)\cong H_{\sheaf}^k(X,\Omega_X^k) \]
acts on $\hypercohomology_{\sheaf}^\bullet(X,\Lambda^\bullet T_X)$ by contraction. 
\end{theorem}

The Kontsevich-Duflo theorem for complex manifolds is due 
to Kontsevich \cite{MR2062626} (for associative algebra structures). 
See~\cite{MR2646112} for a detailed proof including Gerstenhaber algebra structures.
See~\cite{MR2348030,MR3311762} for further developments. 

\subsection{Lie algebra pairs}

A Lie algebra pair is a Lie pair $(\frakg,\frakh)$, where $\frakg$ is a finite-dimensional Lie algebra 
and $\frakh$ is a Lie subalgebra of $\frakg$.

\subsubsection{Atiyah and Todd classes of Lie algebra pairs}

A $\frakg$-connection on $\frakg/\frakh$ is simply a bilinear map 
$\nabla:\frakg\times\frakg/\frakh\to\frakg/\frakh$. 
Its torsion is the linear map $T^\nabla:\Lambda^2\frakg\to\frakg/\frakh$ defined by 
\[ T^\nabla(X,Y)=\nabla_Xq(Y) -\nabla_Yq(X) -q(\lie{X}{Y}) ,\quad\forall\; X,Y\in\frakg .\] 
The map $q:\frakg\to\frakg/\frakh$ is the canonical projection. 

Let $\nabla$ be a $\frakg$-connection on $\frakg/\frakh$ which extends the Bott $\frakh$-connection: 
$\nabla^{\Bott}_{a}q(l)=q([a,l])$, for all $a\in\frakh$ and $l\in\frakg$. 
The Atiyah cocycle associated with $\nabla$ is the bilinear map 
\[ \atiyahcocycle: \frakh \otimes \frakg/\frakh \to \End(\frakg/\frakh) \] 
defined by 
\[ \atiyahcocycle(a; q(l)) = \nabla_a \nabla_{l} - \nabla_l \nabla_a - \nabla_{[a,l]}, \qquad\forall\; a\in\frakh,\; l\in\frakg .\] 
According to Proposition~\ref{Thm:atiyahclass}, the element 
$\atiyahcocycle\in\frakh\dual\otimes\frakh^\perp\otimes\End(\frakg/\frakh)$ 
is a Chevalley--Eilenberg 1-cocycle for the Lie algebra $\frakh$ 
with values in the $\frakh$-module $\frakh^\perp\otimes\End(\frakg/\frakh)$. 
Its cohomology class $\atiyahclass_{\frakg/\frakh}\in H_{\CE}^1\big(\frakh,\frakh^\perp\otimes\End(\frakg/\frakh)\big)$ 
is independent of the choice of $\frakg$-connection $\nabla$ and is called the Atiyah class of the Lie algebra pair $(\frakg,\frakh)$.

The Todd cocycle of the Lie algebra pair $(\frakg,\frakh)$ associated with the connection $\nabla$ is the Chevalley-Eilenberg cocycle
\[ \todd_{\frakg/\frakh}^{\nabla}=\det\left(\frac{\atiyahcocycle}{1-e^{-\atiyahcocycle}}\right)
\in \bigoplus_{k=0} \Lambda^k \frakh\dual\otimes \Lambda^k \frakh^\perp .\] 
The Todd class of the Lie algebra pair $(\frakg ,\frakh )$ is the corresponding Chevalley-Eilenberg cohomology class
\[ \Todd_{\frakg/\frakh} =\det\left(\frac{\atiyahclass_{\frakg/\frakh}}{1-e^{-\atiyahclass_{\frakg/\frakh}}}\right)
\in \bigoplus_{k=0} H_{\CE}^{k}(\frakh, \Lambda^k \frakh^\perp ) .\]

\subsubsection{Polyvector fields and polydifferential operators on Lie algebra pairs}

For a Lie algebra pair $(\frakg,\frakh)$, it follows from Corollary~\ref{Bari} and Corollary~\ref{thm:Naples} that both 
$\tot\big(\Lambda^\bullet\frakh\dual\otimes\Lambda^{\bullet+1}(\frakg/\frakh)\big)$ 
and $\tot\big(\Lambda^\bullet\frakh\dual\otimes\Big(\frac{\enveloping{\frakg}}{\enveloping{\frakg}\cdot\frakh}\Big)^{\otimes\bullet+1}\big)$ 
carry $L_\infty$ algebra structures unique up to  $L_\infty$ isomorphisms.
(Whenever $\frakg=\frakh\bowtie\frakm$ is a matched pair, $\frac{\enveloping{\frakg}}{\enveloping{\frakg}\cdot\frakh}$ 
is isomorphic to $\enveloping{\frakm}$ and the two $L_\infty$ algebras above are actually differential graded Lie algebras.) 

The quotient $\frakg / \frakh$ of a Lie algebra pair $(\frakg,\frakh)$ is an $\frakh$-module with the action 
\[ a \cdot q(l) = \nabla^{\Bott}_{a} q(l) = q([a,l]) ,\quad\forall\; a\in\frakh,\; l\in\frakg .\]
Again, $q:\frakg\to\frakg/\frakh$ is the canonical projection. 
This action extends by the Leibniz rule to an $\frakh$-action on $\Tpoly{\bullet}=\Lambda^{\bullet+1}(\frakg/\frakh)$.
Let $d_{\frakh}^{\Bott}:\Lambda^p \frakh\dual\otimes
\Lambda^{q +1} (\frakg/\frakh) \to \Lambda^{p+1} \frakh \dual \otimes
\Lambda^{q +1} (\frakg/\frakh)$ be the corresponding Chevalley--Eilenberg differential.
According to Corollary~\ref{Bari}, the space $\tot\big(\Lambda^\bullet \frakh\dual \otimes \Lambda^{\bullet+1}(\frakg/\frakh)\big)$ 
carries an $L_\infty$ algebra structure, unique up to $L_\infty$ isomorphism, 
with $d_{\frakh}^{\Bott}$ as unary bracket. 
Furthermore, when endowed with the wedge product, the hypercohomology $\hypercohomology_{\CE}^\bullet\big(\frakh,\Lambda^{\bullet+1}(\frakg/\frakh)\big)$ 
becomes a Gerstenhaber algebra.

Similarly, the Lie algebra $\frakh$ acts on $\Dpoly{0} = \frac{\enveloping{\frakg}}{\enveloping{\frakg} \cdot \frakh}$ 
by left multiplication and henceforth it acts on $\Dpoly{\bullet}=\Big(\frac{\enveloping{\frakg}}{\enveloping{\frakg}\cdot\frakh}\Big)^{\otimes\bullet+1}$ as well. 
The Chevalley--Eilenberg differential associated with this action is denoted 
\[ \dau:\Lambda^p \frakh\dual\otimes\Dpoly{q}\to \Lambda^{p+1} \frakh\dual \otimes\Dpoly{q} .\] 
Meanwhile, the Hochschild differential $\hochschild:\Dpoly{q} \to \Dpoly{q+1}$ extends to 
\[ \dHH:\Lambda^p\frakh\dual\otimes\Dpoly{q}\to\Lambda^p\frakh\dual\otimes\Dpoly{q+1} \]
by graded linearity. 
By Corollary~\ref{thm:Naples}, the graded vector space 
$\tot\big(\Lambda^\bullet\frakh\dual\otimes\Big(\frac{\enveloping{\frakg}}{\enveloping{\frakg}
\cdot\frakh}\Big)^{\otimes\bullet+1}\big)$ 
carries an $L_\infty$ algebra structure, unique up to $L_\infty$ isomorphism, 
with $\dau+\dHH$ as unary bracket.
When endowed with the cup product, the corresponding hypercohomology 
$\hypercohomology_{\CE}^\bullet\Big(\frakh,\Big(\frac{\enveloping{\frakg}}{\enveloping{\frakg}\cdot\frakh}\Big)^{\otimes\bullet+1}\Big)$ 
becomes a Gerstenhaber algebra.

The above $L_\infty$ algebra structures depend on the choice of 
a splitting of the short exact sequence $0\to\frakh\to\frakg\to\frakg/\frakh\to 0$
and a torsion-free $\frakg$-connection on $\frakg/\frakh$. 
However, different choices induce isomorphic $L_\infty$ algebras. 
Moreover, the first `Taylor coefficient' of the $L_\infty$ isomorphism is the identity map. 
Therefore, the Gerstenhaber algebra structures inherited by the cohomologies 
are in fact canonical \cite{arXiv:1901.04602}.

The natural map induced by skew-symmetrization (see Section~\ref{section:HKR})
\[ \hkr: \tot\big( \Lambda^\bullet \frakh^\vee \otimes \Lambda^{\bullet+1}(\frakg/\frakh) \big)
\to \tot \left( \Lambda^\bullet \frakh^\vee\otimes \Big(\tfrac{\enveloping{\frakg}}{\enveloping{\frakg} \cdot \frakh} \Big)^{\otimes \bullet +1} \right) \] 
is a quasi-isomorphism of cochain complexes.

\subsubsection{Formality theorem for Lie algebra pairs}

Theorems~\ref{thm:main} and~\ref{KD-thm} imply the following corollaries.

\begin{theorem}[Formality theorem for Lie algebra pairs] 
Let $(\frakg, \frakh)$ be a Lie algebra pair. Given a splitting of the short exact sequence 
$0 \to \frakh \to \frakg \to \frakg/\frakh \to 0$ and a torsion-free $\frakg$-connection $\nabla$ on $\frakg/\frakh$, 
there exists an $L_\infty$ quasi-isomorphism
\[ \Phii: \tot\big( \Lambda^\bullet \frakh^\vee \otimes \Lambda^{\bullet+1}(\frakg/\frakh) \big)
\to \tot \left( \Lambda^\bullet \frakh^\vee\otimes \Big(\tfrac{\enveloping{\frakg}}{\enveloping{\frakg} \cdot \frakh} \Big)^{\otimes \bullet +1} \right) \]
with first `Taylor coefficient' $\Phii_1$ satisfying the following two properties: 
\begin{enumerate}
\item $\Phii_1$
preserves the associative algebra structures (wedge and cup product, respectively) up to homotopy;
\item $\Phii_1 =\hkr\circ (\todd_{\frakg/\frakh}^{\nabla})^{\frac{1}{2}}$, where 
\[ (\todd_{\frakg/\frakh}^{\nabla})^{\frac{1}{2}}\in\bigoplus_{k=0}^\infty \Lambda^k\frakh\dual\otimes\Lambda^k\frakh^\perp 
= \bigoplus_{k=0}\Lambda^k\frakh\dual\otimes\Lambda^k(\frakg/\frakh)\dual \] 
acts on $\tot\big(\Lambda^\bullet\frakh^\vee\otimes\Lambda^{\bullet+1}(\frakg/\frakh)\big)$ by contraction.
\end{enumerate}
\end{theorem}

\begin{theorem}[Kontsevich-Duflo type theorem for Lie algebra pairs]
Given a Lie algebra pair $(\frakg, \frakh)$, the map
\[ \hkr\circ\Todd_{\frakg/\frakh}^{\frac{1}{2}}: \hypercohomology^\bullet_{\CE}\big(\frakh,\Lambda^{\bullet+1}(\frakg/\frakh)\big)
\xto{\cong} \hypercohomology^\bullet_{\CE}\left(\frakh,\Big(\tfrac{\enveloping{\frakg}}{\enveloping{\frakg} \cdot \frakh} \Big)^{\otimes \bullet +1}\right) \]
is an isomorphism of Gerstenhaber algebras. 
It is understood that the square root $\Todd_{\frakg/\frakh}^{\frac{1}{2}}$ of the Todd class 
$\Todd_{\frakg/\frakh}\in\bigoplus_{k=0}H_{\CE}^{k}(\frakh,\Lambda^k(\frakg/\frakh)\dual)$ 
acts on $\hypercohomology^\bullet_{\CE}(\frakh,\Lambda^{\bullet+1}(\frakg/\frakh))$ by contraction. 
\end{theorem}

\subsection{\texorpdfstring{$\mathfrak{g}$}{g}-manifolds}

In this section, we consider the formality theorem for a $\frakg$-manifold, i.e.\ a smooth manifold 
with a Lie algebra action (see~\cite{MR3650387} for more details). Let $M$ be a $\mathfrak{g}$-manifold 
with infinitesimal action $\mathfrak{g}\ni a\mapsto\hat{a}\in\XX(M)$. 
Every $\mathfrak{g}$-manifold $M$ determines in a canonical way 
a matched pair of Lie algebroids $(\mathfrak{g}\ltimes M)\bowtie T_M$ 
(see e.g. \cite[Example~5.5]{MR1460632} or \cite{MR1650045}).
The notation $\mathfrak{g}\ltimes M$ refers to the transformation Lie algebroid arising from the infinitesimal $\mathfrak{g}$-action on $M$.
Therefore, we can form a Lie pair $(L,A)$, where $L=(\mathfrak{g}\ltimes M)\bowtie T_M$ and $A=\mathfrak{g}\ltimes M$. 
In this case, the quotient $L/A$ is isomorphic to $T_M$ and the Bott $A$-connection on $L/A$ is the map 
\[ \nabla^{\Bott}:C^\infty(M,\mathfrak{g})\otimes\XX(M)\to\XX(M) \] defined by 
\[ \nabla^{\Bott}_{f\cdot a}X=f\cdot\lie{\hat{a}}{X} ,\]
for all $a\in\mathfrak{g}$, $f\in C^\infty(M)$, and $X\in\XX(M)$.

\subsubsection{Atiyah and Todd classes of $\mathfrak{g}$-manifolds}

It is not difficult to see that, for the Lie pair constituted of the Lie algebroid $L=(\mathfrak{g}\ltimes M)\bowtie T_M$ and its Lie subalgebroid $A=\mathfrak{g}\ltimes M$, 
a choice of $L$-connection on $L/A$ extending the Bott $A$-connection is essentially a choice of affine connection on $M$. 
Moreover, the torsion of the $L$-connection on $L/A$ reduces to the torsion of the corresponding affine connection on $M$. 

Given an affine connection $\nabla$ on $M$, the Atiyah 1-cocycle associated with $\nabla$ is the map 
\[ \atiyahcocycle: \mathfrak{g} \times\XX(M)\to \End_R \XX(M) \] defined by
\[ \atiyahcocycle(a,X)=\liederivative{\hat{a}}\circ\nabla_X - \nabla_X\circ\liederivative{\hat{a}} - \nabla_{\liederivative{\hat{a}}X} ,\]
for all $a\in\mathfrak{g}$ and $X\in\XX(M)$. 

Following Proposition~\ref{Thm:atiyahclass}, we prove the following

\begin{proposition} 
\begin{enumerate}
\item The Atiyah cocycle $\atiyahcocycle \in \mathfrak{g}\dual \otimes \sections{T_M\dual \otimes \End T_M}$ 
is a Chevalley--Eilenberg $1$-cocycle of the $\mathfrak{g}$-module $\sections{T_M\dual \otimes \End T_M}$.
\item The cohomology class $\atiyahclass_{M/\mathfrak{g}}\in H_{\CE}^1(\mathfrak{g}, \sections{T_M\dual \otimes \End T_M})$ 
of the $1$-cocycle $\atiyahcocycle$ does not depend on the choice of connection $\nabla$.
\end{enumerate}
\end{proposition}

The cohomology class $\atiyahclass_{M/\mathfrak{g}}$ is called the Atiyah class of the $\mathfrak{g}$-manifold $M$. 
It is the obstruction class to the existence of a $\mathfrak{g}$-invariant connection on $M$, 
i.e.\ an affine connection $\nabla$ on $M$ satisfying \[ [\hat{a},\nabla_X Y] = \nabla_{[\hat{a},X]} Y + \nabla_X [\hat{a},Y] \] 
for all $a\in\mathfrak{g}$ and $X,Y\in\XX(M)$. Note that if $\mathfrak{g}$ is a compact Lie algebra, $\atiyahclass_{M/\mathfrak{g}}$ 
vanishes since $\mathfrak{g}$-invariant connections always exist.

In the context of $\frakg$-manifolds, the Todd cocycle of a $\mathfrak{g}$-manifold $M$ 
is the Chevalley--Eilenberg cocycle
\[ \todd_{M/\mathfrak{g}}^{\nabla}=\det\left(\frac{\atiyahcocycle}{1-e^{-\atiyahcocycle}}\right)
\in\bigoplus_{k=0}\Lambda^{k}\mathfrak{g}^\vee\otimes\Omega^k(M) .\]
Its corresponding Chevalley--Eilenberg cohomology class is the \emph{Todd class} 
\[ \Todd_{M/\mathfrak{g}}=\det\left(\frac{\atiyahclass_{M/\mathfrak{g}}}{1-e^{-\atiyahclass_{M/\mathfrak{g}}}}\right)
\in\bigoplus_{k=0}H_{\CE}^{k}(\mathfrak{g},\Omega^k(M)) .\]
The spaces $\Omega^k (M)$, with $k\geq 0$, are endowed with their natural $\mathfrak{g}$-module structures.
Since the Lie algebra $\mathfrak{g}$ is finite dimensional, the above expression for the Todd class $\Todd_{M/\mathfrak{g}}$ reduces to a finite sum.

\subsubsection{Polyvector fields and polydifferential operators on $\mathfrak{g}$-manifolds}

The space of polyvector fields and the space of polydifferential operators on the Lie pair $\big((\mathfrak{g}\ltimes M)\bowtie T_M,\ \mathfrak{g}\ltimes M\big)$
are naturally isomorphic to $\tot\big(\Lambda^\bullet\mathfrak{g}\dual\otimes_{\KK}\Tpoly{\bullet}\big(M))$ 
and $\tot\big(\Lambda^\bullet\mathfrak{g}\dual\otimes_{\KK}\Dpoly{\bullet}\big(M))$, respectively.
Here $\Tpoly{\bullet}\big(M)$ denotes the space of ordinary polyvector fields on $M$, 
while $\Dpoly{\bullet}\big(M)$ denotes the space of ordinary polydifferential operators on $M$.
Since $(\mathfrak{g}\ltimes M)\bowtie T_M$ is a matched pair,
it follows from Proposition~\ref{pro:zurich} that both
$\tot\big(\Lambda^\bullet\mathfrak{g}\dual\otimes_{\KK}\Tpoly{\bullet}\big(M))$ 
and $\tot\big(\Lambda^\bullet\mathfrak{g}\dual\otimes_{\KK}\Dpoly{\bullet}\big(M))$ are dglas.

We proceed to describe these dgla structures. 
The $\mathfrak{g}$-action on $M$ and the Schouten bracket together determine a $\mathfrak{g}$-module structure on $\Tpoly{k}$ for each $k\geq -1$: 
\[ a \cdot \gamma = \schouten{\hat{a}}{\gamma} \qquad \forall \; a \in \mathfrak{g},\; \gamma \in \Tpoly{k}(M) .\]

Therefore, the complex with trivial differential 
\[ \cdots\to\Tpoly{k}(M)\xto{0}\Tpoly{k+1}(M)\to\cdots \] 
is a complex of $\mathfrak{g}$-modules and we obtain the differential Gerstenhaber algebra 
\[ \big(\tot(\Lambda^\bullet\mathfrak{g}\dual\otimes_{\KK}\Tpoly{\bullet}\big(M)),d_{\CE}+0,\schouten{\argument}{\argument},\wedge\big) ,\] 
whose graded Lie bracket and product are respectively defined by
\begin{align*}
\schouten{\alpha\otimes\mathcal{X}}{\beta\otimes\mathcal{Y}} &= (-1)^{q_1 p_2} \alpha\wedge\beta\otimes\schouten{\mathcal{X}}{\mathcal{Y}} \\
\intertext{and}
(\alpha\otimes\mathcal{X})\wedge(\beta\otimes\mathcal{Y}) &= (-1)^{q_1 p_2} (\alpha\wedge\beta)\otimes(\mathcal{X}\wedge\mathcal{Y})
,\end{align*}
for all $\alpha\otimes\mathcal{X}\in\Lambda^{p_1}\mathfrak{g}\dual\otimes_{\KK}\Tpoly{q_1}(M)$ 
and $\beta\otimes\mathcal{Y}\in\Lambda^{p_2}\mathfrak{g}\dual\otimes_{\KK}\Tpoly{q_2}(M)$.

Likewise, the $\mathfrak{g}$-action on $M$ and the Gerstenhaber bracket together determine a $\mathfrak{g}$-module structure on $\Dpoly{\bullet}$: 
\[ a\cdot\mu=\gerstenhaber{\hat{a}}{\mu} \qquad\forall\; a\in\mathfrak{g},\; \mu\in\Dpoly{\bullet}(M) .\] 
Since the Gerstenhaber bracket satisfies the graded Jacobi identity, 
the infinitesimal $\mathfrak{g}$-action on $\Dpoly{\bullet}(M)$ is compatible with the Hochschild differential. 
Consequently, the Hochschild cochain complex \[ \cdots\to\Dpoly{k}(M)\xto{\hochschild}\Dpoly{k+1}(M)\to\cdots \] 
is a complex of $\mathfrak{g}$-modules. 
Next, we endow $\Lambda^\bullet\mathfrak{g}\dual\otimes_{\KK}\Dpoly{\bullet}(M)$ 
with the differential $d_{\CE} + \dHH$, 
the cup product $\smile$, and the Gerstenhaber bracket $\gerstenhaber{\argument}{\argument}$ 
defined by 
\begin{align*}
(\alpha\otimes\xi)\smile(\beta\otimes\eta) &= (-1)^{q_1 p_2} (\alpha\wedge\beta)\otimes(\xi\smile\eta) \\
\gerstenhaber{\alpha\otimes\xi}{\beta\otimes\eta} &= (-1)^{q_1 p_2} \alpha\wedge\beta\otimes\gerstenhaber{\xi}{\eta} 
\end{align*}
for all $\alpha\otimes\xi\in\Lambda^{p_1}\mathfrak{g}\dual\otimes_{\KK}\Dpoly{q_1}(M)$ 
and $\beta\otimes\eta\in\Lambda^{p_2}\mathfrak{g}\dual\otimes_{\KK}\Dpoly{q_2}(M)$. 
It follows from Proposition~\ref{pro:zurich} that 
\[ \big(\tot(\Lambda^\bullet\mathfrak{g}\dual\otimes_{\KK}\Dpoly{\bullet}(M)),d_{\CE}+\dHH ,\gerstenhaber{\argument}{\argument}) \]
is a dgla whose cohomology $\hypercohomology^\bullet_{\CE}(\mathfrak{g},\Dpoly{\bullet}(M))$, 
endowed with the cup product and the Gerstenhaber bracket, is a Gerstenhaber algebra.

The $\Lambda^\bullet\mathfrak{g}$-linear extension 
$\hkr:\Lambda^\bullet\mathfrak{g}\dual\otimes_{\KK}\Tpoly{\bullet}(M) \to\Lambda^\bullet\mathfrak{g}\dual\otimes_{\KK}\Dpoly{\bullet}(M)$
of the classical HKR map of the smooth manifold $M$ is a quasi-isomorphism of cochain complexes 
but does not preserve the Lie structures on cohomologies. 

\subsubsection{Formality theorem for $\mathfrak{g}$-manifolds}

Theorem~\ref{thm:main} and Theorem~\ref{KD-thm} imply the following:

\begin{theorem}[Formality theorem for $\mathfrak{g}$-manifolds] 
\label{Curacao}
Given a $\mathfrak{g}$-manifold $M$ and an affine torsion-free connection $\nabla$ on $M$,
there exists an $L_\infty$ quasi-isomorphism $\Phii$ from the dgla
$\tot\big(\Lambda^\bullet \mathfrak{g}\dual \otimes_{\KK} \Tpoly{\bullet}(M)\big)$
to the dgla
$\tot \big(\Lambda^\bullet\mathfrak{g}\dual \otimes_{\KK} \Dpoly{\bullet}(M)\big)$
with first `Taylor coefficient' $\Phii_1$ satisfying the following two properties: 
\begin{enumerate} 
\item $\Phii_1$ is, up to homotopy, an isomorphism of associative algebras (and hence induces an isomorphism of associative algebras of the cohomologies); 
\item $\Phii_1=\hkr\circ (\todd_{M/\mathfrak{g}}^{\nabla})^{\frac{1}{2}}$, 
where $(\todd_{M/\mathfrak{g}}^{\nabla})^{\frac{1}{2}}\in\bigoplus_{k=0}\Lambda^{k}\mathfrak{g}\dual\otimes\Omega^k(M)$ 
acts on $\tot\big(\Lambda^\bullet\mathfrak{g}\dual\otimes_{\KK}\Tpoly{\bullet}(M)\big)$ by contraction. 
\end{enumerate}
\end{theorem}

\begin{theorem}[Kontsevich--Duflo type theorem for $\mathfrak{g}$-manifolds] 
Given a $\mathfrak{g}$-manifold $M$, the map
\[ \hkr\circ\Todd^{\frac{1}{2}}_{M/\mathfrak{g}}:
\hypercohomology^\bullet_{\CE}\big(\mathfrak{g},\Tpoly{\bullet}(M) \big)
\xrightarrow{\cong} \hypercohomology^\bullet_{\CE}\big(\mathfrak{g},\Dpoly{\bullet}(M) \big) \]
is an isomorphism of Gerstenhaber algebras.
It is understood that the square root $\Todd^{\frac{1}{2}}_{M/\mathfrak{g}}$ of the 
Todd class $\Todd_{M/\mathfrak{g}}\in\bigoplus_{k=0}H_{\CE}^{k}\big(\mathfrak{g},\Omega^k(M)\big)$ 
acts on $\hypercohomology^\bullet_{\CE}\big(\mathfrak{g},\Tpoly{\bullet}(M) \big)$ by contraction. 
\end{theorem}

\begin{remark}
In~\cite{MR2254188}, \v{S}evera studied quantization of ``Poisson actions up to homotopy'' 
using a version of Theorem~\ref{Curacao} applicable to trivial $\mathfrak{g}$-actions.
\end{remark}

\begin{remark}
To the best of our knowledge, the first construction of an $L_\infty$ quasi-isomorphism 
from the dgla $\tot\big(\Lambda^\bullet\mathfrak{g}\dual\otimes_{\KK}\Tpoly{\bullet}(M)\big)$ 
to the dgla $\tot\big(\Lambda^\bullet\mathfrak{g}\dual\otimes_{\KK}\Dpoly{\bullet}(M)\big)$ 
can be credited to Dolgushev~\cite[concluding remarks]{MR2102846}.
\end{remark}

\subsection{Foliations}

Let $\mathcal{F}$ be a regular foliation of a smooth manifold $M$. 
The tangent bundle of $\mathcal{F}$ is a subbundle of $T_M$, denoted $T_{\mathcal{F}}$, 
whose sections are closed under the Lie bracket of vector fields. 
Therefore, $(T_M,T_{\mathcal{F}})$ is a Lie pair. 
Its quotient $N_{\mathcal{F}}=T_M/T_{\mathcal{F}}$ is the normal bundle of the foliation $\mathcal{F}$. 
We have the short exact sequence of vector bundles 
\[ 0 \to T_{\mathcal{F}} \to T_M \xto{q} N_{\mathcal{F}} \to 0 .\] 
The Bott $T_{\mathcal{F}}$-connection on $N_{\mathcal{F}}$ is defined by 
\[ \nabla^{\Bott}_{a} q(l) = q([a,l]), \quad\forall\; a\in\sections{T_{\mathcal{F}}},\; l\in\XX(M) .\]
The Chevalley--Eilenberg Lie algebroid cohomology $H_{\CE}^\bullet(T_{\mathcal{F}},\mathfrak{M})$ 
with coefficients in a $T_{\mathcal{F}}$-module $\mathfrak{M}$ 
coincides exactly with the leafwise de Rham cohomology $H^\bullet_{\dR}(\mathcal{F},\mathfrak{M})$ 
of the foliation $\mathcal{F}$ with coefficients in the module $\mathfrak{M}$.

\subsubsection{Atiyah and Todd classes of foliations}

The torsion of a $T_M$-connection $\nabla$ on $N_{\mathcal{F}}$ is 
the bundle map $T^\nabla:\Lambda^2 T_{M}\to N_{\mathcal{F}}$ defined by 
\[ T^\nabla(X,Y)= \nabla_{X}q(Y) - \nabla_{Y}q(X) - q(\lie{X}{Y}) ,\] 
for all $X,Y\in\XX(M)$. 
The Atiyah cocycle associated with a torsion-free $T_M$-connection $\nabla$ on $N_{\mathcal{F}}$ 
is the bundle map $\atiyahcocycle:T_{\mathcal{F}}\otimes N_{\mathcal{F}}\to\End(N_{\mathcal{F}})$ 
--- or the corresponding section of $T_{\mathcal{F}}\dual\otimes T_{\mathcal{F}}^\perp\otimes\End(N_{\mathcal{F}})$ --- defined by 
\[ \atiyahcocycle\big(a;q(l)\big) = \nabla_a\nabla_l - \nabla_l\nabla_a - \nabla_{\lie{a}{l}}, \qquad\forall\; a\in\sections{T_{\mathcal{F}}},\; l\in\sections{T_M} .\]
According to Proposition~\ref{Thm:atiyahclass}, 
$\atiyahcocycle \in \sections{ T_{\mathcal{F}}\dual \otimes T_{\mathcal{F}}^\perp \otimes \End(N_{\mathcal{F}})}$ 
is a leafwise de Rham closed 1-form with values in the $T_{\mathcal{F}}$-module 
$T_{\mathcal{F}}^\perp \otimes \End(N_{\mathcal{F}})$. 
Its cohomology class 
$\atiyahclass_{\mathcal{F}}\in H_{\dR}^1({\mathcal{F}}, T_{\mathcal{F}}^\perp \otimes \End(N_{\mathcal{F}}))$ 
is independent of the chosen connection $\nabla$ and is called the Molino class of the foliation $\mathcal{F}$. 
It is an invariant of the foliation that was first introduced by Molino~\cite{MR0346814}.

The Todd cocycle of the foliation $\mathcal{F}$ associated with the connection $\nabla$ is the leafwise closed form 
\[ \todd_{\mathcal{F}}^{\nabla} =\det\left(\frac{\atiyahcocycle}{1-e^{-\atiyahcocycle }}\right) \in \bigoplus_{k=0} 
\sections{ \Lambda^k T_{\mathcal{F}}\dual\otimes \Lambda^k T_{\mathcal{F}}^\perp} .\]
The Todd class of the foliation $\mathcal{F}$ is the corresponding cohomology class
\[ \Todd_{\mathcal{F}} =\det\left(\frac{\atiyahclass_{\mathcal F}}{1-e^{-\atiyahclass_{\mathcal{F}}}}\right)
\in \bigoplus_{k=0}
H^k_{\dR}(\mathcal{F}, \Lambda^k T_{\mathcal{F}}^\perp ) .\]

\subsubsection{Transversal polyvector fields and transversal polydifferential operators on foliations}

It follows from Corollary~\ref{Bari} and Corollary~\ref{thm:Naples} applied to the Lie pair $(T_M,T_{\mathcal{F}})$
that both $\tot\big(\sections{\Lambda^\bullet T_\mathcal{F}\dual} \otimes_R \nTpoly{\bullet}\big)$ 
and $\tot\big(\sections{\Lambda^\bullet T_\mathcal{F}\dual}\otimes_R\nDpoly{\bullet}\big)$ 
can be endowed with $L_\infty$ algebra structures, unique up to $L_\infty$ isomorphism.
Here $\nTpoly{\bullet} = \sections{\Lambda^{\bullet+1}N_{\mathcal F}}$
can be considered as the space of polyvector fields transversal 
to the foliation $\mathcal{F}$ \cite{MR3277952,MR3300319}.
The unary bracket on $\tot\big(\sections{\Lambda^\bullet T_\mathcal{F}\dual}\otimes_R\nTpoly{\bullet}\big)$ 
is the leafwise de Rham differential $d_{\dR}$ with values in $\nTpoly{\bullet}$.
Similarly, $\nDpoly{\bullet} = \bigoplus_{k=-1}^\infty\nDpoly{k}$ 
can be considered as the space of polydifferential operators transversal to $\mathcal{F}$. 
Here $\nDpoly{-1}$ denotes the algebra $R$ of smooth functions on the manifold $M$, 
$\nDpoly{0}$ denotes the left $R$-module
$\frac{\enveloping{T_M}}{\enveloping{T_M}\cdot\sections{T_{\mathcal{F}}}}
\cong \frac{\mathcal{D}(M)}{\mathcal{D}(M)\cdot\sections{T_{\mathcal{F}}}}$
of `transverse differential operators,' 
and $\nDpoly{k}$ denotes the tensor product $\nDpoly{0}\otimes_R\cdots\otimes_R\nDpoly{0}$ 
of $(k+1)$ copies of the left $R$-module $\nDpoly{0}$. 
(Should there exist a foliation $\mathcal{F}'$ transverse to $\mathcal{F}$, 
the space $\nDpoly{0}$ would be isomorphic to 
the space $\enveloping{T_{\mathcal{F}'}}$ of differential operators in the direction of $\mathcal{F}'$.) 

For every $k\geq 0$, $\nDpoly{k}$ is naturally a left $\enveloping{T_{\mathcal{F}}}$-module 
and we can consider the associated leafwise de Rham differential 
\[ d_{\dR}: \sections{\Lambda^\bullet T_\mathcal{F}\dual}\otimes_R\nDpoly{k} 
\to \sections{\Lambda^{\bullet+1} T_\mathcal{F}\dual}\otimes_R\nDpoly{k} .\]

Since $\frac{\mathcal{D}(M)}{\mathcal{D}(M)\cdot\sections{T_{\mathcal{F}}}}$ is a coalgebra over $R$ 
with an associative comultiplication 
\[ \Delta: \tfrac{\mathcal{D}(M)}{\mathcal{D}(M)\cdot\sections{T_{\mathcal{F}}}} 
\to \tfrac{\mathcal{D}(M)}{\mathcal{D}(M)\cdot\sections{T_{\mathcal{F}}}} 
\otimes_R \tfrac{\mathcal{D}(M)}{\mathcal{D}(M)\cdot\sections{T_{\mathcal{F}}}} ,\] 
there is a Hochschild differential
\[ \hochschild:\nDpoly{k}\to\nDpoly{k+1} ,\]
which extends to a $\sections{\Lambda^\bullet T_\mathcal{F}\dual}$-graded linear operator of degree $+1$ 
on $\tot\big(\sections{\Lambda^\bullet T_\mathcal{F}\dual}\otimes_R\nDpoly{\bullet}\big)$ denoted $\dHH$. 
The unary bracket of the $L_\infty$ algebra structure on 
$\tot\big(\sections{\Lambda^\bullet T_\mathcal{F}\dual}\otimes_R\nDpoly{\bullet}\big)$ 
is $d_{\dR}+\dHH$.

The $L_\infty$ structures on $\tot\big(\sections{\Lambda^\bullet T_\mathcal{F}\dual}\otimes_R\nTpoly{\bullet}\big)$ 
and $\tot\big(\sections{\Lambda^\bullet T_\mathcal{F}\dual}\otimes_R\nDpoly{\bullet}\big)$ 
depend on the choice of a splitting of the short exact sequence $0 \to T_{\mathcal F} \to T_M \to N_{\mathcal F} \to 0$
and a torsion-free $T_M$-connection on $N_{\mathcal F}$ --- see~\cite{arXiv:1901.04602}. 
However, different choices induce isomorphic $L_\infty$ algebra structures.
Moreover, the first `Taylor coefficient' of the $L_\infty$ isomorphism is the identity map. 
Therefore, the resulting Gerstenhaber algebra structures on the cohomologies
$\hypercohomology^\bullet_{\dR}( \mathcal{F}, \nTpoly{\bullet})$ and
$\hypercohomology^\bullet_{\dR}( \mathcal{F}, \nDpoly{\bullet})$ are indeed canonical \cite{arXiv:1901.04602}.

According to Section~\ref{section:HKR}, the skew-symmetrization map 
$\nTpoly{\bullet}\to\nDpoly{\bullet}$
induces a quasi-iso\-mor\-phism of cochain complexes
\[ \hkr: \tot\big(\sections{\Lambda^\bullet T_\mathcal{F}\dual} \otimes_R \nTpoly{\bullet}\big) 
\to \tot\big(\sections{\Lambda^\bullet T_\mathcal{F}\dual} \otimes_R \nDpoly{\bullet}\big) .\] 

\subsubsection{Formality theorem for foliations}

Theorem~\ref{thm:main} and Theorem~\ref{KD-thm} imply the following

\begin{theorem}[Formality theorem for foliations] 
Let $\mathcal{F}$ be a regular foliation on a smooth manifold $M$. 
Given a splitting of the short exact sequence $0 \to T_{\mathcal F} \to T_M \to N_{\mathcal F} \to 0$
and a torsion-free $T_M$-connection $\nabla$ on $N_{\mathcal F}$, 
there exists an $L_\infty$ quasi-isomorphism
\[ \Phii: \tot\big(\sections{\Lambda^\bullet T_\mathcal{F}\dual} \otimes_R \nTpoly{\bullet}\big) 
\to \tot\big(\sections{\Lambda^\bullet T_\mathcal{F}\dual} \otimes_R \nDpoly{\bullet}\big) \] 
with first `Taylor coefficient' $\Phii_1$ satisfying the following two properties: 
\begin{enumerate}
\item $\Phii_1$
preserves the associative algebra structures (wedge and cup product, respectively) up to homotopy;
\item $\Phii_1 =\hkr\circ(\todd_{\mathcal F}^{\nabla})^{\frac{1}{2}}$, where 
$(\todd_{\mathcal F}^{\nabla})^{\frac{1}{2}}\in\bigoplus_{k=0}^\infty\sections{\Lambda^k T_{\mathcal{F}}\dual\otimes\Lambda^k T_{\mathcal{F}}^\perp}$
acts on $\tot\big(\sections{\Lambda^\bullet T_\mathcal{F}\dual}\otimes_R\nTpoly{\bullet}\big)$ by contraction.
\end{enumerate}
\end{theorem}

\begin{theorem}[Kontsevich-Duflo type theorem for foliations]
Given a regular foliation $\mathcal{F}$ on a smooth manifold $M$, the map
\[ \hkr\circ\Todd_{\mathcal F}^{\frac{1}{2}}: \hypercohomology^\bullet_{\dR}(\mathcal{F},\nTpoly{\bullet})
\xto{\cong} \hypercohomology^\bullet_{\dR}(\mathcal{F},\nDpoly{\bullet}) \]
is an isomorphism of Gerstenhaber algebras. 
It is understood that the square root $\Todd_{\mathcal{F}}^{\frac{1}{2}}$ 
of the Todd class $\Todd_{\mathcal{F}}\in\bigoplus_{k=0}H_{\dR}^{k}({\mathcal{F}},\Lambda^k T_{\mathcal{F}}^\perp)$ 
acts on $\hypercohomology^\bullet_{\dR}(\mathcal{F},\nTpoly{\bullet})$ by contraction. 
\end{theorem}

\appendix

\section{Fedosov dg Lie algebroids}

In this section, we recall basic ingredients needed to establish our main result (Theorem~\ref{thm:main}) in Section~\ref{Kabul}. 
For details, we refer the interested reader to~\cite{arXiv:1901.04602}.

\subsection{DGLAs associated to dg Lie algebroids}\label{section:dgla}

One can make sense of polyvector fields and polydifferential operators for a dg Lie algebroid just as one does for ordinary Lie algebroids.
Both give rise to dglas and their cohomology groups are in fact Gerstenhaber algebras. 
More precisely, a $(k+1)$-vector field on a dg Lie algebroid $\cL\to\cM$ is a 
section of the vector bundle $\Lambda^{k+1}\cL\to\cM$
while a $(k+1)$-differential operator is an element of $\enveloping{\cL}^{\otimes k+1}$,
the tensor product (as left $C^\infty (\cM )$-modules) of $k+1$ copies of the universal enveloping algebra $\enveloping{\cL}$.

It is clear that the differential $\mathcal{Q}:\sections{\cL}\to\sections{\cL}$ 
and the homological vector field $\mathcal{Q}:C^\infty(\cM)\to C^\infty(\cM)$ 
extend naturally to a degree $+1$ differential
$\mathcal{Q}: \sections{\Lambda^{k+1}\cL} \to \sections{\Lambda^{k+1}\cL}$ 
and the Lie algebroid structure on $\cL$ yields a Schouten bracket 
\[ \schouten{\argument}{\argument}: \sections{\Lambda^{u+1}\cL}\otimes
\sections{\Lambda^{v+1}\cL}\to \sections{\Lambda^{u+v+1}\cL} .\]
 
\begin{proposition}\label{pro:hongkong}
Let $\cL$ be a dg Lie algebroid over $\cM$.
\begin{enumerate}
\item When endowed with the differential $\mathcal{Q}$, the wedge product, and the Schouten bracket, 
the space of `polyvector fields' $\sections{\Lambda^{\bullet+1}\cL}$ is a differential Gerstenhaber algebra --- whence a dgla.
\item When endowed with the wedge product and the Schouten bracket, 
the cohomology $H^\bullet\big(\sections{\Lambda^{\bullet+1}\cL},\mathcal{Q})$ is a Gerstenhaber algebra.
\end{enumerate}
\end{proposition}

Adapting the definition given for Lie algebroids, 
one can define the universal enveloping algebra of a dg Lie algebroid. 
The universal enveloping algebra of a dg Lie algebroid $\cL\to\cM$
is a dg Hopf algebroid $\enveloping{\cL}$ over the cdga $\mathcal{R}:=C^\infty(\cM)$.
For each $k\geqslant 0$, the dg structure on the dg Lie algebroid
$\cL\to\cM$ determines a differential 
$\mathcal{Q}:\enveloping{\cL}^{\otimes k+1}\to\enveloping{\cL}^{\otimes k+1}$ 
of degree $+1$. 
A Hochschild coboundary differential
\[ \hochschild: \enveloping{\cL}^{\otimes k}\to 
\enveloping{\cL}^{\otimes k+1} \] 
and Gerstenhaber bracket 
\begin{equation}\label{eq:Gbraket}
\gerstenhaber{\argument }{\argument }: \UcL{u}\otimes_{\KK} \UcL{v}\to \UcL{u+v}
\end{equation} 
can be defined by the following explicit algebraic expressions:
\begin{multline}\label{hola}
\hochschild(u_1\otimes\cdots\otimes u_k) = 1\otimes u_1\otimes\cdots\otimes u_k
+ \sum_{i=1}^{k} (-1)^i u_1\otimes\cdots\otimes\Delta(u_i)\otimes\cdots\otimes u_k \\
+ (-1)^{k+1} u_1\otimes\cdots\otimes u_k\otimes 1,
\end{multline}
and, for $\phi\in\UcL{u}$ and $\phi\in\UcL{v}$,
\begin{equation}\label{hazmat} 
\gerstenhaber{\phi}{\psi} = \phi\star\psi - (-1)^{uv} \psi\star\phi \in\UcL{u+v} 
,\end{equation}
where $\phi\star\psi\in\UcL{u+v}$ is defined by
\[ \phi\star\psi = \sum_{k=0}^{u} (-1)^{kv}
d_0\otimes\cdots\otimes d_{k-1}\otimes (\Delta^v d_k)\cdot\psi\otimes d_{k+1}
\otimes\cdots\otimes d_u \]
if $\phi=d_0\otimes d_1 \otimes\cdots\otimes d_u$ for some $d_0,d_1,\dots,d_u\in\enveloping{\cL}$.

We refer the reader to~\cite{MR1815717} for the precise meaning 
of the product $(\Delta^v d_k)\cdot \psi$ in $\UcL{v}$ appearing in the last equation above. 

\begin{proposition}\label{pro:hongkong1}
Let $\cL$ be a dg Lie algebroid over $\cM$. 
\begin{enumerate}
\item When endowed with the differential $\mathcal{Q}+\hochschild$ 
and the Gerstenhaber bracket \eqref{eq:Gbraket}, 
$\enveloping{\cL}^{\otimes \bullet+1}$ is a dgla.
\item When endowed with the cup product (i.e.\ the tensor product $\otimes_{\mathcal{R}}$) 
and the Gerstenhaber bracket, the Hochschild cohomology $H^\bullet\big(\enveloping{\cL}^{\otimes \bullet+1},\mathcal{Q}+\hochschild\big)$,
is a Gerstenhaber algebra.
\end{enumerate}
\end{proposition}

Recall that, for a Lie pair $(L,A)$, if a splitting $j:B(=L/A)\to L$ of the short exact sequence $0\to A\to L\to B\to 0$ is given, 
whose image $j(B)$ happens to be a Lie subalgebroid of $L$, then $A$ and $B$ are said to form a \emph{matched pair of Lie algebroids}
--- see \cite{MR1460632} for more details. 
In such a situation, we write $L=A\bowtie B$ to highlight that $A$ and $B$ play symmetric roles as 
a pair of complementary Lie subalgebroids of the Lie algebroid $L$.

\begin{lemma}
If $A\bowtie B$ is a matched pair of Lie algebroids, 
then $(A[1]\oplus B,d^{\Bott}_A)$ is a dg Lie algebroid over $(A[1],\dace)$.
\end{lemma}

\begin{proof}
A classical result of Mackenzie~\cite{MR2831518} asserts that, if $A\bowtie B$ is a matched pair of Lie algebroids over a smooth manifold $M$, 
then \[ \begin{tikzcd}[cramped, sep=small] A\bowtie B \arrow[r] \arrow[d] & B \arrow[d] \\ A \arrow[r] & M \end{tikzcd} \] is a double Lie algebroid. 
Moreover, Gracia-Saz and Mehta~\cite{MR2581370} proved that, given a double Lie algebroid 
\[ \begin{tikzcd}[cramped, sep=small] D \arrow[r] \arrow[d] & B \arrow[d] \\ A \arrow[r] & M \end{tikzcd} ,\]
the graded vector bundle $D[1]\to A[1]$ is automatically a dg Lie algebroid. 
\end{proof}

Here the dg manifold structures on $(A[1]\oplus B, \dabott)$ and $(A[1], \dace)$ 
result from the Lie algebroid structures on $A\oplus B\to B$ and $A\to M$, respectively. 
In what follows, denote by $\cB$ the dg manifold $A[1]\oplus B$.
The space of sections of $\cB\to A[1]$ can be naturally identified with $\sections{\Lambda^\bullet A\dual\otimes B}$.
The bracket on $\sections{\Lambda^\bullet A\dual\otimes B}$ is defined 
in terms of the Bott $B$-connection on $\Lambda A\dual$ by
\begin{equation}
\label{eq:Brussels1}
[\xi_1\otimes b_1, \xi_2\otimes b_2]
=\xi_1\wedge \xi_2 \otimes [b_1, b_2]
+\xi_1\wedge \nabla^{\Bott}_{b_1} \xi_2 \otimes b_2
- \nabla^{\Bott}_{b_2} \xi_1\wedge \xi_2 \otimes b_1
\end{equation}
for all $\xi_1,\xi_2\in\sections{\Lambda^\bullet A\dual}$ and $b_1,b_2\in\sections{B}$, 
while the anchor map 
$\sections{\Lambda^\bullet A\dual \otimes B}\xto{\bar\rho}\Der(\Lambda^\bullet A\dual)$
is defined by
\begin{equation}\label{eq:Brussels2}
\bar{\rho}(\xi \otimes b)(\eta) = \xi\wedge\nabla^{\Bott}_b\eta
,\end{equation}
for all $\xi,\eta\in\sections{\Lambda^\bullet A\dual}$ and $b\in \sections{B}$.
Finally, the differential on the space of sections of $\cB\to A[1]$ is simply 
the Chevalley--Eilenberg differential 
$\dabott: \sections{\Lambda^\bullet A\dual\otimes B} \to \sections{\Lambda^{\bullet+1}A\dual\otimes B}$
corresponding to the Bott $A$-connection on $B$. 

According to Proposition~\ref{pro:hongkong}, the dg Lie algebroid $\cB\to A[1]$ 
induces a differential Gerstenhaber algebra structure on 
$\sections{\Lambda^{\bullet+1}\cB}\cong\sections{\Lambda^\bullet A\dual\otimes\Lambda^{\bullet+1}B}$.
Its differential is the Chevalley--Eilenberg differential 
\begin{equation}\label{eq:Vienna}
\dabott: \sections{\Lambda^\bullet A\dual\otimes\Lambda^{\bullet+1}B}
\to \sections{\Lambda^{\bullet+1} A\dual\otimes\Lambda^{\bullet+1}B}
\end{equation}
corresponding to the Bott $A$-connection on $\Lambda B$ 
and its Lie bracket is the Schouten bracket of the dg Lie algebroid $\cB\to A[1]$ 
--- essentially the extension of Equations~\eqref{eq:Brussels1} and~\eqref{eq:Brussels2} by the graded Leibniz rule.

Next, consider the universal enveloping algebra $\enveloping{\cB}$ of the dg Lie algebroid $\cB\to A[1]$,
which is a dg Hopf algebroid over $(\sections{\Lambda^\bullet A\dual},d_A)$.
It is clear that $\enveloping{\cB}\cong\sections{\Lambda^\bullet A\dual}\otimes_R\enveloping{B}$
and $\enveloping{\cB}^{\otimes k+1} \cong\sections{\Lambda^\bullet A\dual}\otimes_R\enveloping{B}^{\otimes k+1}$. 
Under this identification, the differential
$\mathcal{Q}:\enveloping{\cB}^{\otimes k+1}\to\enveloping{\cB}^{\otimes k+1}$
becomes the Cheval\-ley--Eilenberg differential
\begin{equation}
\label{eq:AUB}
\dau: \sections{\wedge^\bullet A\dual}\otimes_R\enveloping{B}^{\otimes k+1}
\to\sections{\wedge^{\bullet+1} A\dual}\otimes_R\enveloping{B}^{\otimes k+1}.
\end{equation}
Here the $A$-module structure on $\enveloping{B}$ follows from the canonical identification 
of $\enveloping{B}$ with $\frac{\enveloping{L}}{\enveloping{L}\sections{A}}$ 
--- the Lie algebroid $A$ acts on the latter by multiplication from the left --- 
and extends to an $A$-module structure on $\enveloping{B}^{\otimes k+1}$ in the natural way. 
As a consequence, the total differential $\mathcal{Q}+\hochschild$ on $\enveloping{\cB}^{\otimes \bullet+1}$ 
coincides with $\dau+\dHH$ on $\tot\big(\sections{\Lambda^\bullet A\dual}\otimes_R\enveloping{B}^{\bullet+1}\big)$.

The following proposition summarizes the discussion above: 
\begin{proposition}
\label{pro:zurich}
Suppose $A\bowtie B$ is a matched pair of Lie algebroids. 
\begin{enumerate}
\item When endowed with the differential $\dabott$ as in~\eqref{eq:Vienna} 
and the Schouten bracket defined by Equations~\eqref{eq:Brussels1}-\eqref{eq:Brussels2}, 
$\sections{\wedge^\bullet A\dual\otimes\wedge^{\bullet+1}B}$ is a differential Gerstenhaber algebra, whence a dgla.
\item When endowed with the wedge product and the Schouten bracket, 
the cohomology $\hypercohomology^\bullet_{\CE}(A,\Lambda^{\bullet+1}B)$ is a Gerstenhaber algebra.
\item When endowed with the differential $\dau+\dHH$ (see~\eqref{hola} and~\eqref{eq:AUB}) and the Gerstenhaber bracket,
$\tot\big(\sections{\Lambda^\bullet A\dual}\otimes_R\enveloping{B}^{\otimes\bullet+1}$ is a dgla. 
\item When endowed with the cup product and the Gerstenhaber bracket, 
the Hochschild cohomology $\hypercohomology^\bullet_{\CE}\big(A,{\enveloping{B}}^{\otimes\bullet+1}\big)$, 
i.e.\ the cohomology of the complex 
$\big(\tot(\Lambda^\bullet A\dual\otimes_R\enveloping{B}^{\otimes\bullet+1},\dau+\dHH\big)$, 
is a Gerstenhaber algebra.
\end{enumerate}
\end{proposition}

\begin{remark}
Note that the Gerstenhaber bracket on 
$\tot\big(\sections{\Lambda^\bullet A\dual}\otimes_R\enveloping{B}^{\otimes\bullet+1}\big)$ 
is \emph{not} the obvious extension of the Gerstenhaber bracket on $\enveloping{B}^{\otimes\bullet+1}$ 
obtained by tensoring with the commutative associative algebra $\sections{\Lambda^\bullet A\dual}$.
In fact, to write down an explicit formula --- which is quite involved --- one needs to use the 
Bott representation of $B$ on $\sections{\Lambda^\bullet A\dual}$.
\end{remark}

\subsection{Fedosov dg Lie algebroids}
\label{Fedosov dg abd}
Let $(L,A)$ be a Lie pair. 
We use the symbols $B$ to denote the quotient vector bundle $L/A$ and $r$ to denote its rank. 

Consider the endomorphism $\delta$ of the vector bundle 
$\Lambda^\bullet L\dual\otimes\hat{S}B\dual$ defined by 
\[ \delta(\omega\otimes\chi^J)=\sum_{m=1}^r \big(q^\top(\chi_m)\wedge\omega\big)\otimes J_m\,\chi^{J-e_m} ,\] 
for all $\omega\in\Lambda L\dual$ and $J\in\NN^r$. 
Here $\{\chi_k\}_{k=1}^r$ denotes an arbitrary local frame for the vector bundle $B\dual$, 
\[ \chi^J=\underset{J_1 \text{ factors}}{\underbrace{\chi_1\odot\cdots\odot\chi_1}}
\odot \underset{J_2 \text{ factors}}{\underbrace{\chi_2\odot\cdots\odot\chi_2}} 
\odot \cdots \odot \underset{J_r \text{ factors}}{\underbrace{\chi_r\odot\cdots\odot\chi_r}} \] 
if $J=(J_1,J_2,\cdots,J_r)$, 
the symbol $e_m$ denotes the multi-index $(0,\cdots,0,1,0,\cdots,0)$ having its single nonzero entry in $m$-th position, 
and $q^\top:B\dual\to L\dual$ is the map dual to the projection $q:L\to B$. 

The operator $\delta$ is a derivation of degree $+1$ of the 
graded commutative algebra $\sections{\Lambda^\bullet L^\vee 
\otimes\hat{S} B^\vee }$ and satisfies $\delta^2=0$. 
The resulting cochain complex
\[ \begin{tikzcd}[column sep=small] 
\cdots \arrow[r] & \Lambda^{n-1} L\dual \otimes\hat{S} B\dual \arrow[r, "\delta"] &
\Lambda^{n} L\dual \otimes\hat{S} B\dual \arrow[r, "\delta"] & 
\Lambda^{n+1} L\dual \otimes\hat{S} B\dual \arrow[r] & \cdots 
\end{tikzcd} \] 
deformation retracts onto the trivial complex
\[ \begin{tikzcd} 
\cdots \arrow[r] & \Lambda^{n-1} A\dual \arrow[r, "0"] & 
\Lambda^{n} A\dual \arrow[r, "0"] & \Lambda^{n+1} A\dual \arrow[r] & \cdots 
\end{tikzcd} \] 
Indeed, for every choice of splitting $i\circ p+j\circ q=\id_L$ of the short 
exact sequence 
\begin{equation}
\label{eq:9} 
\begin{tikzcd} 
0 \arrow[r] & A \arrow[r, "i"] & L \arrow[l, "p", bend left, dashed] \arrow[r, "q"] & B \arrow[r] \arrow[l, "j", bend left, dashed] & 0 
\end{tikzcd} 
\end{equation}
and its dual 
\[ \begin{tikzcd} 
0 \arrow[r] & B\dual \arrow[r, "q\transpose"] & L\dual \arrow[l, "j\transpose", bend left, dashed] \arrow[r, "i\transpose"] 
& A\dual \arrow[r] \arrow[l, "p\transpose", bend left, dashed] & 0 
\end{tikzcd} ,\]
the chain maps
\[ \sigma:\Lambda^\bullet L\dual \otimes\hat{S} B\dual \to\Lambda^\bullet A\dual \] 
and
\[ \tau:\Lambda^\bullet A\dual \to\Lambda^\bullet L\dual \otimes\hat{S} B\dual \]
respectively defined by 
\[ \sigma(\omega\otimes\chi^J)=\begin{cases} \omega\otimes\chi^J & 
\text{if } v=0 \text{ and } \abs{J}=0 \\ 
0 & \text{otherwise,} \end{cases} \]
for all $\omega\in p^\top(\Lambda^u A\dual)\otimes q^\top(\Lambda^v B\dual)$, 
and \[ \tau(\alpha)=p^\top(\alpha)\otimes 1 ,\] for all $\alpha\in\Lambda^\bullet(A\dual)$, 
satisfy
\[ \sigma\tau=\id \qquad \text{and} \qquad \id-\tau\sigma= h\delta+\delta h ,\]
where the homotopy operator 
\[ h:\Lambda^{\bullet} L\dual \otimes\hat{S} B\dual \to\Lambda^{\bullet-1} L\dual \otimes\hat{S} B\dual \] 
is defined by 
\[ h(\omega\otimes\chi^J)=\begin{cases} 
\frac{1}{v+\abs{J}}\sum_{k=1}^r (\iota_{j(\partial_k)}\omega)\otimes\chi^{J+e_k} & \text{if } v\geqslant 1 \\ 
0 & \text{if } v=0 \end{cases} \]
for all $\omega\in p^\top(\Lambda^u A\dual)\otimes q^\top(\Lambda^v B\dual)$. 
Here $\{\partial_k\}_{k=1}^r$ denotes the local frame for $B$ dual to $\{\chi_k\}_{k=1}^r$. 
Note that the operator $h$ is \emph{not} a derivation of the algebra 
$\sections{\Lambda^\bullet L^\vee \otimes\hat{S} B^\vee }$. 
Also, we note that $h\tau=0$, $\sigma h=0$, and $h^2 =0$.

\begin{lemma}
Let $(L,A)$ be a Lie pair and let $\nabla$ be an $L$-connection on $B$ extending the Bott $A$-connection. 
The torsion $\torsion$ of $\nabla$ vanishes (see Proposition~\ref{eq:dog}) if and only if $\delta d_L^{\nabla}+d_L^{\nabla}\delta=0$. 
\end{lemma}

Consider the four maps $\etendu{\delta}$, $\etendu{\sigma}$, $\etendu{h}$, and $\etendu{\tau}$
\[ \begin{tikzcd}
\sections{\Lambda^\bullet A\dual \otimes B} \arrow[shift right]{r}[swap]{\etendu{\tau}}
& \sections{\Lambda^{\bullet} L\dual \otimes\hat{S} B\dual \otimes B} \arrow[shift right]{l}[swap]{\etendu{\sigma}} \arrow[shift left]{r}{\etendu{\delta}} 
& \sections{\Lambda^{\bullet+1} L\dual \otimes\hat{S} B\dual \otimes B} \arrow[shift left]{l}{\etendu{h}} 
\end{tikzcd} \]
defined by 
\begin{align*} 
\etendu{\delta}(\omega\otimes\sigma\otimes b)&=\delta(\omega\otimes\sigma)\otimes b,
& \etendu{\sigma}(\omega\otimes\sigma\otimes b)&=\sigma(\omega\otimes\sigma)\otimes b,
\\ \etendu{h}(\omega\otimes\sigma\otimes b)&=h(\omega\otimes\sigma)\otimes b, 
& \etendu{\tau}(\alpha\otimes b)&=\tau(\alpha)\otimes b,
\end{align*} 
for all $\alpha\in\sections{\Lambda A^\vee}$, 
$\omega\in\sections{\Lambda L^\vee}$, $\sigma\in\sections{\hat{S} B\dual }$, and $b\in\sections{B}$. 

\begin{theorem}[\cite{arXiv:1605.09656}] \label{strawberry}
Let $(L,A)$ be a Lie pair with quotient $B=L/A$.
We interpret the sections of the bundle $L\dual\otimes\hat{S}B\dual\otimes B$ as derivations of the algebra 
$\sections{\Lambda^\bullet L\dual\otimes\hat{S}B\dual}$ in the natural way. 
Given a splitting of the short exact sequence \eqref{eq:9} 
and a torsion-free $L$-connection $\nabla$ on $B$, 
there exists a \emph{unique} derivation 
\[ X^\nabla\in\sections{L\dual\otimes\hat{S}^{\geqslant 2}B\dual\otimes B}, \] 
satisfying $\etendu{h}(X^\nabla)=0$ and such that the derivation 
$Q: \sections{\Lambda^\bullet L^\vee\otimes\hat{S}B^\vee} 
\to \sections{\Lambda^{\bullet+1}L^\vee\otimes\hat{S}B^\vee}$ 
defined by 
\[ Q=-\delta+d_L^\nabla+X^\nabla \]
satisfies $Q^2=0$. 
Moreover, writing $X_k$ for the component of $X^\nabla$ in $L\dual\otimes S^k B\dual\otimes B$, 
we have $X^\nabla=\sum_{k=2}^{\infty} X_k$ with 
\[ X_2=\etendu{h}(R^\nabla)=\etendu{h}(\widetilde{\atiyahcocycle})+\etendu{h}(R^\nabla_{0,2}) .\]
As a consequence, $(\cM=L[1]\oplus B,Q=-\delta+d_L^\nabla+X^\nabla)$ is a dg manifold, 
which we call a \emph{Fedosov dg manifold} associated with the Lie pair $(L,A)$.
\end{theorem}

\begin{proof}[Sketch of proof]
Suppose there exists such an $\xnabla$ and consider its decomposition $\xnabla=\sum_{k=2}^{\infty} X_k$, where 
$X_k\in\sections{L\dual\otimes\hat{S}^{k}B\dual\otimes B}$. 
Then $Q=-\delta+d_L^\nabla+\xtwo+\xhigher$ 
with $\xhigher=\sum_{k=3}^{\infty} X_k$ and 
\[ \begin{split} 
Q^2 =&\ \delta^2 
- \big(\delta d_L^\nabla+d_L^\nabla\delta\big) 
+ \big\{d_L^\nabla d_L^\nabla-\delta\xtwo-\xtwo\delta\big\} \\ 
&\ + \big\{d_L^\nabla\xnabla+\xnabla d_L^\nabla+{\xnabla}^2 
-\delta\xhigher-\xhigher\delta\big\} \\ 
=&\ \delta^2 - \lie{\delta}{d_L^\nabla} 
+ \big\{ R^\nabla-\lie{\delta}{\xtwo} \big\}
+ \big\{ \lie{d_L^\nabla+\tfrac{1}{2}\xnabla}{\xnabla}
- \lie{\delta}{\xhigher} \big\} 
.\end{split} \] 

For degree reasons, the requirement $Q^2=0$ is equivalent to the pair of equations 
\[ \lie{\delta}{\xtwo}=R^\nabla \qquad \text{and} \qquad \lie{\delta}{\xhigher}=\lie{d_L^\nabla+\tfrac{1}{2}\xnabla}{\xnabla} .\]

Note that $\etendu{\sigma}(\xtwo)=0$ and $\etendu{\sigma}(\xhigher)=0$, 
since $\xtwo,\xhigher\in\sections{L^\vee\otimes\hat{S}^{\geqslant 2}B^\vee\otimes B}$, 
and also that $\etendu{h}(\xtwo)=0$ and $\etendu{h}(\xhigher)=0$, as $\etendu{h}(\xnabla)=0$. 
Since $\etendu{\delta}\etendu{h}+\etendu{h}\etendu{\delta}=\id-\etendu{\tau}\etendu{\sigma}$, 
we obtain $\etendu{h}\etendu{\delta}(\xtwo)=\xtwo$ and 
$\etendu{h}\etendu{\delta}(\xhigher)=\xhigher$. 

It follows that 
\begin{gather*}
\xtwo=\etendu{h}\etendu{\delta}(\xtwo)=\etendu{h}(\lie{\delta}{\xtwo})=\etendu{h}(R^\nabla) \\ 
\intertext{while} 
\xhigher=\etendu{h}\etendu{\delta}(\xhigher)=\etendu{h}(\lie{\delta}{\xhigher})
=\etendu{h}\lie{d_L^\nabla+\tfrac{1}{2}\xnabla}{\xnabla}
.\end{gather*}
Projecting the latter equation onto 
$\sections{L^\vee\otimes\hat{S}^{k+1} B^\vee \otimes B}$, 
we obtain 
\[ X_{k+1} = \etendu{h}\Big(d_L^\nabla\circ X_k+X_k\circ d_L^\nabla+
\sum_{\substack{p+q=k+1 \\ 2\leqslant p,q\leqslant k-1}} X_p\circ X_q\Big), \qquad\text{for}\ k\geqslant 2 ,\] 
which shows that the higher terms of $\xnabla=\sum_{k=2}^{\infty} X_k$ 
can be computed iteratively starting from $X_2=\etendu{h}(R^\nabla)$. 
The derivation $X^\nabla$ is thus uniquely determined by the torsion-free connection $\nabla$. 
\end{proof}

The Fedosov dg manifold $(\cM,Q)$ of Theorem~\ref{strawberry} was also obtained independently by
Ba\-ta\-ki\-dis--Vo\-glai\-re~\cite{MR3724780} in the case of matched pairs.

\begin{remark}
When $L$ is the tangent bundle to a smooth manifold and $A$ is its trivial subbundle of rank $0$, 
Theorem~\ref{strawberry} reduces to a classical theorem of Emmrich--Weinstein~\cite{MR1327535} (see also~\cite{MR2102846}).
In the particular case of the Lie pair comprised of the complex Lie algebroids $L=T_X\otimes\CC$ and $A=T^{0,1}X$ associated with a complex manifold $X$, 
Theorem~\ref{strawberry} reduces to Theorem~5.9 in~\cite{MR2364075}.
\end{remark}

The identification of $C^{\infty}(M)$ with the subalgebra $\sections{\Lambda^0 L\dual\otimes S^0(B\dual)}$ 
of $C^{\infty}(\cM)=\sections{\Lambda^\bullet L\dual\otimes\hat{S}(B\dual)}$ determines a surjective submersion $\cM\onto M$. 
Let $\cF\to\cM$ denote the pullback of the vector bundle $B\to M$ through $\cM\onto M$. 
It is a graded vector bundle whose total space $\cF$ is the graded manifold with support $M$ 
associated with the graded vector bundle $L[1]\oplus B\oplus B\to M$. 
Its space of sections $\sections{\cF\to\cM}$ is canonically identified with 
$C^\infty(\cM)\otimes_{C^\infty(M)}\sections{B}=\sections{\Lambda^\bullet L\dual\otimes\hat{S}(B\dual)\otimes B}$. 
It is naturally a vector subbundle of $T_{\cM}\to\cM$; the inclusion $\sections{\cF\to\cM}\into\XX(\cM)$ 
takes the section $(\lambda\otimes\chi^J)\otimes\partial_k\in C^\infty(\cM)\otimes_{C^\infty(M)}\sections{B}$ of the vector bundle $\cF\to\cM$
to the derivation $\mu\otimes\chi^M\mapsto \lambda\wedge\mu\otimes M_k\chi^{J+M-e_k}$ of $C^{\infty}(\cM)$. 

\begin{proposition}[\cite{arXiv:1901.04602}]\label{pro:Rome}
The pullback $\cF\to\cM$ of the vector bundle $B\to M$ to the Fedosov dg manifold $(\cM,Q)$ 
is a dg Lie subalgebroid of the tangent dg Lie algebroid $T_{\cM}\to\cM$. 
\end{proposition}

In other words, $\cF$ is a dg foliation of the dg manifold $(\cM, Q)$. 
Each such dg Lie algebroid $\cF\to\cM$ is called a \emph{Fedosov dg Lie algebroid} associated with the Lie pair $(L,A)$.

\subsection{Dolgushev--Fedosov type quasi-isomorphisms on \texorpdfstring{$\Tpoly{\bullet,\bullet}$}{tensors} and \texorpdfstring{$\Tpoly{\bullet}$}{polyvector fields}}

Below we describe an extension of Dolgushev--Fedosov type quasi-isomorphisms \cite{MR2102846} to the context of Lie pairs. 
Actually, a stronger result holds: the quasi-isomorphisms are contractions. 

Set $\Tpoly{r,s}:=\sections{(B\dual)^{\otimes r}\otimes B^{\otimes s}}$ and let $\verticalTpoly{r,s}$ 
denote the space of formal vertical tensors of type $(r,s)$ on the vector bundle $B\to M$, 
i.e.\ \[ \verticalTpoly{r,s} = \sections{\hat{S}(B\dual)} \otimes_R \Tpoly{r,s} .\]
It is simple to see that
\[ \sections{\cM;(\cF\dual)^{\otimes r}\otimes\cF^{\otimes s}}
\cong\sections{\Lambda^\bullet L\dual}\otimes_R \verticalTpoly{r,s}
\cong\sections{\Lambda^\bullet L\dual\otimes \hat{S}B\dual}\otimes_R\Tpoly{r,s} .\] 

Since $Q$ is a homological vector field on the graded manifold $\cM=L[1]\oplus B$, 
the Lie derivative $\liederivative{Q}$ is a coboundary operator on the space
$T^{r,s}_{\cM}$ of tensors of type $(r,s)$ on $\cM$.
The Lie derivative $\liederivative{Q}$ stabilizes the subspaces of tensors of type $(r,s)$ 
``tangent to the dg Lie subalgebroid $\cF$ of $T_{\cM}$.'' 
\begin{lemma}
The subspace $\sections{\Lambda^\bullet L\dual}\otimes_R\verticalTpoly{r,s}$ of $T^{r,s}_{\cM}$ is stable under $\liederivative{Q}$.
\end{lemma}

By $\bsigma$, we denote the map $\sigma \otimes \id$:
\[ \sections{\Lambda^\bullet L\dual}\otimes_R\verticalTpoly{r,s}\cong 
\sections{\Lambda^\bullet L\dual\otimes \hat{S}B\dual}\otimes_R\Tpoly{r,s}
\xto{\sigma \otimes \id}\sections{\Lambda^\bullet A\dual}\otimes_R\Tpoly{r,s} .\] 

We have the following Dolgushev--Fedosov type quasi-isomorphism \cite{MR2102846}.

\begin{proposition}[\cite{arXiv:1901.04602}]\label{sigma for general tensor}
For each type $(r,s)$, the chain map
\[ \begin{tikzcd}[column sep=large]
\Big(\sections{\Lambda^\bullet L\dual}\otimes_R\verticalTpoly{r,s}, \liederivative{Q}\Big)\arrow[r, "\bsigma"]
& \Big(\sections{\Lambda^\bullet A\dual}\otimes_R\Tpoly{r,s},d_A^{\Bott}\Big)
\end{tikzcd} \]
is a quasi-isomorphism.
\end{proposition}

For polyvector fields, a stronger result was proved in~\cite{arXiv:1901.04602}.

Set $\Tpoly{k}:=\sections{\Lambda^{k+1} B}$ and let $\verticalTpoly{k}$ denote 
the space of formal vertical $(k+1)$-vector fields on $B$, i.e.\ 
\[ \verticalTpoly{k} = \sections{\hat{S}(B\dual)}\otimes_R \Tpoly{k} .\] 
Note that $\Tpoly{k}\subset\Tpoly{0,k+1}$ and $\verticalTpoly{k}\subset\verticalTpoly{0,k+1}$.
Then \[ \sections{\Lambda^{\bullet}L\dual}\otimes_R\verticalTpoly{k}
\cong \sections{\Lambda^{\bullet}L\dual\otimes \hat{S}B\dual}\otimes_R \Tpoly{k} .\] 

Denote by $\bsigma$ the map 
\[ \sections{\Lambda^{\bullet}L\dual}\otimes_R\verticalTpoly{k}
\cong \sections{\Lambda^{\bullet}L\dual\otimes \hat{S}B\dual}\otimes_R \Tpoly{k} 
\xto{\sigma\otimes\id}\sections{\Lambda^\bullet A\dual}\otimes_R \Tpoly{k} \] 

\begin{theorem}[\cite{arXiv:1901.04602}]
\label{thm:contractionTpol}
There exists a contraction
\begin{equation}\label{Yerevan} 
\begin{tikzcd}[cramped]
\Big(\tot\big(\sections{\Lambda^\bullet A\dual}\otimes_R
\Tpoly{\bullet}\big)
,d_A^{\Bott}\Big)
\arrow[r, " \etendu{\perturbed{\tau}}", shift left] &
\Big(\tot\big(\sections{\Lambda^\bullet L\dual}\otimes_R
\verticalTpoly{\bullet}\big), \liederivative{Q}\Big)
\arrow[l, " \etendu{\sigma}", shift left] 
\arrow[loop, "\etendu{\perturbed{h}}",out=5,in=-5,looseness = 3]
\end{tikzcd} 
\end{equation}
\end{theorem}

Consider the Fedosov dg Lie algebroid $\cF\to\cM$ of Section~\ref{Fedosov dg abd}.
It is clear that
\[ \sections{\cM; \Lambda^k\cF}\cong \sections{\Lambda^\bullet L\dual}\otimes_R \verticalTpoly{k} .\] 

Applying Proposition~\ref{pro:hongkong} to the dg Lie subalgebroid $\cF$ of $T_{\cM}$, we obtain 
\begin{proposition}\label{pro:Lyon}
\begin{enumerate}
\item Since the subspace $\sections{\Lambda^\bullet L\dual}\otimes_R\verticalTpoly{k}$ 
of the space $T_{\poly}^{k}(\cM)$ of $(k+1)$-vector fields on $\cM=L[1]\oplus B$ 
is stable under $\liederivative{Q}$, we obtain a cochain complex 
\[ \begin{tikzcd}
\cdots \arrow[r] & 
\sections{\Lambda^{u}L\dual}\otimes_R\verticalTpoly{k} 
\arrow[r, "\liederivative{Q}"] & 
\sections{\Lambda^{u+1}L\dual}\otimes_R\verticalTpoly{k} 
\arrow[r] & \cdots 
\end{tikzcd} \] for each $k\geqslant -1$. 
\item The total complex 
$\Big(\tot\big(\sections{\Lambda^\bullet L\dual}\otimes_R\verticalTpoly{\bullet}\big),\liederivative{Q}\Big)$
is a differential Gerstenhaber algebra, whence a dgla.
\end{enumerate}
\end{proposition}

It follows from the homotopy transfer theorem for $L_\infty$ algebras (see Lemma~\ref{lem:ext}) 
applied to the contraction \eqref{Yerevan} that the dgla structure carried by 
$\tot\big(\sections{\Lambda^\bullet L\dual}\otimes_R\verticalTpoly{\bullet}\big)$ 
determines an $L_\infty$ algebra structure on 
$\tot\big(\sections{\Lambda^\bullet A\dual}\otimes_R\Tpoly{\bullet}\big)$.
Moreover, since the retraction $\etendu{\sigma}$ intertwines the associative algebra structures, 
we immediately obtain the following corollary of Proposition~\ref{pro:Lyon}.

\begin{corollary}[\cite{arXiv:1901.04602}] \label{Bari}
Given a Lie pair $(L,A)$, each choice of a splitting $j:B\to L$ of the short exact sequence of vector bundles $0\to A \to L \to B \to 0$
and of a torsion-free $L$-connection $\nabla$ on $B$ 
determines 	
\begin{enumerate}
\item an $L_\infty$ algebra structure on $\tot\big(\sections{\Lambda^\bullet A\dual}\otimes_R \Tpoly{\bullet}\big)$ 
with the operator $d_A^{\Bott}$ as unary bracket
\item and a Gerstenhaber algebra structure on $\hypercohomology^\bullet_{\CE}(A,\Tpoly{\bullet})$, 
the cohomology of the complex \[ \Big(\tot\big(\sections{\Lambda^\bullet A\dual}\otimes_R\Tpoly{\bullet}\big),d_A^{\Bott}\Big) .\]
\end{enumerate}
\end{corollary}

A priori, the $L_\infty$ algebra structure on $\tot\big(\sections{\Lambda^\bullet A\dual}\otimes_R\Tpoly{\bullet}\big)$ 
in Corollary~\ref{Bari} is not canonical; it depends on a choice of `Dolgushev--Fedosov type' replacement for the complex 
$\big(\tot\big(\sections{\Lambda^\bullet A\dual}\otimes_R\Tpoly{\bullet}\big),\dabott \big)$ via a Fedosov dg Lie algebroid $\cF\to\cM$.
The construction of the Fedosov differential involves the choice of a torsion-free connection $\nabla:\sections{L}\times\sections{B}\to\sections{B}$ 
and a splitting $j:B\to L$ of the short exact sequence of vector bundles $0\to A \to L \to B \to 0$.
However, different choices yield isomorphic $L_\infty$ algebra structures
on $\tot\big(\sections{\Lambda^\bullet A\dual}\otimes_R\Tpoly{\bullet}\big)$. 
Hence, we obtain the following improvement on Corollary~\ref{Bari}:
\begin{theorem}[\cite{arXiv:1901.04602}]
\label{Bari2}
Let $(L,A)$ be a Lie pair. 
\begin{enumerate}
\item The space $\tot\big(\sections{\Lambda^\bullet A\dual}\otimes_R \Tpoly{\bullet}\big)$ admits an $L_\infty$ algebra structure 
with the operator $\dabott$ as unary bracket. This $L_\infty$ algebra structure is unique up an $L_\infty$ isomorphism 
having the identity map as linear part.
\item The corresponding cohomology group $\hypercohomology_{\CE}(A,\Tpoly{\bullet})$ admits a canonical 
Gerstenhaber algebra structure.
\end{enumerate}
\end{theorem}

Moreover, when the Lie pair happens to be a matched pair, the transferred $L_\infty$ algebra structure on
$\tot\big(\sections{\Lambda^\bullet A\dual}\otimes_R\Tpoly{\bullet}\big)$ is precisely the dgla structure described in Proposition~\ref{pro:zurich}.
\begin{proposition}
Under the hypotheses of Corollary~\ref{Bari} and the additional assumption that $j(B)$ is a Lie subalgebroid of $L$ --- 
i.e.\ $L=A\bowtie B$ is a matched pair --- the $L_\infty$ algebra $\tot\big(\sections{\Lambda^\bullet A\dual}\otimes_R\Tpoly{\bullet}\big)$ 
and the Gerstenhaber algebra $\hypercohomology^\bullet_{\CE}(A,\Tpoly{\bullet})$ of Corollary~\ref{Bari} coincide respectively 
with the dgla $\tot\big(\sections{\Lambda^\bullet A\dual\otimes\Lambda^{\bullet+1}B}\big)$
and the Gerstenhaber algebra $\hypercohomology^\bullet_{\CE}(A,\Lambda^{\bullet+1}B)$ of Proposition~\ref{pro:zurich}.
\end{proposition}

\subsection{Dolgushev--Fedosov type quasi-isomorphism on \texorpdfstring{$\Dpoly{\bullet}$}{polydifferential operators}}

Let $C^n:=\Hom_{\KK}(\salgebra^{\otimes n+1},\salgebra)$ denote the space of Hochschild $n$-cochains of the algebra $\salgebra:=C^\infty(L[1]\oplus B)$.
The Gerstenhaber bracket of two cochains $\phi\in C^u$ and $\psi\in C^v$ is the cochain 
\[ \gerstenhaber{\phi}{\psi} = \phi\star\psi - (-1)^{uv} \psi\star\phi \in C^{u+v}\] 
where $\phi\star\psi\in C^{u+v}$ is defined by 
\[ (\phi\star\psi ) (a_0\otimes a_1\otimes \cdots\otimes a_{u+v}) = \sum_{k=0}^{u} (-1)^{kv}
\phi\big(a_0\otimes\cdots\otimes a_{k-1}\otimes 
\psi(a_k\otimes\cdots\otimes a_{k+v})\otimes a_{k+1+v} 
\otimes \cdots\otimes a_{u+v} \big) ,\] 
for all $a_0,a_1,\dots,a_{u+v}\in\salgebra$. 
The Gerstenhaber bracket satisfies the graded Jacobi identity. 
Since the multiplication $m$ in $C^\infty(L[1]\oplus B)$ is associative, we have $\gerstenhaber{m}{m}=0$ 
and the standard Hochschild coboundary operator $\gerstenhaber{m}{\argument}$ turns $C^\bullet$ into a cochain complex. 

The space $D_{\poly}^{\bullet}(L[1]\oplus B)$ of polydifferential operators on $L[1]\oplus B$ 
is a subspace of $C^\bullet$ closed under the Gerstenhaber bracket. 
Note that $Q\in\XX(L[1]\oplus B)\subset D_{\poly}^{0}(L[1]\oplus B)$ and $m\in D_{\poly}^{1}(L[1]\oplus B)$. 

\begin{lemma} 
We have $\gerstenhaber{Q+ m}{\argument}^2=0$.
\end{lemma}

\begin{proof}
We have $\gerstenhaber{m}{m}=0$ since the multiplication $m$ is associative, 
$\gerstenhaber{Q}{m}=0$ since $Q$ is a derivation of $m$, 
and $\gerstenhaber{Q}{Q}=0$ since $Q$ is a homological vector field. 
The conclusion follows from the Jacobi identity.
\end{proof}

Let $\verticalDpoly{k}$ denote the space of formal vertical 
$(k+1)$-polydifferential operators on the vector bundle $B$, 
and let $\verticalDpoly{\bullet}=
\bigoplus_{k=-1}^{\infty}\verticalDpoly{k}$.
Set $\mathscr{S}=\sections{\hat{S}(B\dual)}$. 
There exists a canonical isomorphism 
\[ \begin{tikzcd}
\sections{\hat{S}(B\dual)\otimes
\underset{k+1\text{ factors}}{\underbrace{S(B)\otimes\cdots\otimes S(B)}}} 
\arrow[r, "\varphi", "\cong"'] &
\verticalDpoly{k} 
\end{tikzcd} .\]
To $\chi^I\otimes\partial^{J_0}\otimes
\cdots\otimes\partial^{J_k}\in \sections{\hat{S}(B\dual)\otimes
\underset{k+1\text{ factors}}{\underbrace{S(B)\otimes\cdots\otimes S(B)}}}$, 
the isomorphism $\varphi$ associates the poly\-differential operator 
\[ \mathscr{S}^{\otimes k+1}\ni 
\chi^{M_0}\otimes\cdots\otimes\chi^{M_k}\longmapsto 
\chi^I\cdot\partial^{J_0}(\chi^{M_0})
\cdots\partial^{J_k}(\chi^{M_k})
\in\mathscr{S} .\]

The algebra of functions $C^\infty(L[1]\oplus B)$ is a module over its subalgebra 
$\sections{\Lambda^\bullet L}\equiv\sections{\Lambda^\bullet L\dual\otimes S^0(B\dual)}$. 
The subspace of $D_{\poly}^{\bullet}(L[1]\oplus B)$ comprised of all 
$\sections{\Lambda^\bullet L\dual}$-multilinear polydifferential operators 
is easily identified to 
$\tot\big(\sections{\Lambda^\bullet L\dual}\otimes_R\verticalDpoly{\bullet}\big)$.
It is simple to see that the universal enveloping algebra
$\enveloping{\cF}$ of the Fedosov dg Lie algebroid $\cF\to\cM$ is naturally 
identified 
with $\sections{\Lambda^\bullet L\dual}\otimes_R\verticalDpoly{0}$,
which is a dg Hopf algebroid over $C^\infty (\cM)\cong 
\sections{\Lambda^\bullet L\dual \otimes\hat{S} B\dual }$.
Moreover, $\enveloping{\cF}$ is a dg Hopf subalgebroid of $D_{\poly}^{0}(L[1]\oplus B)$.
Note that
\[ \enveloping{\cF}^{\otimes k+1}\cong \sections{\Lambda^\bullet L\dual}
\otimes_R\verticalDpoly{k} .\] 

Thus, as a consequence of Proposition~\ref{pro:hongkong1}, we have the following

\begin{proposition}
\label{cotedor}
\begin{enumerate}
\item The subspace $\tot\big(\sections{\Lambda^\bullet L\dual}\otimes_R\verticalDpoly{\bullet}\big)$ 
of $D_{\poly}^{\bullet}(L[1]\oplus B)$ is stable under the Hoch\-schild coboundary operator $\gerstenhaber{Q+ m}{\argument}$. 
\item The triple $\big( \tot\big(\sections{\Lambda^\bullet L\dual}\otimes_R\verticalDpoly{\bullet}\big), \gerstenhaber{Q+ m}{\argument}, \gerstenhaber{}{} \big)$ is a dgla.
\item The cohomology group $H^\bullet \Big(\tot\big(\sections{\Lambda^\bullet L\dual}\otimes_R\verticalDpoly{\bullet}\big), \gerstenhaber{Q+ m}{\argument}\Big)$
is a Gerstenhaber algebra.
\end{enumerate}
\end{proposition}

Now consider the map
\[ \bsigma: \sections{\Lambda^{u}L\dual}\otimes_R\verticalDpoly{v}
\to \sections{\Lambda^{u}A\dual}\otimes_R\Dpoly{v} \] 
defined by the commutative diagram 
\[ \begin{tikzcd}[row sep=tiny]
\sections{\Lambda^{u}L\dual}\otimes_R\verticalDpoly{v}
\arrow[dr, "\bsigma"] & \\
& \sections{\Lambda^{u}A\dual}\otimes_R\Dpoly{v} \\
\sections{\Lambda^{u}L\dual\otimes\hat{S}B\dual
\otimes (SB)^{\otimes v+1}}
\arrow[uu, "\id\otimes\varphi", "\equiv"']
\arrow[ur, swap, "\sigma\otimes\pbw^{\otimes v+1}"] &
\end{tikzcd} .\]

The map $\bsigma$ is a quasi-isomorphism of Dolgushev--Fedosov type similar to the classical Fedosov resolution 
of the polydifferential operators on a smooth manifold obtained by Dolgushev \cite{MR2102846}.
Indeed, we have the following

\begin{theorem}[\cite{arXiv:1901.04602}]
\label{thm:contractionDpol}
There exists a contraction
\begin{equation}\label{Karachi}
\begin{tikzcd}[cramped]
\Big(\tot\big(\sections{\Lambda^\bullet A\dual}\otimes_R\Dpoly{\bullet}\big),\dau+\dHH \Big)
\arrow[r, shift left, "\etendu{\perturbed{\tau}}"] & 
\Big(\tot\big(\sections{\Lambda^\bullet L\dual}\otimes_R\verticalDpoly{\bullet}\big),\gerstenhaber{Q+ m}{\argument}\Big)
\arrow[l, shift left, "\etendu{\sigma}"] 
\arrow[loop, "\etendu{\perturbed{h}}", out=5,in=-5,looseness = 3]
\end{tikzcd} 
\end{equation}
\end{theorem}

It follows from the homotopy transfer theorem for $L_\infty$ algebras (see Lemma~\ref{lem:ext}) 
applied to the contraction \eqref{Karachi} that the dgla structure carried by 
$\tot\big(\sections{\Lambda^\bullet L\dual}\otimes_R\verticalDpoly{\bullet}\big)$ 
determines an $L_\infty$ algebra structure on 
$\tot\big(\sections{\Lambda^\bullet A\dual}\otimes_R\Dpoly{\bullet}\big)$.
Moreover, since the retraction $\etendu{\sigma}$ intertwines the associative algebra structures, 
we immediately obtain the following corollary of Proposition~\ref{cotedor}.

\begin{corollary}[\cite{arXiv:1901.04602}]
\label{thm:Naples}
Given a Lie pair $(L,A)$, each choice of a splitting $j:B\to L$ of the short exact sequence of vector bundles $0\to A \to L \to B \to 0$
and of a torsion-free $L$-connection $\nabla$ on $B$ 
determines 	
\begin{enumerate}
\item an $L_\infty$ algebra structure on $\tot\big(\sections{\Lambda^\bullet A\dual}\otimes_R\Dpoly{\bullet}\big)$ 
with the operator $\dau+\dHH$ as unary bracket 
\item and a Gerstenhaber algebra structure on $\hypercohomology^\bullet_{\CE}(A,\Dpoly{\bullet})$, 
the cohomology of the complex \[ \Big(\tot\big(\sections{\Lambda^\bullet A\dual}\otimes_R\Dpoly{\bullet}\big),\dau+\dHH\Big) .\]
\end{enumerate}
\end{corollary}

A priori, the $L_\infty$ algebra structure on $\tot\big(\sections{\Lambda^\bullet A\dual}\otimes_R\Dpoly{\bullet}\big)$ 
in Corollary~\ref{thm:Naples} is not canonical; it depends on a choice of 
`Dolgushev--Fedosov type' replacement for the complex 
$\big(\tot\big(\sections{\Lambda^\bullet A\dual}\otimes_R\Dpoly{\bullet}\big),\dau+\dHH\big)$ 
via a Fedosov dg Lie algebroid $\cF\to\cM$.
The construction of the Fedosov differential involves the choice of a torsion-free connection 
$\nabla:\sections{L}\times\sections{B}\to\sections{B}$ 
and a splitting $j:B\to L$ of the short exact sequence of vector bundles $0\to A \to L \to B \to 0$.
However, different choices yield isomorphic $L_\infty$ algebra structures
on $\tot\big(\sections{\Lambda^\bullet A\dual}\otimes_R\Dpoly{\bullet}\big)$.
Hence, we obtain the following improvement on Corollary~\ref{thm:Naples}:

\begin{theorem}[\cite{arXiv:1901.04602}]
\label{thm:Naples2}
Let $(L,A)$ be a Lie pair. 
\begin{enumerate}
\item The space $\tot\big(\sections{\Lambda^\bullet A\dual}\otimes_R \Dpoly{\bullet}\big)$ admits an $L_\infty$ algebra structure 
with the operator $\dau+\dHH$ as unary bracket. This $L_\infty$ algebra structure is unique up an $L_\infty$ isomorphism 
having the identity map as linear part.
\item The corresponding cohomology group $\hypercohomology_{\CE}(A,\Dpoly{\bullet})$ admits a canonical Gerstenhaber algebra structure.
\end{enumerate}
\end{theorem}

Moreover, when the Lie pair happens to be a matched pair, the transferred $L_\infty$ algebra structure on
$\tot\big(\sections{\Lambda^\bullet A\dual}\otimes_R\Dpoly{\bullet}\big)$ is precisely the dgla structure described in Proposition~\ref{pro:zurich}.
\begin{proposition}
Under the hypotheses of Corollary~\ref{thm:Naples} and the additional assumption that $j(B)$ is a Lie subalgebroid of $L$ --- 
i.e.\ $L=A\bowtie B$ is a matched pair --- the $L_\infty$ algebra $\tot\big(\sections{\Lambda^\bullet A\dual}\otimes_R\Dpoly{\bullet}\big)$ 
and the Gerstenhaber algebra $\hypercohomology^\bullet_{\CE}(A,\Dpoly{\bullet})$ of Corollary~\ref{thm:Naples} coincide respectively 
with the dgla $\tot\big(\sections{\wedge^\bullet A\dual} \otimes_R \enveloping{B}^{\bullet+1}\big)$
and the Gerstenhaber algebra $\hypercohomology^\bullet_{\CE}(A,\enveloping{B}^{\bullet+1})$ of Proposition~\ref{pro:zurich}.
\end{proposition}

\bibliography{references}

\end{document}